\documentclass[a4paper,11pt]{article}

\usepackage[utf8]{inputenc}	
\usepackage[english]{babel} 

\usepackage{tikz}
\usepackage{mathtools}
\usepackage{amsmath}
\usepackage{dsfont}
\usepackage{amsthm}
\usepackage{amssymb}
\usepackage{xcolor}
\usepackage{multirow}
\usepackage{bm}
\usepackage{lineno}
\usepackage[pdftex,pdfborder={0 0 0}, 
colorlinks=true, 
linkcolor=blue, 
citecolor=red, 
pagebackref=true, 
]{hyperref}

\usepackage{enumitem}

\usepackage{hyperref}

\usepackage{graphicx}
\usepackage{subcaption}

\usepackage{algorithm,algorithmic}
\usepackage{lineno}

\usepackage{vmargin}
\setmarginsrb{3.5cm}{3cm}{3.5cm}{2cm}{0cm}{0cm}{0cm}{1.5cm}
\allowdisplaybreaks[3]

\theoremstyle{plain}
\newtheorem{theorem}{Theorem}[section]
\newtheorem{proposition}[theorem]{Proposition}
\newtheorem{lemma}[theorem]{Lemma}
\newtheorem{corollary}[theorem]{Corollary}
\newtheorem{remark}[theorem]{Remark}

\newtheorem{assumption}[theorem]{Assumption}

\newcommand{\R}{\mathbb{R}}

\DeclareMathOperator*{\argmin}{arg\,min}
\DeclareMathOperator{\interior}{int}
\DeclareMathOperator{\dom}{dom}

\DeclareMathOperator{\bfdom}{{\normalfont \textbf{dom}}}
\DeclareMathOperator{\val}{val}
\DeclareMathOperator{\prox}{prox}
\DeclareMathOperator{\proj}{proj}

\begin{document}

\title{Discrete potential mean field games: duality and numerical resolution\footnote{This work was supported by a public grant as part of the
Investissement d'avenir project, reference ANR-11-LABX-0056-LMH,
LabEx LMH, and by the FIME Lab (Laboratoire de Finance des March\'{e}s de l'Energie), Paris.}}

\author{J.~Fr{\'e}d{\'e}ric Bonnans\footnote{
    Universit\'{e} Paris-Saclay, CNRS, CentraleSup\'{e}lec, Inria, Laboratoire des signaux et syst\`{e}mes, 91190, Gif- sur-Yvette, France. E-mails: \href{mailto:frederic.bonnans@inria.fr}{frederic.bonnans@inria.fr}, \href{mailto:laurent.pfeiffer@inria.fr}{laurent.pfeiffer@inria.fr}.}
\and
Pierre Lavigne \footnote{
CMAP UMR 7641, Ecole polytechnique, route de Saclay, 91128, Palaiseau Cedex, Institut Polytechnique de Paris, France. E-mail: \href{mailto:pierre.lavigne@polytechnique.edu}{pierre.lavigne@polytechnique.edu}.}
\and
Laurent Pfeiffer\textsuperscript{\textdagger}}

\maketitle

\begin{abstract}
We propose and investigate a general class of discrete time and finite state space mean field game (MFG) problems with potential structure.
Our model incorporates interactions through
a congestion term and a price variable.
It also allows hard constraints on the distribution of the agents.
We analyze the connection between the MFG problem and two optimal control problems in duality.
We present two families of numerical methods and detail their implementation: (i) primal-dual proximal methods (and their extension with nonlinear proximity operators), (ii) the alternating direction method of multipliers (ADMM) and a variant called ADM-G. We give some convergence results. Numerical results are provided for two examples with hard constraints.
\end{abstract}



\section{Introduction}

The class of mean field game (MFG) problems was introduced by J.-M.~Lasry and P.-L.~Lions in \cite{LL06cr1,LL06cr2,LL07mf} and M.~Huang, R.~Malham\'e, and P.~Caines in \cite{HCMieeeAC06} to study interactions among a large population of agents.
The agents of the game optimize their own dynamical system with respect to a criterion; the criterion is parameterized by some endogenous coupling terms. These coupling terms are linked to the collective behavior of all agents and thus induce an interaction between them. It is assumed that an isolated agent has no impact on the coupling terms and that all agents are identical.
At a mathematical level, MFGs typically take the form of a system of coupled equations: a dynamic programming equation (characterizing the optimal behavior of the agents), a Kolmogorov equation (describing the distribution of the agents), and coupling equations.

In this work we study a class of discrete time and finite state space mean field games with potential structure.
The dynamical system of each agent is a Markov chain, with controlled probability transitions.
Few publications deal with fully discrete models; in a seminal work, D.~Gomes, J.~Mohr, and R.~R.~Souza \cite{GomeMohrSouza} have studied the existence of a Nash equilibrium via a fixed point approach and investigated the long-term behavior of the game.
The proof relies on a monotonicity assumption for the congestion term. In \cite{hadikhanloo2019finite}, in a similar setting, the convergence of the fictitious play algorithm is established. 
In addition \cite{hadikhanloo2019finite} proves 
the convergence of a discrete mean field game problem (with an entropic regularization of the running cost) toward a continuous first-order mean field game.

Potential (also called variational) MFGs are coupled systems which can be interpreted as first-order conditions of two control problems in duality whose state equations are respectively a Kolmogorov equation and a dynamic programming equation. The primal problem (involving the Kolmogorov equation) can be interpreted as a stochastic optimal control problem with cost and constraints on the law of the state and the control. Its numerical resolution is thus of interest beyond the context of MFGs.

\paragraph{\textbf{Framework.}}
In our model, the agents interact with each other via two coupling terms: a congestion variable $\gamma$ and a price variable $P$. The congestion $\gamma$ is linked to the distribution of the agents via the subdifferential of a proper convex and l.s.c.\@ potential $F$. The price $P$ is linked to the joint law of states and controls of the agents via the subdifferential of a proper convex and l.s.c.\@ potential $\phi$.
A specificity of our discrete model is that the potentials $F$ and $\phi$ can take the value $+\infty$ and thus induce constraints on the distribution of the agents, referred to as hard constraints.
In the continuous case, four classes of variational MFGs can be identified. Our model is general enough to be seen as the discrete counterpart of these four cases.
Case 1: MFGs with monotone congestion terms ($F$ is differentiable, $\phi=0$). The first variational formulation was given in \cite{LL06cr2} and has been widely studied in following works \cite{benamou2017variational,cardaliaguet2015second,cardaliaguet2016first,meszaros2015variational,prosinski2017global}. 
Case 2: MFGs with density constraints ($F$ has a bounded domain, $\phi=0$). These models are of particular interest for describing crowd motions.
The coupling variable $\gamma$ has there an incentive role. The reader can refer to  \cite{cardaliaguet2016first,meszaros2015variational,santambrogio2011modest,santambrogio2018crowd}.
Case 3: MFGs with Cournot interactions ($F=0$, $\phi$ is differentiable). In this situation, each agent optimally chooses a quantity to be sold at each time step of the game. Interactions with the other players occur through the gradient of $\phi$ which maps the mean strategy (the market demand) to a market price.
See for example \cite{BHP-schauder,graber2015existence,graber2020nonlocal,graber2018variational,graber2020weak}.
Case 4: MFGs with price formation ($F=0$, $\phi$ has a bounded domain). These models incorporate a hard constraint on the demand. The price variable is the associated Lagrange multiplier and has a incentive role. We refer to \cite{saude2020mean}.

The first part of the article is devoted to the theoretical analysis of the MFG system.
We first introduce a potential problem,
shown to be equivalent to a convex problem involving the Kolmogorov equation via a change of variable, similar to the one widely employed in the continuous setting (e.g.\@ in \cite{benamou2000computational}).
Under a suitable qualification condition, we establish a duality result between this problem and an optimal control problem involving the dynamic programming equation. We show the existence of solutions to these problems and finally we show the existence of a solution to the MFG system. A uniqueness result is proved (when $F$ and $\phi$ are differentiable).

The second part of the article is devoted to the numerical resolution of the MFG system. We focus on two families of methods: primal-dual methods and augmented Lagrangian methods.
These two classes exploit the duality structure discussed above and can deal with hard constraints.
They have already been applied to continuous MFGs, see for example the survey article \cite{achdou2020mean}.
Primal-dual methods have been applied to stationary MFGs with hard congestion terms in \cite{briceno2018proximal} and to time-dependent MFGs in \cite{briceno2019implementation}. In a closely related setting, \cite{erbar2020computation} applies a primal-dual method to solve a discretized optimal transport problem.
Augmented Lagrangian methods have been applied to MFGs
in \cite{Benamou2015} and to MFGs with hard congestion terms in \cite{benamou2017variational}.
Other methods exploiting the potential structure have been investigated in the literature, they are out of the scope of the current article. Let us mention the Sinkhorn algorithm \cite{benamou2019entropy}. The fictitious play method has been investigated in various settings:
\cite{geist2021concave} shows the connection between the fictitious play method and the Frank-Wolfe algorithm in a discrete and potential setting;
\cite{CDHD2015} considers a continuous setting, with a non-convex potential.


Let us emphasize that the above references all deal with interaction terms depending on the distribution of the states of the agents; very few publications are concerned by interactions through the controls (see \cite{kobeissi2020meanfinite}). The present work is the first to address methods for ``Cournot" mean field games.

\paragraph{\textbf{Contributions.}}
Let us comment further on the families of methods under investigation and our contributions.
The primal-dual algorithms that we have implemented were introduced by A.~Chambolle and T.~Pock \cite{chambolle2011first} and applied to mean field games in \cite{briceno2018proximal}. 
A novelty of our work is also to show that the extension of primal-dual methods of \cite{chambolle2016ergodic}, involving nonlinear proximity operators (based on Bregman divergences), can also be used to solve MFGs.
The augmented Lagrangian method that we have implemented is applied to the dual problem (involving the dynamic programming equation), as originally proposed in \cite{benamou2000computational} for optimal transportation problems.
As in \cite{benamou2000computational}, we have actually implemented a variant of the augmented Lagrangian method, called alternating direction method of multipliers (ADMM). The method was introduced by R.~Glowinski and A.~Marroco \cite{glowinski1975approximation} and studied by  D.~Gabay and B.~Mercier \cite{gabay1976dual}. It relies on a successive minimization of the augmented Lagrangian function.
One of the main limitations of ADMM is that when the number of involved variables is greater or equal to three, as it is the case for our problem, convergence is not granted.
A novelty of our work is to consider a variant of ADMM, the alternating direction method with Gaussian back substitution (ADM-G), introduced in \cite{he2012alternating}. At each iteration of this method, the ADMM step is followed by a Gaussian back substitution step. Convergence is ensured. The practical implementation of the additional step turns out to be inexpensive in our framework.

The last contribution of this work is to propose and solve numerically two hard constraints problems: a congestion mean field game problem and a ``Cournot" mean field game. Following our analysis we define a notion of residuals allowing us to compare the empirical convergence of each method in a common setting.

\paragraph{\textbf{Organization of the article.}}
The article is organized as follows. In section \ref{section:discrete-mfg} we provide the main notations, the mean field game system under study and the underlying individual player problem.
In section \ref{section:potential-problem} we formulate a potential problem and perform the announced change of variable. In section \ref{sec:Duality} we form a dual problem and we establish a duality result.  In section \ref{sec:connexion} we provide our main results: existence and uniqueness of a solution to the mean field game.
In section \ref{sec:numerics} we provide a detailed implementation of the primal-dual proximal algorithms, ADMM and ADM-G, and we give theoretical convergence results when possible.
In section \ref{sec:numerical} we present numerical results for two concrete problems. We provide outputs obtained for each method: errors, value function, equilibrium measure, mean displacement, congestion, demand and price.

\section{Discrete mean field games \label{section:discrete-mfg}}

\subsection{Notation}

\paragraph{Sets.}
Let $T \in \mathbb{N}^\star$ denote the duration of the game. We set $\mathcal{T}= \{ 0,..., T-1 \}$ and $\bar{\mathcal{T}} = \{ 0,...,T \}$. Let $S =  \{ 0,..., n-1 \}$ denote the state space. We set
\begin{align*}
\Delta(S) = \ & \Big\{ \pi \colon S \rightarrow [0,1] \, \big|\, \sum_{x \in S} \pi(x) = 1 \Big\}, \\
\Delta = \ & \Big\{ \pi \colon \mathcal{T} \times S \times S \rightarrow [0,1] \, \big|\, \pi(t,x,\cdot) \in \Delta(S), \ \forall (t,x) \in \mathcal{T} \times S \Big\}.
\end{align*}
For any finite set $A$, we denote by $\mathbb{R}(A)$ the finite-dimensional vector space of mappings from $A$ to $\mathbb{R}$. For any finite set $B$ and linear operator $L \colon \mathbb{R}(A) \to \mathbb{R}(B)$, we denote $L^\star \colon  \mathbb{R}(B) \to \mathbb{R}(A)$ the adjoint operator satisfying the relation
\begin{equation} \nonumber
    \sum_{x \in A} L[u](x) v(x) = \sum_{y \in B} u(y) L^\star[v](y).
\end{equation}
All along the article, we make use of the following spaces:
\begin{equation*}
\begin{array}{rlrl}
\mathcal{R} = &    \R(\bar{\mathcal{T}} \times S) \times \R(\mathcal{T} \times S^2), \qquad
& \mathcal{U} = &    \R(\bar{\mathcal{T}} \times S) \times \R(\mathcal{T}), \\
\mathcal{C}= &    \mathcal{R} \times \R(\bar{\mathcal{T}} \times S) \times \R(\mathcal{T}),
& \mathcal{K} = &    \R(\bar{\mathcal{T}} \times S) \times \mathcal{U} .
\end{array}
\end{equation*}

\paragraph{Convex analysis.}
For any function $g \colon \R^d \to \mathbb{R} \cup \{+ \infty \}$, we denote
\begin{equation} \nonumber
\dom(g) = \left\{x \in X \,\big|\, g(x) < + \infty \right\}.
\end{equation}
The subdifferential of $g$ is defined by
\begin{equation} \nonumber
\partial g(x) = \left\{ x^\star \in \mathbb{R}^d \,\big|\, g(x') \geq g(x) + \langle x^\star, x' - x \rangle, \ \forall x' \in \mathbb{R}^d \right\}.
\end{equation}
By convention, $\partial g(x)= \emptyset$ if $g(x)= + \infty$.
Note also that $x^\star \in \partial g(x)$ if and only if $g(x) + g^\star(x^\star) = \langle x, x^\star \rangle$, where $g^\star$ is the Fenchel transform of $g$, defined by
\begin{equation} \nonumber
g^\star(x^\star) = \sup_{x \in \mathbb{R}^d} \langle x, x^\star \rangle - g(x).
\end{equation}
Note that the subdifferential and Fenchel transforms of $\ell$, $F$, and $\phi$ (introduced in the next paragraph) are considered for fixed values of the time and space variables.

We denote by $\chi$ the indicator function of $\{ 0 \}$ (without specifying the underlying vector space). For any subset $C \subseteq \mathbb{R}^d$, we denote by $\chi_C$ the indicator function of $C$. For any $x \in C$, we denote by $N_{C}(x)$ the normal cone to $C$ at $x$,
\begin{equation} \nonumber
N_{C}(x) = \left\{ x^\star \in \mathbb{R}^{d} \,\big|\, \langle x^\star, x'-x \rangle \leq 0, \ \forall x' \in C \right\}.
\end{equation}
We set $N_{C}(x)= \emptyset$ if $x \notin C$.

\paragraph{Nemytskii operators.}
Given two mappings $g \colon \mathcal{X} \times \mathcal{Y} \rightarrow \mathcal{Z}$ and $u \colon \mathcal{X} \rightarrow \mathcal{Y}$, we call Nemytskii operator the mapping $\bm{g}[u] \colon \mathcal{X} \rightarrow \mathcal{Z}$ defined by
\begin{equation*}
\bm{g}[u](x) = g(x,u(x)).
\end{equation*}
We will mainly use this notation in order to avoid the repetition of time and space variables, for example, we will write $\bm{\ell}[\pi](t,x)$ instead of $\ell(t,x,\pi(t,x))$.

All along the article, we will transpose some notions associated with $g$ to the Nemytskii operator $\bm{g}[u]$. When $\mathcal{Y}= \R^d$ and $\mathcal{Z}= \R \cup \{ + \infty \}$, we define the domain of $g$ by
\begin{equation} \nonumber
\bfdom(\bm{g}) = \big\{ u \colon \mathcal{X} \to \mathbb{R}^{d} \,\big|\, u(x) \in \dom(g(x,\cdot)), \ \forall x \in \mathcal{X} \big\}.
\end{equation}
We define $\bm{g}^\star[v] \colon \mathbb{R}^{d}  \to \mathbb{R}\cup \{+\infty\}$ by
$\bm{g}^\star[v](x) = g^\star(x,v(x)),
$
where $g^\star$ is the Fenchel transform of $g$ with respect to the second variable.

\subsection{Coupled system}

\paragraph{Data and assumption.}
We fix an initial distribution $m_0 \in \Delta(S)$ and four maps: a running cost $\ell$, a potential price function $\phi$, a potential congestion cost $F$, and a displacement cost $\alpha$,
\begin{equation} \nonumber
\begin{array}{rlrl}
\ell \colon &    \mathcal{T} \times S \times \mathbb{R}(S) \rightarrow \R\cup \{+\infty\}, \quad & \phi \colon &   \mathcal{T} \times \R \rightarrow \R \cup \{+\infty\}, \\
F \colon &    \bar{\mathcal{T}} \times \mathbb{R}(S) \rightarrow \R \cup \{+\infty\}, & \alpha \colon &   \mathcal{T} \times S^2 \rightarrow \R.
\end{array}
\end{equation}
The following convexity assumption is in force all along the article. Note that we will later make use of an additional qualification assumption (Assumption \ref{A:m-pi-kolmogorov}).

\begin{assumption}[Convexity] \label{A:l-convex}
For any $(t,s,x) \in \mathcal{T} \times \bar{\mathcal{T}} \times S$, the maps $\ell(t,x,\cdot)$, $F(s,\cdot)$, and $\phi(t,\cdot)$ are proper, convex and lower semicontinuous. In addition $\dom (\ell(t,x,\cdot)) \subseteq \Delta(S)$.
\end{assumption}

\paragraph{Coupled system.}
The unknowns of the MFG system introduced below are denoted $((m,\pi),(u,\gamma,P)) \in \mathcal{R} \times \mathcal{K}$. They can be described as follows:
\begin{itemize}
\item[\tiny$\bullet$]  $\gamma$ and $P$ are the coupling terms of the MFG: $\gamma(t,x)$ is a congestion term incurred by agents located at $x \in S$ at time $t \in \bar{\mathcal{T}}$ and $P(t)$ is a price variable
\item[\tiny$\bullet$]  $\pi(t,x,y)$ denotes the probability transition from $x \in S$ to $y \in S$, for agents located at $x$ at time $t$
\item[\tiny$\bullet$]  $m(t,x)$ denotes the proportion of agents located at $x \in S$ at time $t \in \bar{\mathcal{T}}$
\item[\tiny$\bullet$]  $u(t,x)$ is the value function of the agents.
\end{itemize}
For any $(\gamma,P) \in \mathcal{U}$, we define the individual cost $c \colon \mathcal{T} \times S \times S \times \Delta(S)  \to \mathbb{R}$,
\begin{equation} \nonumber
c_{\gamma,P}(t,x,y,\rho) = \ell(t,x,\rho) + \gamma(t,x) + \alpha(t,x,y) P(t).
\end{equation}
Given $(m,\pi) \in \mathcal{R}$, we denote
\begin{equation*}
\bm{Q}[m,\pi](t)  = \sum_{(x,y) \in S^2} m(t,x) \pi(t,x,y) \alpha(t,x,y).
\end{equation*}
We aim at finding a quintuplet $(m,\pi,u,\gamma,P)$ such that for any $(t,s,x) \in \mathcal{T} \times \bar{\mathcal{T}} \times S$,
\begin{equation} \label{eq:MFG} \tag{MFG}
\begin{cases}
\begin{array}{cl}
\text{(i)} &   
\begin{cases}
\begin{array}{rl}
u(t,x) = &   {\displaystyle \inf_{\rho \in \Delta(S)} \
\sum_{y \in S} \rho(y) \Big(
c_{\gamma,P}(t,x,y,\rho) + u(t+1,y) 
\Big),
}\\
u(T,x) = &   \gamma(T,x),
\end{array}
\end{cases}
\\[2.5em]
\text{(ii)} & \ \pi(t,x,\cdot) \in   {\displaystyle \underset{\rho \in \Delta(S)}{\text{arg min}}\
\sum_{y \in S} \rho(y) \Big(
c_{\gamma,P}(t,x,y,\rho) + u(t+1,y) 
\Big),
} \\[2em]
\text{(iii)} &
\begin{cases}
\begin{array}{rl}
m(t+1,x)= & {\displaystyle \sum_{y \in S} m(t,y) \pi(t,y,x), } \\
m(0,x)= & m_0(x),
\end{array}
\end{cases}
\\[2.5em]
\text{(iv)} & \  {\displaystyle \gamma(s,\cdot) \in \partial F (s,m(s,\cdot)) },
\\[1em]
\text{(v)} &\ {\displaystyle P(t) \in \partial \phi \big(t, \bm{Q}[m,\pi](t) \big). }
\end{array}
\end{cases}
\end{equation}

\paragraph{Heuristic interpretation.}
\begin{itemize}
\item[\tiny$\bullet$] The dynamical system of each agent is a Markov chain $(X_s^\pi)_{s\in \bar{\mathcal{T}}}$ controlled by $\pi \in \Delta$, with initial distribution $m_0$: for any $ (t,x,y) \in \mathcal{T} \times S^2$,
\begin{equation} \label{eq:markov_chain}
\mathbb{P}\left(X_{t+1}^\pi = y \vert X_{t}^\pi = x \right) = \pi(t,x,y), \quad \mathbb{P}( X_{0}^\pi = x) = m_0(x).
\end{equation}
Given the coupling terms $(\gamma, P) \in \mathcal{U}$, the individual control problem is
\begin{equation} \label{eq:pb_rep_agent}
\inf_{\pi \in \Delta}  J_{\gamma,P}(\pi) := 
 \mathbb{E} \Big(  \sum_{t \in \mathcal{T}} c_{\gamma,P}(t,X_{t}^\pi,X_{t+1}^\pi,\pi(t,X_t^\pi)) +  \gamma(T,X_T^\pi) \Big).
\end{equation}
The equations (\ref{eq:MFG},i-ii) are the associated dynamic programming equations:
given $(\gamma,P) \in \mathcal{U}$, if $u$ and $\pi$ satisfy these equations, then $\pi$ is a solution to \eqref{eq:pb_rep_agent}. The reader can refer to \cite[Chapter 7]{frederic} for a detailed presentation of the dynamic programming approach for the optimal control of Markov chains.
\item[\tiny$\bullet$] Given $\pi \in \Delta$, denote by $m^\pi$ the probability distribution of $X^\pi$, that is, $m^\pi(t,x)=\mathbb{P}(X_t^\pi= x)$. Then $m^\pi$ is obtained by solving the Kolmogorov equation (\ref{eq:MFG},iii). In the limit when the number of agents tends to $\infty$, the distribution $m^\pi$ coincides with the empirical distribution of the agents.
\item[\tiny$\bullet$] Finally, the equations (\ref{eq:MFG},iv-v) link the coupling terms $\gamma$ and $P$ to the distribution of the agents $m$ and their control $\pi$.
\end{itemize}
In summary: given a solution $((m,\pi),(u,\gamma,P)) \in \mathcal{R} \times \mathcal{K}$ to \eqref{eq:MFG}, the triplet $(\pi,\gamma,P)$ is a solution to the mean field game
\begin{equation} \nonumber
\pi \in \argmin_{\rho \in \Delta}  J_{\gamma,P}(\rho), \quad \gamma \in \partial \bm{F}[m^\pi], \quad P \in \partial \bm{\phi}[\bm{Q}[m^\pi,\pi]].
\end{equation}

\paragraph{Potential problem.}
The next section of the article will be dedicated to the connection between the coupled system and the potential problem \eqref{eq:pot_pb}, introduced page \pageref{eq:pot_pb}. We provide here a stochastic formulation of \eqref{eq:pot_pb} as an optimal control problem of a Markov chain, which has its own interest:
\begin{align} \nonumber
\inf_{\pi \in \Delta} \ \ &
 \sum_{t \in \mathcal{T}}  \mathbb{E} \big[ \ell(t,X_{t}^\pi,\pi(t,X_{t}^\pi)) \big] + \sum_{t \in \bar{\mathcal{T}}} F(t, \mathcal{L}(X_{t}^\pi)) \\
 & \quad +  \sum_{t \in \mathcal{T}} \phi\left(t,  \mathbb{E} \left[ \alpha(t,X_{t}^\pi,X_{t+1}^\pi) \pi(t,X_{t}^\pi,X_{t+1}^\pi) \right] \right), \nonumber
\end{align}
where $(X^{\pi}_t)_{t\in \bar{\mathcal{T}}}$ is a controlled Markov satisfying \eqref{eq:markov_chain}.

\begin{remark} \label{remark:model}
\begin{itemize}
\item[\tiny$\bullet$] At any time $t \in \mathcal{T}$, it is possible to encode constraints on the transitions of the agents located at $x \in S$ by defining $\ell$ in such a way that $\dom (\ell(t,x,\cdot))$ is strictly included into $\Delta(S)$. An example will be considered in Section \ref{sec:numerical}.
\item[\tiny$\bullet$] If $F$ and $\phi$ are differentiable, then their subdifferentials are singletons and thus the coupling terms $\gamma$ and $P$ are uniquely determined by $m$ and $\pi$ through the equations (\ref{eq:MFG},iv-v).
\item[\tiny$\bullet$] The equations (\ref{eq:MFG},iv-v) imply that $m \in \bfdom(\bm{F})$ and $\bm{Q}[m,\pi] \in \bfdom(\bm{\phi})$. Thus they encode hard constraints on $m$ and $\pi$ if the coupling functions $F$ or $\phi$ take the value $+\infty$. For example, they can be chosen in the form $G \colon \mathbb{R}^d \to \mathbb{R} \cup \{ + \infty\}$,
$G = g + \chi_K$,
where $g \colon \mathbb{R}^d \to \mathbb{R}$ is convex and differentiable and where $K$ is a closed and convex subset of $\mathbb{R}^d$. Then by \cite[Corollary 16.38]{bauschke2011convex},
\begin{equation} \nonumber
\partial G(x) = \nabla g(x) + N_K(x), \quad \forall x \in \R^d.
\end{equation}
\end{itemize}
\end{remark}

\subsection{Further notation}

We introduce now two linear operators, $\bm{A}$ and $\bm{S}$. They will allow to bring out the connection between the coupled system and the potential problem.
The operator
$\bm{A} \colon \R(\mathcal{T} \times S^2) \rightarrow \R(\mathcal{T})$
and its adjoint
$\bm{A}^\star \colon  \R(\mathcal{T}) \rightarrow \R(\mathcal{T} \times S^2)$
are given by
\begin{equation*}
\bm{A}[w](t)  = \sum_{(x,y) \in S^2} w(t,x,y) \alpha(t,x,y), \quad
 \bm{A}^\star[P](t,x,y) = \alpha(t,x,y) P(t).
\end{equation*}
The operator $\bm{S} \colon \R(\mathcal{T} \times S^2) \rightarrow \R(\bar{\mathcal{T}} \times S)$ and its adjoint $\bm{S}^\star \colon  \R(\bar{\mathcal{T}} \times S) \rightarrow \R(\mathcal{T} \times S^2)$ are given by
\begin{align*}
\bm{S}[w](s,x)  = {} &
\begin{cases}
\begin{array}{ll}
{
\sum_{y \in S} w(s-1,y,x) } \ & \text{if $s>0$}, \\
0 & \text{if $s=0$},
\end{array}
\end{cases}
\\[0.5em]
\bm{S}^\star[u](t,x,y)= {} & u(t+1,y). \nonumber
\end{align*}
We can now reformulate the dynamic programming equations of the coupled system (\ref{eq:MFG},i-ii) as follows:
\begin{equation*}
\begin{cases}
\begin{array}{cl}
\text{(i)} &
\begin{cases}
 u(t,x)+ \bm{\ell}^\star[-\bm{A}^\star P - \bm{S}^\star u](t,x) =  \gamma(t,x),\\
 u(T,x) =  \gamma(T,x),
\end{cases}
\\[1.5em]
\text{(ii)} & (\bm{\ell}[\pi] +  \bm{\ell}^\star[-\bm{A}^\star P - \bm{S}^\star u])(t,x)  =- \langle \pi(t,x),(\bm{A}^\star P + \bm{S}^\star u)(t,x) \rangle.
\end{array}
\end{cases}
\end{equation*}
 
\section{Potential problem and convex formulation \label{section:potential-problem}}

\subsection{Perspective functions}

Given $h \colon \mathbb{R}^d \to \mathbb{R}\cup\{+\infty\}$ a proper l.s.c.\@ and convex function with bounded domain,
we define the perspective function $\tilde{h} \colon \mathbb{R} \times  \mathbb{R}^d  \to \mathbb{R}\cup\{+\infty\}$ (following  \cite[Section 3]{rockafellar1966level}) by
\begin{equation} \nonumber
\tilde{h}(\theta,x) = \begin{cases}
\theta h(x/ \theta), & {\normalfont  \text{if }} \theta>0,\\
0,  & {\normalfont  \text{if }} (\theta,x)=(0,0),\\
+ \infty, & {\normalfont  \text{otherwise}}.
 \end{cases}
\end{equation}

\begin{lemma} \label{lemma:persp}
The perspective function $\tilde{h}$ is proper,  convex, l.s.c.\@ and its domain is given by
$\dom(\tilde{h}) = \big\{ (\theta,x) \in \mathbb{R}_+ \times \mathbb{R}^d \, \big| \, x \in \theta \dom(h) \big\}.$
For any $(\theta^\star, x^\star) \in \mathbb{R} \times  \mathbb{R}^d$, we have
\begin{equation}  \label{eq:h-star}
\tilde{h}^\star(\theta^\star, x^\star) = \chi_Q(\theta^\star, x^\star),
\end{equation}
where $Q := \{ (\theta^\star, x^\star) \in \mathbb{R} \times  \mathbb{R}^d, \, h^\star(x^\star) + \theta^\star \leq 0 \}$.
\end{lemma}

\begin{proof}
The proof is a direct application of \cite[Lemmas 1.157, 1.158]{bonnans2013perturbation} when $h$ has a bounded domain. In this case the recession function of $h$ is the indicator function of zero.
 \end{proof}

\begin{lemma} \label{lemma:hhstar-fenchel}
Let $(\theta, x),(\theta^\star, x^\star) \in \mathbb{R} \times  \mathbb{R}^d$.
Then $(\theta^\star,x^\star) \in \partial \tilde{h}(\theta,x)$
if and only if
\begin{equation*}
\begin{array}{ll}
\text{either: } & h^\star(x^\star) + \theta^\star \leq 0 \quad \ \text{and} \quad (\theta,x) = (0,0), \\
\text{or:} & h^\star(x^\star) + \theta^\star = 0, \quad h(x/\theta) + h^\star(x^\star) - \langle x/\theta, x^\star \rangle = 0, \quad \text{and} \quad  \theta > 0.
\end{array}
\end{equation*}
%
\end{lemma}
\begin{proof}
Direct application of \cite[Proposition 2.3]{combettes2018perspective}.
 \end{proof}

\subsection{Potential problem}

We define the following criterion
\begin{equation*}
\mathcal{J}(m,\pi) =  \sum_{(t,x) \in \mathcal{T} \times S} m(t,x)\bm{\ell}[\pi](t,x)  + \sum_{t \in \mathcal{T}} \bm{\phi}[\bm{Q}[m,\pi]](t) + \sum_{s \in \bar{\mathcal{T}}}\bm{F}[m](s)
\end{equation*}
and the following potential problem (recall that $m^\pi$ is the solution to the Kolmogorov equation (\ref{eq:MFG},iii), given $\pi \in \Delta$):
\begin{equation} \label{eq:pot_pb} \tag{P}
  \displaystyle \inf_{(m,\pi) \in\mathcal{R}} \mathcal{J}(m,\pi), \quad
\text{ subject to: } \, m= m^\pi.
\end{equation}
The link between the mean field game system \eqref{eq:MFG} and the potential problem \eqref{eq:pot_pb} will be exhibited in Section \ref{sec:connexion}.
Notice that Problem \eqref{eq:pot_pb} is not convex. Yet we can define a closely related convex problem, whose link with  \eqref{eq:pot_pb} is established in Lemma \ref{lemma:L-L}.

We denote by $\tilde{\ell} \colon \mathcal{T} \times S \times \mathbb{R} \times \mathbb{R}(S) \to \mathbb{R}\cup \{+\infty\}$ the perspective function of $\ell$ with respect to the third variable. By Lemma \ref{lemma:persp} the function $\tilde{\ell}(t,x,\cdot,\cdot)$
is proper convex and l.s.c.\@ for any $(t,x) \in \mathcal{T} \times S$.
We define
\begin{equation*}
\tilde{\mathcal{J}}(m,w) =
\sum_{(t,x) \in \mathcal{T} \times S} \tilde{\bm{\ell}}[m,w](t,x) + \sum_{t \in \mathcal{T}} \bm{\phi}[\bm{A}w](t) + \sum_{s \in \bar{\mathcal{T}}}\bm{F}[m](s).
\end{equation*}
In the above definition, $\tilde{\bm{\ell}}$ is the Nemytskii operator of $\tilde{\ell}$, that is, for any $(t,x) \in \mathcal{T} \times S$,
\begin{equation*}
\tilde{\bm{\ell}}[m,w](t,x)=
\begin{cases}
\begin{array}{ll}
m(t,x) \ell \Big( t,x,
\frac{w(t,x,\cdot)}{m(t,x)}\Big), & \text{if $m(t,x) >0$}, \\
0, & \text{if $m(t,x)=0$ and $w(t,x,\cdot)=0,$} \\
+ \infty, & \text{otherwise.}
\end{array}
\end{cases}
\end{equation*}
We consider now the following convex problem:
\begin{equation} \label{eq:pot_pb_conv} \tag{{\~P}}
 \displaystyle  \inf_{(m,w) \in \mathcal{R}} \tilde{\mathcal{J}}(m,w),\quad \text{ subject to: }  \bm{S}w - m +  \bar{m}_0= 0,
\end{equation}
where $\bar{m}_0 \in \R(\bar{\mathcal{T}} \times S)$ is defined by
\begin{equation} \nonumber
\bar{m}_0(s,x) = \begin{cases}
\begin{array}{ll}
m_0(x), & \text{if $s= 0$}, \\
0, & \text{otherwise}.
\end{array}
\end{cases}
\end{equation}

\begin{lemma} \label{lemma:L-L} 
Let $\val({\normalfont \text{P}})$ and $\val({\normalfont \tilde{\text{P}}})$ respectively denote the values of problems \eqref{eq:pot_pb} and \eqref{eq:pot_pb_conv}. Then $\val({\normalfont \text{P}}) = \val({\normalfont \tilde{\text{P}}})$.
In addition, if Problem \eqref{eq:pot_pb} is feasible, then both problems \eqref{eq:pot_pb} and \eqref{eq:pot_pb_conv} have a non-empty bounded set of solutions.
\end{lemma}

\begin{proof}
 \textit{Step 1: $\val({\normalfont \text{P}}) \geq \val({\normalfont \tilde{\text{P}}})$.} Let $(m,\pi) \in \dom(\mathcal{J})$ be such that $m=m^\pi$. Let
\begin{equation} \label{eq:change-var-mpi-w}
w(t,x,\cdot) := m(t,x)\pi(t,x,\cdot),
\end{equation}
for any $(t,x) \in \mathcal{T} \times S$. Then $(m,w)$ is feasible for problem \eqref{eq:pot_pb_conv} and
\begin{equation} \label{eq:ellmpi-ellmw}
m(t,x)\ell(t,x,\pi(t,x,\cdot)) = \tilde{\ell}(t,x,m(t,x),w(t,x,\cdot)),
\end{equation}
for any $(t,x)\in \mathcal{T}\times S$. Indeed by definition of $\tilde{\ell}(t,x,\cdot,\cdot)$, if $m(t,x) > 0$ then \eqref{eq:ellmpi-ellmw} holds and if $m(t,x) = 0$ then $w(t,x,\cdot) = 0$ and \eqref{eq:ellmpi-ellmw} still holds.
It follows that $\mathcal{J}(m,\pi) = \tilde{\mathcal{J}}(m,w)$ and consequently, $\val({\normalfont \text{P}}) \geq \val({\normalfont \tilde{\text{P}}})$.

\noindent \textit{Step 2: $\val({\normalfont \text{P}}) \leq \val({\normalfont \tilde{\text{P}}})$.} Let $(m,w) \in \dom(\tilde{\mathcal{J}})$ be such that $\bm{S}w-m = \bar{m}_0$ and let $\pi$ be such that
\begin{equation} \label{eq:change-var-mw-pi}
\begin{cases} \pi(t,x,\cdot) = w(t,x,\cdot)/m(t,x), & \text{if $m(t,x) >0$,}\\
\pi(t,x,\cdot) \in \dom(\ell(t,x,\cdot)), & \text{otherwise,}
\end{cases}
\end{equation} 
for all $(t,x)\in \mathcal{T}\times S$.
Then \eqref{eq:ellmpi-ellmw} is satisfied and $(m,\pi)$ is feasible for  \eqref{eq:pot_pb}. Thus $\mathcal{J}(m,\pi) = \tilde{\mathcal{J}}(m,w)$, and consequently, $\val({\normalfont \text{P}}) \leq \val({\normalfont \tilde{\text{P}}})$.

\noindent \textit{Step 3: non-empty and bounded sets of solutions.} Since $\mathcal{J}(m^\pi,\pi)$ is l.s.c.\@ with non-empty bounded domain, it reaches its minimum on its domain. Then the set of solutions to \eqref{eq:pot_pb} is non-empty and bounded. Now let $(m,\pi)$ be a solution to \eqref{eq:pot_pb} and let $w$ be given by \eqref{eq:change-var-mpi-w}. We have that
\begin{equation*}
 \tilde{\mathcal{J}}(m,w)  = \mathcal{J}(m,\pi) = \val({\normalfont \text{P}}) = \val({\normalfont \tilde{\text{P}}}),
\end{equation*}
thus we deduce that the set of solutions to \eqref{eq:pot_pb_conv} is non-empty. It remains to show that the set of solutions to \eqref{eq:pot_pb_conv} is bounded. Let $(m,w)$ be  a solution to \eqref{eq:pot_pb_conv}. The Kolmogorov equation implies that $0 \leq m(t,x) \leq 1$, for any $(t,x) \in \bar{\mathcal{T}} \times S$. By Lemma \ref{lemma:persp}, we have $w(t,x,\cdot) \in m(t,x) \Delta(S)$, which implies that $0 \leq w(t,x,y) \leq 1$.
 \end{proof}

Note that the above proof shows how to deduce a solution to \eqref{eq:pot_pb_conv} out of a solution to \eqref{eq:pot_pb} and vice-versa, thanks to relations \eqref{eq:change-var-mpi-w} and \eqref{eq:change-var-mw-pi}.


\section{Duality \label{sec:Duality}}

We show in this section that Problem \eqref{eq:pot_pb_conv} is the dual of an optimization problem, denoted \eqref{Pb:dual}, itself equivalent to an optimal control problem of the dynamic programming equation, problem \eqref{definition:new_dual}.
For this purpose, we introduce a new assumption (Assumption \ref{A:m-pi-kolmogorov}), which is assumed to be satisfied all along the rest of the article.

\subsection{Duality result}

The dual problem is given by
\begin{equation}  \label{Pb:dual} \tag{D} 
\begin{array}{c}
 \displaystyle \sup_{\begin{subarray}{c}
 (u,\gamma,P) \in \mathcal{K}
\end{subarray}}  \mathcal{D}(u,\gamma,P)
 :=  \langle m_0, u(0,\cdot)\rangle -\sum_{t\in \mathcal{T}} \bm{\phi}^\star[P](t) -  \sum_{s\in \bar{\mathcal{T}}} \bm{F}^\star[\gamma](s),
\\[1.5em]
 \text{subject to: }  
\begin{cases}
\begin{array}{ll}
 u(t,x)+ \bm{\ell}^\star[-\bm{A}^\star P - \bm{S}^\star u](t,x)  \leq  \gamma(t,x), & (t,x) \in \mathcal{T} \times S,\\
 u(T,x) =  \gamma(T,x), & x\in S.
 \end{array}
\end{cases}
\end{array}
\end{equation}
Note that the above kind of dynamic programming equation involves inequalities
(and not equalities as in (\ref{eq:MFG},i)).

We introduce now a qualification condition, which will allow to prove the main duality result of the section.
For any $\varepsilon = (\varepsilon_1,\varepsilon_2,\varepsilon_3) \in \mathcal{K}$ and $\pi \in \bfdom(\bm{\ell})$
we define $m_1[\varepsilon,\pi]$ the solution to the following perturbed Kolmogorov equation
\begin{equation}  \label{eq:kolmogorov_eps}
m_1(t+1,x) = \sum_{y\in S}m_1(t,y)\pi(t,y,x) - \varepsilon_1(t+1,x), \qquad m_1(0) - \varepsilon_1(0) = \bar{m}_0.
\end{equation}
We also define, for any $(t,x,y) \in \mathcal{T}\times S \times S$,
\begin{equation} \label{eq:wmD_eps}
\begin{array}{rl}
w[\varepsilon,\pi](t,x,y)  = & m_1[\varepsilon,\pi](t,x)\pi(t,x,y)\\
m_2[\varepsilon,\pi](t,x)  = & m_1[\varepsilon,\pi](t,x) + \varepsilon_2(t,x)\\
D[\varepsilon,\pi](t) = & \sum_{(x,y)\in S^2}w[\varepsilon,\pi](t,x,y)\alpha(t,x,y) + \varepsilon_3(t).
\end{array}
\end{equation}

\begin{assumption}[Qualification]  \label{A:m-pi-kolmogorov}
There exists $\alpha>0$ such that for any $\varepsilon = (\varepsilon_1,\varepsilon_2,\varepsilon_3)$ in $\mathcal{K}$ with $\|\varepsilon\| \leq \alpha$, there exists $\pi \in \bfdom(\bm{\ell})$ such that
\begin{equation} \label{hyp:qualif}
m_1[\varepsilon,\pi] \geq 0, \quad m_2[\varepsilon,\pi]  \in \bfdom(\bm{F}),  \quad D[\varepsilon,\pi]  \in \bfdom(\bm{\phi}).
\end{equation}
\end{assumption}

Note that the qualification assumption implies the feasibility of Problems \eqref{eq:pot_pb_conv} and \eqref{eq:pot_pb}.

\begin{remark}
Assume that $\interior(\bfdom(\bm{F}))$ and $\interior(\bfdom(\bm{\phi}))$ are non-empty sets. Then in this case, Assumption \ref{A:m-pi-kolmogorov} is satisfied if there exists $\pi \in \bfdom(\bm{\ell})$ such that
\begin{equation} \nonumber
 m_1[0,\pi] = m_2[0,\pi]   \in \interior \left(\bfdom(\bm{F}) \cap \mathbb{R}_+(\bar{\mathcal{T}} \times S)\right),  \quad D[0,\pi]  \in \interior(\bfdom(\bm{\phi})).
\end{equation}
\end{remark}

\begin{remark}[Mean field game planning problem and optimal transport]
Let $\bar{m}_T \in \Delta(S)$ be such that $\bar{m}_T(x) > 0$ for any $x \in S$, let $F(T) = \chi_{\left\{\bar{m}_T\right\}}$ and let $\phi = 0$. In this case \eqref{eq:MFG} is a discrete mean field game planning problem (see \cite{achdou2012mean}). 

Now further assume that $F(t) = 0$ for all $t\in \mathcal{T}$, then \eqref{eq:MFG} can be interpreted as an optimal transport problem (see \cite{achdou2012mean,benamou2000computational}).
In this setting, Assumption \ref{A:m-pi-kolmogorov} is satisfied if and only if there exists $\alpha>0$ such that the following holds: 
for any $\varepsilon \in \mathbb{R}(\bar{\mathcal{T}} \times S)$ with $\|\varepsilon\| \leq \alpha$, there exists $\pi \in \bfdom(\bm{\ell})$ and $m_1$ such that
\begin{equation} \nonumber \begin{array}{rlr}
     m_1(t+1,x)  & =  \sum_{y\in S}m_1(t,y)\pi(t,x,y) - \varepsilon(t+1,x), &  x \in S,\\
     m_1(0,x) &= \bar{m}_0(x) + \varepsilon(0), & x \in S, \\ 
     m_1(T,x) & = \bar{m}_T(x) + \varepsilon(T), & x \in S,\\
     m_1(t,x) & \geq 0, & (t,x) \in \bar{\mathcal{T}} \times S.
\end{array}
\end{equation}
\end{remark}

\begin{theorem} \label{thm:qualif}
Let Assumption \ref{A:m-pi-kolmogorov} hold true.
Then the dual problem \eqref{Pb:dual} has a bounded set of solutions and $\val({\normalfont \text{D}}) = \val({\normalfont \tilde{\text{P}}})$.
\end{theorem}

\begin{proof}
The primal problem \eqref{eq:pot_pb_conv} can formulated as follows:
\begin{equation} \label{Pb:P''} \tag{$\mathfrak{P}$}
\inf_{(m_1,w,m_2,D) \in \mathcal{C}}
\mathcal{F}(m_1,w,m_2,D) + \mathcal{G}(\mathcal{A}(m_1,w,m_2,D)),
\end{equation}
where the maps $\mathcal{F}\colon \mathcal{C} \to \mathbb{R} \cup \{ + \infty \}$ and $\mathcal{G}\colon \mathcal{K} \to \mathbb{R}\cup \{ + \infty \}$ and the operator $\mathcal{A} \colon \mathcal{C} \to \mathcal{K}$ are defined by
\begin{equation} \label{eq:F}
\begin{array}{rl}
\mathcal{F}(m_1,w,m_2,D)  = & {\displaystyle \sum_{(t,x) \in \mathcal{T} \times S}  \tilde{\bm{\ell}}[m_1,w](t,x) +  \sum_{t \in \mathcal{T}} \bm{\phi}[D](t) + \sum_{s \in \bar{\mathcal{T}}}\bm{F}[m_2](s),}\\[1.5em]
\mathcal{G}(y_1,y_2,y_3)  = &  \chi(y_1 + \bar{m}_0) + \chi(y_2) + \chi(y_3), \\[1em]
\mathcal{A}(m_1,w,m_2,D) = & (\bm{S}w-m_1,m_1-m_2,\bm{A}w-D).
\end{array}
\end{equation}
We next prove that the qualification condition 
\begin{equation*}
0 \in \interior \left( \dom(\mathcal{G}) -  \mathcal{A} \dom(\mathcal{F}) \right)
\end{equation*}
is satisfied.  This is equivalent to show the existence of $\alpha>0$ such that for any $\varepsilon = (\varepsilon_1,\varepsilon_2,\varepsilon_3) \in \mathcal{K}$, with $\| \varepsilon \|  \leq \alpha$, there exists $(m_1,w,m_2,D) \in \dom(\mathcal{F})$ satisfying
\begin{equation*}  
(\bm{S}w - m_1 +\varepsilon_1,m_1-m_2 + \varepsilon_2,\bm{A}w - D+\varepsilon_3) \in \dom(\mathcal{G}) = \{  - \bar{m}_0\} \times \{0\} \times  \{0\}.
\end{equation*}
This is a direct consequence of Assumption \ref{A:m-pi-kolmogorov}.
Therefore, we can apply the Fenchel-Rockafellar theorem (see \cite[Theorem 31.2]{rockafellar1970convex}) to problem \eqref{Pb:P''}. It follows that the following dual problem has the same value as \eqref{Pb:P''} and possesses 
a solution:
\begin{equation} \label{Pb:dual-fenchel} \tag{$\mathfrak{D}$}
\inf_{(u,\gamma,P) \in \mathcal{K}}
\mathcal{F}^\star(-\mathcal{A}^\star(u,\gamma,P)) + \mathcal{G}^\star(u,\gamma,P).
\end{equation}
It remains to calculate $\mathcal{F}^\star$, $\mathcal{G}^\star$, and $\mathcal{A}^\star$.
For any $(s,x) \in \bar{\mathcal{T}} \times S$, we define
\begin{align} \nonumber Q_{s,x}  & = \begin{cases} \left\{ (a,b) \in \mathbb{R} \times \mathbb{R}(S), \quad \ell^\star(s,x,b) + a  \leq 0 \right\},& \text{ if } s<T,\\
 \left\{ a \in \mathbb{R}, \quad a = 0 \right\},& \text{ if } s=T.
\end{cases}
\end{align}
We then define
\begin{equation} \label{eq:def_Q}
Q = \prod_{(s,x) \in  \bar{\mathcal{T}} \times S} Q_{s,x}.
\end{equation}
For any $(y_1,y_2,y_3,y_4) \in \mathcal{C}$ we have by Lemma \ref{lemma:persp} that
\begin{equation} \label{eq:fenchel_F}
\mathcal{F}^\star(y_1,y_2,y_3,y_4) =  \chi_{Q}(y_1,y_2) + \sum_{t \in \mathcal{T}} \bm{\phi}^\star[y_4](t) + \sum_{s \in \bar{\mathcal{T}}}\bm{F}^\star[y_3](s). \nonumber
\end{equation}
The adjoint operator $\mathcal{A}^\star \colon \mathcal{K}\to \mathcal{C}$ is given by
\begin{equation} \nonumber
\mathcal{A}^\star(u,\gamma,P) = (\gamma-u,\bm{A}^\star P + \bm{S}^\star u,-\gamma,-P).
\end{equation}
It follows that
\begin{equation*}
\mathcal{F}^\star(-\mathcal{A}^\star(u,\gamma,P))   =  \chi_{Q}( u- \gamma, - \bm{A}^\star P - \bm{S}^\star u) + \sum_{t \in \mathcal{T}} \bm{\phi}^\star[P](t) + \sum_{s \in \bar{\mathcal{T}}}\bm{F}^\star[\gamma](s).
\end{equation*}
Moreover, $\mathcal{G}^\star(u,\gamma,P)   =  - \langle u(0,\cdot), m_0\rangle.$ It follows that \eqref{Pb:dual} and \eqref{Pb:dual-fenchel} are equivalent, which concludes the proof of the theorem.
 \end{proof}

\subsection{A new dual problem}

We introduce in this section a new optimization problem, equivalent to \eqref{Pb:dual}.
We define the mapping $\bm{U} \colon \mathcal{U} \to \mathbb{R}(\bar{\mathcal{T}} \times S)$ which associates with $(\gamma,P) \in \mathcal{U}$ the solution $u \in \mathbb{R}(\bar{\mathcal{T}} \times S)$ to the dynamic programming equation
\begin{equation}  \label{definition:U}
\begin{cases}
 u(t,x)+ \bm{\ell}^\star[-\bm{A}^\star P - \bm{S}^\star u](t,x)  = \gamma(t,x) & (t,x) \in \mathcal{T} \times S,\\
 u(T,x) =  \gamma(T,x), & x\in S.
\end{cases}
\end{equation}
We define the following problem
\begin{align} \label{definition:new_dual} \tag{{\~D}}
\max_{(\gamma,P) \in \mathcal{U}} \tilde{\mathcal{D}}(\gamma,P) & := \mathcal{D}(\bm{U}[\gamma,P],\gamma,P) \\
& = \langle \bar{m}_0, \bm{U}[\gamma,P] \rangle -\sum_{t\in \mathcal{T}} \bm{\phi}^\star[P](t) -  \sum_{s\in \bar{\mathcal{T}}} \bm{F}^\star[\gamma](s). \nonumber 
\end{align}
Note that the above dual criterion is of similar nature as the dual criterion in \cite[Remark 2.3]{chow2019algorithm}.

\begin{lemma} \label{lemma:DDprime}
Problems \eqref{Pb:dual} and \eqref{definition:new_dual} have the same value. Moreover, for any solution $(u,\gamma,P)$ to \eqref{Pb:dual}, $(\gamma,P)$ is a solution to \eqref{definition:new_dual}; conversely, for any solution $(\gamma,P)$ to \eqref{definition:new_dual} (there exists at least one), $(\bm{U}[\gamma,P],\gamma,P)$ is a solution to \eqref{Pb:dual}.
\end{lemma}

\begin{proof}
Let $(\gamma,P) \in \mathcal{U}$. Then $(u:=\bm{U}[\gamma,P],\gamma,P)$ is feasible for problem \eqref{Pb:dual} and by definition, $\mathcal{D}(u,\gamma,P)= \tilde{\mathcal{D}}(\gamma,P)$. Therefore, val(\ref{Pb:dual}) $\geq$ val(\ref{definition:new_dual}). 

Conversely, let $(u,\gamma,P)$ be feasible for \eqref{Pb:dual}. Let $\hat{u}= \bm{U}[\gamma,P]$. Now we claim that $\hat{u}(t,x) \geq u(t,x)$, for any $(t,x) \in \bar{\mathcal{T}} \times S$ (this is nothing but a comparison principle for our dynamic programming equation).
The proof of the claim relies on a monotonicity property of $\ell^\star$. Given $b$ and $b' \in \R(S)$, we say that $b \leq b'$ if $b(x) \leq b'(x)$, for all $x \in S$. Since $\ell(t,x,\cdot)$ has its domain included in $\Delta(S)$, we have
\begin{equation*}
b \leq b' \Longrightarrow \ell^\star(t,x,b) \leq \ell^\star(t,x,b').
\end{equation*}
Using the above property, it is easy to prove the claim by backward induction. It follows that $\tilde{\mathcal{D}}(\gamma,P) = \mathcal{D}(\hat{u},\gamma,P) \geq \mathcal{D}(u,\gamma,P)$ and finally, val(\ref{definition:new_dual}) $\geq$ val(\ref{Pb:dual}). Thus the two problems have the same value. 

The other claims of the lemma are then easy to verify.
 \end{proof}

\begin{lemma} \label{lemma:U-concave}
For any $(t,x) \in \mathcal{T} \times S$, the map $(\gamma,P) \in \mathcal{U} \mapsto \bm{U}[\gamma,P](t,x)$ is concave.
\end{lemma}

\begin{proof}
Let $(t,x) \in \mathcal{T} \times S$. Given $\pi \in \Delta$, consider the Markov chain $(X_s^\pi)_{s=t,...,T}$ defined by
\begin{equation} \nonumber
\mathbb{P}\left(X_{s+1}^\pi = y \vert X_{s}^\pi = x \right) = \pi(s,x,y), \quad \forall s= t,...,T-1, \quad X_t^\pi= x.
\end{equation}
By the dynamic programming principle, we have
\begin{equation*}
\mathcal{U}[\gamma,P](t,x)
= \inf_{\pi \in \Delta}
 \mathbb{E} \Big(  \sum_{s=t}^T c_{\gamma,P}(t,X_{s}^\pi,X_{s+1}^\pi,\pi(s,X_s^\pi)) +  \gamma(T,X_T^\pi) \Big).
\end{equation*}
The criterion to be minimized in the above equality is affine with respect to $(\gamma,P)$, thus it is concave. The infimum of a family of concave functions is again concave, therefore, $\mathcal{U}[\gamma,P](t,x)$ is concave with respect to $(\gamma,P)$.
 \end{proof}
As a consequence of the above Lemma, the criterion $\tilde{\mathcal{D}}$ is concave.

%

\section{Connection between the MFG system and potential problems \label{sec:connexion}}

The connection between the MFG system and the potential problems can be established with the help of seven conditions, which we introduce first.
We say that $(m_1,w,m_2,D) \in \mathcal{C}$ and $(u,\gamma,P) \in \mathcal{K}$ satisfy the condition {\normalfont \text{(C1)}} if for any $(t,x) \in \mathcal{T} \times S$,
\begin{equation*}
\begin{array}{l}
\text{either: } \quad  \!
\begin{cases}
\begin{array}{l}
u(t,x)+ \bm{\ell}^\star[-\bm{A}^\star P - \bm{S}^\star u](t,x)  \leq \gamma(t,x), \\
(m_1(t,x),w(t,x, \cdot)) = (0,0),
\end{array}
\end{cases}
\\[2em]
\text{or: } \quad \! \begin{cases}
\begin{array}{l}
u(t,x)+ \bm{\ell}^\star[-\bm{A}^\star P - \bm{S}^\star u](t,x)  = \gamma(t,x), \\
\bm{\ell}[\pi](t,x) + \bm{\ell}^\star[-\bm{A}^\star P - \bm{S}^\star u](t,x) + \langle \pi(t,x) , (\bm{A}^\star P + \bm{S}^\star u)(t,x) \rangle = 0, \\
m_1(t,x) > 0,
\end{array}
\end{cases}
\end{array}
\end{equation*}
where $ \pi(t,x) = w(t,x)/m_1(t,x)$. We say that the conditions {\normalfont \text{(C2-C7)}} are satisfied if
\begin{equation} \nonumber
\begin{array}{clcl}
({\normalfont \text{C2}}) & u(T) = \gamma(T), \qquad & ({\normalfont \text{C5}}) & m_1 =\bm{S}w + \bar{m}_0,\\
({\normalfont \text{C3}}) &  \gamma \in \partial \bm{F}[m_2], & ({\normalfont \text{C6}}) & m_1 = m_2, \\
({\normalfont \text{C4}}) &  P \in \partial \bm{\phi}[D], & ({\normalfont \text{C7}}) & D = \bm{A}w.\\
\end{array}
\end{equation}
We show in the next lemma that the conditions (C1-C7) are necessary and sufficient optimality conditions for \eqref{Pb:P''} and \eqref{Pb:dual-fenchel}.

\begin{lemma} \label{lemma:complementary-condition}
We have that $(m_1,w,m_2,D) \in \mathcal{C}$ and $(u,\gamma,P) \in \mathcal{K}$ are respectively solutions of \eqref{Pb:P''} and \eqref{Pb:dual-fenchel} if and only if the conditions {\normalfont \text{(C1-C7)}} hold. 
\end{lemma}

\begin{proof}
Let $(m_1,w,m_2,D) \in \mathcal{C}$ and $(u,\gamma,P)  \in \mathcal{K}$. We define the two quantities $a$ and $b$ as follows:
\begin{align*}
a  = & \ \mathcal{F}(m_1,w,m_2,D) + \mathcal{F}^\star(-\mathcal{A}^\star(u,\gamma,P)) + \langle (m_1,w,m_2,D),  \mathcal{A}^\star(u,\gamma,P) \rangle,\\[1em]
b = &\ \mathcal{G}(\mathcal{A}(m_1,w,m_2,D))+ \mathcal{G}^\star(u,\gamma,P) - \langle \mathcal{A}(m_1,w,m_2,D), (u,\gamma,P) \rangle. 
\end{align*}
By Theorem \ref{thm:qualif}, $(m_1,w,m_2,D) \in \mathcal{C}$ and $(u,\gamma,P) \in \mathcal{K}$ are respectively solutions of \eqref{Pb:P''} and \eqref{Pb:dual-fenchel} if and only if $a+b = 0$.
Then we have the following decomposition
\begin{align} \nonumber
a  = & \  \sum_{(s,x) \in \mathcal{T} \times S} a_1(t,x)  + \sum_{x \in S} a_2(x)   + \sum_{s \in \bar{\mathcal{T}}} a_3(s) + \sum_{t \in \mathcal{T}} a_4(t),\\[1em]
\nonumber b  = & \ \sum_{t\in \mathcal{T}} b_1(t) +  \sum_{s\in \bar{\mathcal{T}}} b_2(s)  + b_3(s), \nonumber
\end{align}
where
\begin{equation} \nonumber
\begin{array}{rl}
a_1(t,x) := &  \tilde{\bm{\ell}}[m_1,w](t,x) +   \chi_{Q_{t,x}}((\gamma-u)(t,x), (- \bm{A}^\star P - \bm{S}^\star u)(t,x))  \\  & + \langle m_1(t,x), (u - \gamma)(t,x) \rangle +
\langle w(t,x),(\bm{A}^\star P + \bm{S}^\star u)(t,x) \rangle, \nonumber
\\[1em]
a_2(x)  := &\chi_{Q_{T,x}}((\gamma-u)(T,x)) + \langle m_1(T,x), (u - \gamma)(T,x) \rangle,
\\[1em]
a_3(s)  := & \bm{F}[m_2](s) + \bm{F}^\star[\gamma](s) - \langle m_2(s),\gamma(s) \rangle,
\\[1em]
a_4(t)   := & \bm{\phi}[D](t) + \bm{\phi}^\star[P](t) -\langle D(t),P(t) \rangle,
\\[1em]
b_1(t)  := & \chi((\bm{A}w-D)(t))  - \langle P(t),(\bm{A}w-D)(t) \rangle,\\[1em]
b_2(s) := & \chi((\bm{S}w - m_1 + \bar{m}_0)(s)) - \langle u(s),(\bm{S}w - m_1  + \bar{m}_0)(s) \rangle,\\[1em]
b_3(s)  := & \chi((m_1-m_2)(s)) - \langle \gamma(s),(m_1-m_2)(s) \rangle,
\end{array}
\end{equation}
for any $(t,s,x) \in \mathcal{T} \times \bar{\mathcal{T}} \times S$. 
By the Fenchel-Young inequality,
\begin{equation} \nonumber
\begin{array}{llll}
a_1(s,x)\geq  0, &  a_2(x) \geq  0, &  a_3(s)  \geq  0, & a_4(t)  \geq  0, \\
b_1(t) \geq  0,&  b_2(s)  \geq  0, & b_3(s) \geq  0. &
\end{array} 
\end{equation}
Then $a+b = 0$ if and only if
\begin{equation}
\begin{array}{llll}
a_1(s,x) =  0, &  a_2(x) =  0, &  a_3(s)  =  0, & a_4(t)  =  0, \\
b_1(t) =  0,&  b_2(s) =  0, & b_3(s) =  0. &
\end{array}  \label{eq:local-conditions2} 
\end{equation}
By Lemma \ref{lemma:hhstar-fenchel} we have that (C1) holds if and only if $a_1(s,x) = 0$ and it is obvious that (C2-C7) holds if and only if $a_2(x) =  a_3(s) = a_4(t)  =  b_1(t) = b_2(s) = b_3(s) =  0$.
Then the conditions (C1-C7) hold if and only if \eqref{eq:local-conditions2} holds, which concludes the proof.
 \end{proof}

\begin{proposition}\label{mfg-to-potential}
Let $(m_1,\pi,u,\gamma,P) \in \mathcal{R} \times \mathcal{K}$ be a solution to \eqref{eq:MFG} and let 
\begin{equation*} 
w(t,x,\cdot) = m_1(t,x) \pi(t,x,\cdot), \quad m_2 = m_1, \quad D = \bm{A}w,
\end{equation*}
for any $(t,x) \in \mathcal{T} \times S$. Then $(m_1,w,m_2,D)$ and $(u,\gamma,P)$ are respectively solutions to \eqref{Pb:P''} and \eqref{Pb:dual-fenchel}.
Moreover, $(m_1,w)$ is solution to \eqref{eq:pot_pb_conv}, $(m_1,\pi)$ is solution to \eqref{eq:pot_pb}, and $(\gamma,P)$ is solution to \eqref{definition:new_dual}.
\end{proposition}

\begin{proof}
The conditions (C1-C7) are obviously satisfied.
It immediately follows from Lemma \ref{lemma:complementary-condition} that $(m_1,w,m_2,D)$ and $(u,\gamma,P)$ are optimal for \eqref{Pb:P''} and \eqref{Pb:dual-fenchel}.
The optimality of $(m_1,w)$ and $(m_1,\pi)$ is then deduced from the proof of Lemma \ref{lemma:L-L}. The optimality of $(\gamma,P)$ is a consequence of Lemma \ref{lemma:DDprime}.
 \end{proof}

For any $(m,w) \in \mathcal{R}$, $(u,\gamma,P)\in \mathcal{K}$ we define the set $\bm{\pi}[m,w,u,\gamma,P]$ of controls $ \pi \in \Delta$ satisfying
\begin{equation} \nonumber
\pi(t,x,\cdot) = w(t,x,\cdot)/m(t,x)
\end{equation} 
if $m(t,x) >0$ and 
\begin{equation} \nonumber
\pi(t,x, \cdot) \in \ \underset{\rho  \in \Delta(S)}\argmin\ \ell(t,x,\rho) + \sum_{y \in S} \rho(y)(P(t) \alpha(t,x,y)+ u(t+1,y))
\end{equation} 
if $m(t,x) = 0$, for any $(t,x)\in \mathcal{T}\times S$. Note that for any $\pi \in \bm{\pi}[m,w,u,\gamma,P]$, we have $w(t,x,\cdot)= m(t,x) \pi(t,x,\cdot)$, for any $(t,x) \in \mathcal{T} \times S$.
We have now the following converse property to Proposition \ref{mfg-to-potential}.

\begin{proposition} \label{potential-to-mfg}
Let $(m_1,w,m_2,D)$ and $(u,\gamma,P)$ be respectively solutions to \eqref{Pb:P''} and \eqref{Pb:dual-fenchel}. Let $\hat{u}= \bm{U}[\gamma,P]$ and let $\pi \in \bm{\pi}[m,w,\hat{u},\gamma,P]$. Then $(m,\hat{\pi},\hat{u},\gamma,P)$ is a solution to \eqref{eq:MFG}.
\end{proposition}

\begin{proof}
By Lemma \ref{lemma:DDprime}, $(\hat{u},\gamma,P)$ is a solution to \eqref{Pb:dual}. The pairs $(m_1,w,m_2,D)$ and $(\hat{u},\gamma,P)$ are solutions to \eqref{Pb:P''} and \eqref{Pb:dual-fenchel}), respectively, therefore they satisfy conditions (C1-C7), by Lemma \ref{lemma:complementary-condition}. Equations (\ref{eq:MFG},iii-v) are then obviously satisfied. By definition, $\hat{u}$ satisfies (\ref{eq:MFG},i). Finally, (\ref{eq:MFG},ii) is satisfied, by condition (C1) and by definition of the set $\bm{\pi}[m,w,u,\gamma,P]$. It follows that $(m,{\pi},\hat{u},\gamma,P) \in \mathcal{R} \times \mathcal{K}$ is solution to \eqref{eq:MFG}.
 \end{proof}

Since the existence of solutions to \eqref{Pb:P''} and \eqref{Pb:dual-fenchel} has been established in Lemmas \ref{lemma:L-L} and \ref{lemma:DDprime}, we have the following corollary.

 \begin{corollary}
There exists a solution to \eqref{eq:MFG}.
\end{corollary}

We finish this section with a uniqueness result.
 
\begin{proposition} \label{prop:uniqueness}
Let $(m,\pi,u,\gamma,P)$ and $ (m',\pi',u',\gamma',P')$ be two solutions to the coupled system \eqref{eq:MFG}.
Assume that $F$ and $\phi$ are differentiable with respect to their second variable. Then $(u,\gamma,P)=(u',\gamma',P')$. If moreover, for any $(t,x) \in \mathcal{T} \times S$, $\ell(t,x,\cdot)$ is strictly convex, then $(m,\pi)=(m',\pi')$ and thus \eqref{eq:MFG} has a unique solution.
\end{proposition}

\begin{proof}
It follows from Proposition \ref{mfg-to-potential} that $(m,w:=m\pi,m,D:=\bm{A}w )$ is a solution to \eqref{Pb:P''} and that $(u,\gamma,P)$ and $(u',\gamma',P')$  are solutions to \eqref{Pb:dual-fenchel}. Thus by Lemma \ref{lemma:complementary-condition}, the conditions (C3) and (C4) are satisfied, both for $(m,w,m,D)$ and $(u,\gamma,P)$ and for $(m,w,m,D)$ and $(u',\gamma',P')$, which implies that $\gamma= \nabla \bm{F}[m] = \gamma'$
and $P= \nabla \bm{\phi}[D]= P'$. It further follows that $u= \bm{U}[\gamma,P]= \bm{U}[\gamma',P']= u'$.

If moreover $\ell(t,x,\cdot)$ is strictly convex for any $(t,x) \in \mathcal{T} \times S$ then the minimal argument in (\ref{eq:MFG},ii) is unique, which implies that $\pi=\pi'$ and finally that $m=m^\pi=m^{\pi'}=m'$. 
 \end{proof}

\section{Numerical methods \label{sec:numerics}}

In this section we investigate the numerical resolution of the problems \eqref{Pb:P''} and \eqref{Pb:dual-fenchel}. We investigate different methods: primal-dual proximal algorithms, ADMM and ADM-G.
For all methods, it is assumed that the computation of the prox operators (defined below) of $\tilde{\ell}(t,x,\cdot)$, $F(t,\cdot)$ and $\phi(t,\cdot)$ are tractable.
Note that for the method involving the Kullback-Leibler distance in Subsubsection \ref{sec:kl}, the prox of $\tilde{\ell}$ is replaced by a stightly more complex optimization problem.

We explain in the Appendix \ref{Appendix} how to calculate the prox of $\ell$ (and the nonlinear proximator based on the entropy function) in the special case where $\ell$ is linear on its domain.
We explain in Section \ref{sec:errors} how to recover a solution to \eqref{eq:MFG}.

\subsection{Notations}

Let $X_1$ be a subset of $\R^d$, let $\bar{X}_1$ denote its closure. Let $f \colon \bar{X}_1 \rightarrow \R$. Assume that the following assumption holds true.

\begin{assumption} \label{assumption:bregman}
The set $\bar{X}_1$ is convex and
the map $f$ is continuous and 1-strongly convex on $\bar{X}_1$. There exists an open subset $X_2$ containing $X_1$ such that $f$ can be extended to a continuous differentiable function on $X_2$.
\end{assumption}

We define then the Bregman distance $d_f \colon X_1 \times X_1 \rightarrow \R$ by
\begin{equation} \nonumber
d_f(x,y)  = f(x) - f(y) - \langle \nabla f(y), x-y \rangle. \label{def:d_f}
\end{equation}
%
If $f$ is the Euclidean distance $\frac{1}{2}|\cdot |^2$, then $d_f(x,y) = \frac{1}{2} |x-y|^2$.

Given a l.s.c., convex and proper function $g \colon \R^d \to \mathbb{R}$, we define its proximal operator $\prox_g \colon \R^d \rightarrow \R^d$ as follows:
\begin{equation} \nonumber
 \prox_g(x) = \argmin_{y \in \mathbb{R}^d} \frac{1}{2} |x-y|^2 + g(y).
\end{equation}
For any non-empty, convex and closed $K \subseteq \mathbb{R}^d$, we define the projection operator $\proj_K$ of $x\in \mathbb{R}^d$ on $K$ by
\begin{equation} \nonumber
 \proj_K(x) =  \prox_{\chi_K}(x).
\end{equation}

Finally, we denote $\bar{\alpha}(t) = \sum_{(x,y) \in S \times S} \alpha(t,x,y)^2$ for any $t \in \mathcal{T}$.

\subsection{Primal-dual proximal algorithms \label{sec:chambolle-pock}}

In this subsection we present the primal-dual algorithms proposed by A.\@ Chambolle and T.\@ Pock in \cite{chambolle2011first} and \cite{chambolle2016ergodic}.
For the sake of simplicity, we denote by $x$ the primal variable $(m_1,w,m_2,D)$ and by $y$ the dual variable $(u,\gamma,P)$.
The primal-dual algorithms rely on the following saddle-point problem
\begin{equation} \label{eq:duality-chambolle-pock}
\min_{x \in \mathcal{C}} \max_{y \in \mathcal{K}} \, \mathcal{L}(x,y) := \mathcal{F}(x) - \mathcal{G}^\star(y) + \langle \mathcal{A} x, y\rangle,
\end{equation}
which is equivalent to problem \eqref{Pb:P''} (defined in the proof of Theorem \ref{thm:qualif}).
Let $\mathcal{C}_1$ and $\mathcal{K}_1$ be two subsets of $\mathcal{C}$ and $\mathcal{K}$, respectively. Let $f \colon \bar{\mathcal{C}}_1 \rightarrow \R$ and let $g \colon \bar{\mathcal{K}}_1 \rightarrow \R$ satisfy Assumption \ref{assumption:bregman}.


For any $\tau,\sigma > 0$ and for any $(x',y') \in \mathcal{C} \times \mathcal{K}$ we define: \\[1em]
\fbox{\parbox{
\dimexpr\linewidth-2
\fboxsep-2
\fboxrule
}{
\begin{algorithmic}
\STATE Iteration $(\hat{x},\hat{y}) = \mathcal{S}_{\tau,\sigma}[d_f,d_g](x',y')$,
\begin{equation} \label{eq:S-proximal-operator}
\begin{cases}  {\normalfont \text{(i)}} & \hat{x} = \argmin_{x \in \mathcal{C}_1}  \,   \mathcal{F} (x) + \langle x,  \mathcal{A}^\star y' \rangle  + \frac{1}{\tau} d_f (x,x'), \\
{\normalfont \text{(ii)}} & \tilde{x} =   2\hat{x} - x', \\
{\normalfont \text{(iii)}} &  \hat{y} = \argmin_{y \in \mathcal{K}_1} \,   \mathcal{G}^\star(y) - \langle \mathcal{A} \tilde{x}, y \rangle + \frac{1}{\sigma} d_g (y,y'). \\
\end{cases}
\end{equation}
\end{algorithmic}}}
\\

\noindent Then we define the following algorithm.

\begin{algorithm}[H]
\caption{Chambolle-Pock} \label{algo:chambolle-pock}
\begin{algorithmic}
\STATE Choose $\sigma,\tau >0$ and  $(x^0,y^0) \in \mathcal{C} \times \mathcal{K}$
\FOR{$0\leq k < N$} 
\STATE {Compute
$(x^{k+1},y^{k+1}) = \mathcal{S}_{\tau,\sigma}[d_f,d_g](x^k,y^k).
$}
\ENDFOR
\RETURN {$(x^N,y^N)$.}
\end{algorithmic}
\end{algorithm}

\begin{theorem} \label{theo:chamboll_pock}
Let $\tau, \sigma >0 $ be such that $\tau \sigma \|\mathcal{A}\|^2  < 1$, where $\| \mathcal{A} \|$ denotes the operator norm of $\mathcal{A}$ (for the Euclidean norm).
Assume that $\text{dom}(\mathcal{F}) \subseteq \bar{\mathcal{C}}_1$ and $\text{dom}(\mathcal{G}^\star) \subseteq \bar{\mathcal{K}}_1$. Assume that the iteration \eqref{eq:S-proximal-operator} is well-defined, that is, the minimal arguments in \text{(i)} and \text{(iii)} exist.
Let
$(x^k,y^k)_{k\in \mathbb{N}}$ denote the sequence generated by the algorithm. For any $k \in \mathbb{N}$ we set
\begin{equation} \label{eq:average_iter}
\bar{x}^k = \frac{1}{k} \sum_{n = 0}^k x^n,
\quad \text{and}
\quad
\bar{y}^k = \frac{1}{k} \sum_{n = 0}^k y^n.
\end{equation}
Let $(x,y) \in \mathcal{C} \times \mathcal{K}$. 
Then the following holds:
\begin{enumerate}
\item \label{point:xbar-ybar-sol} The sequence $(\bar{x}^k)_{k\in \mathbb{N}}$ converges to a solution of \eqref{eq:pot_pb_conv} and the sequence $(\bar{y}^k)_{k\in \mathbb{N}}$ converges to a solution of \eqref{Pb:dual}. 
In addition the saddle-point gap is such that 
\begin{equation} 
\label{gap-lag}
\mathcal{L}(\bar{x}^k,y) - \mathcal{L}(x,\bar{y}^k) \leq \frac{1}{k}\left(d_f(x,\bar{x}^k)/\tau + d_g(y,\bar{y}^k))/\sigma - \langle \mathcal{A}(x-x^0), (y-y^0) \rangle \right).
\end{equation}

\item  \label{point:sol-mfg} If $f$ and $g$ are the Euclidean distance $\frac{1}{2}|\cdot |^2$, then the sequence $(x^k)_{k\in \mathbb{N}}$ converges to a solution of \eqref{eq:pot_pb_conv} and the sequence $(y^k)_{k\in \mathbb{N}}$ converges to a solution of \eqref{Pb:dual}.
\end{enumerate}
\end{theorem}

\begin{proof}
Point \ref{point:xbar-ybar-sol} holds as a direct application of \cite[Theorem 1, Remark 3]{chambolle2016ergodic}.
Point \ref{point:sol-mfg} holds as a direct application of \cite[Theorem 1]{chambolle2011first}, applied with $\theta= 1$.
 \end{proof}

\begin{remark}
Fix $(x,y)$, solution to \eqref{eq:duality-chambolle-pock}.
Let $(\hat x,\hat y) \in \mathcal{C} \times \mathcal{K}$ .
Then we have that
$0 \leq \delta(\hat x) := \mathcal{L}(\hat x,y) - \mathcal{L}(x,y)$
and
$0\leq \delta'(\hat y) := \mathcal{L}(x,y) - \mathcal{L}(x,\hat y)$,
with equality if $\hat x$ (resp.\@ $\hat y$) is a primal
(resp.\@ dual) solution. These measures of optimality
(for the saddle-point problem) trivially satisfy
\begin{equation}
0 \leq \delta(\hat x) + \delta'(\hat y) =
\mathcal{L}(\hat{x}^k,y) - \mathcal{L}(x,\hat{y}^k),
\end{equation}
for which an upper-bound is provided by 
\eqref{gap-lag}.
\end{remark}

\begin{lemma} \label{lemma:bounded-linear-operator}
Let $a = \max_{t \in  \mathcal{T}} \bar{\alpha}(t)$. Then $\|\mathcal{A} \| \leq \sqrt{\max \left\{ n + a, 4 \right\}}$,
where $n$ is the cardinal of the set $S$.
\end{lemma}
\begin{proof}
For any $(m_1,w,m_2,D)\in \mathcal{C}$, we have
\begin{align} \nonumber
|\mathcal{A}(m_1,w,m_2,D)|^2 & \leq  |\bm{S}w-m_1|^2 + |m_1-m_2|^2 + |\bm{A}w-D|^2 \\
& \leq (\|\bm{A}\| + \|\bm{S}\|) |w|^2 + 4 |m_1|^2 + 2 |m_2|^2  + 2 |D|^2. \nonumber
\end{align}
We have $\|\bm{A}\| \leq a$ and  $\|\bm{S}\| \leq n$, which concludes the proof.
 \end{proof}

\subsubsection{Euclidean distance \label{sec:euclidean-distance}}

Now we explicit the update rule \eqref{eq:S-proximal-operator} in the case where $f$ and $g$ are both equal to the Euclidean distance $\frac{1}{2}|\cdot |^2$ and defined on $\mathcal{C}$ and $\mathcal{K}$ respectively.
In this situation (i,\ref{eq:S-proximal-operator}) and (iii,\ref{eq:S-proximal-operator}) can be expressed via proximal operators:
\begin{equation} \label{eq:euclidean-S-prox}
\begin{cases}  {\normalfont \text{(i)}} & \hat{x} =  \prox_{\tau \mathcal{F}}(x' - \tau \mathcal{A}^\star y'), \\
{\normalfont \text{(iii)}} &  \hat{y} =  \prox_{\sigma \mathcal{G}^\star}(y' + \tau \mathcal{A} \tilde{x}). \\
\end{cases}
\end{equation}
Now we detail the computation of the proximal steps in the above algorithm.

\paragraph{Primal step.}
For any $x = (x_1,x_2,x_3,x_4) \in \mathcal{C}$, we have by Moreau's identity
\begin{equation} \nonumber
\prox_{\tau \mathcal{F}}(x) = x - \tau\prox_{\mathcal{F}^\star/\tau}(x/\tau).
\end{equation}
As a consequence of \eqref{eq:fenchel_F}, the proximal operator of $\mathcal{F}^\star$ is given by
\begin{equation} \nonumber  \prox_{\mathcal{F}^\star}(x)  =  \argmin_{x' \in \mathcal{C}} \frac{1}{2} | x - x' |^2 + \chi_{Q}(x_1',x_2') + \sum_{s \in \bar{\mathcal{T}}} F^\star(s,x_3'(s)) + \sum_{t \in \mathcal{T}} \phi^\star(t,x_4'(t)).
\end{equation}
 Then (i,\ref{eq:euclidean-S-prox}) is given by
\begin{align}\nonumber
 (\hat{m}_1,\hat{w})  = \ & (m_1' - \tau (\gamma' - u'),w' - \tau (\bm{A}^\star P' + \bm{S}^\star u')) \\
  & - \tau \proj_{Q}(m_1'/\tau - \gamma' + u'),w'/\tau - \bm{A}^\star P' - \bm{S}^\star u'), \label{eq:proj-details}
\end{align}
and for any $(t,s) \in \mathcal{T} \times \bar{\mathcal{T}}$,
\begin{equation} \label{eq:m_2-details}
\begin{array}{rl}
 \hat{m}_2(s) = &  m_2'(s) + \tau \gamma'(s) - \tau \prox_{F^\star(s)/\tau}(m_2'(s)/\tau + \gamma'(s)), \\
 \hat{D}(t) = &  D'(t) + \tau P'(t) - \tau \prox_{\phi^\star(t)/\tau}(D'(t)/\tau + P'(t)). 
 \end{array}
\end{equation}

\paragraph{Dual step.}
It follows from \eqref{eq:F} that
$\prox_{\sigma \mathcal{G}^\star}(y_1,y_2,y_3)
= (y_1 + \sigma \bar{m}_0, y_2 ,y_3).$
Then (iii,\ref{eq:euclidean-S-prox}) is given by
\begin{equation*}
\hat{u} = u' + \sigma(\bm{S}\tilde{w} - \tilde{m}_1 + \bar{m}_0), \quad
\hat{\gamma} = \gamma' + \sigma(\tilde{m}_1 - \tilde{m}_2), \quad 
\hat{P}  = P' + \sigma (\bm{A} \tilde{w} - \tilde{D}).
\end{equation*}

\begin{remark}
An alternative formulation of the primal problem (avoiding the decoupling $m_1$ and $m_2$) is as follows:
\begin{equation} \nonumber
\inf_{(m,w) \in \mathcal{R}} \bar{\mathcal{F}}(m,w)+\bar{\mathcal{G}}(\bar{\mathcal{A}}(m,w)),
\end{equation}
where
\begin{equation}  \nonumber
    \left\{ \begin{array}{rl}
        \bar{\mathcal{F}} \colon & (m,w) \mapsto \sum_{(t,x)} \tilde{\bm{\ell}}[m,w](t,x) + \sum_{s} \bm{F}[m](s), \\
        \bar{\mathcal{G}} \colon & (y,D) \mapsto  \chi_{\bar{m}_0}(y) + \phi(D),  \\
        \bar{\mathcal{A}} \colon & (m,w) \mapsto (\bm{S}w - m, \bm{A} w).
    \end{array} \right.
\end{equation}
This formulation may have numerical advantages since the operator
$\bar{\mathcal{A}}$ has a smaller norm than $\mathcal{A}$. In full generality, there is however no analytic form for the proximal operator of the $\bar{\mathcal{F}}$ that would be based on the proximal operators of $\tilde{\ell}(t,\cdot)$ and $F(t,\cdot)$. Therefore we do not explore any further this formulation.
\end{remark}

\subsubsection{Kullback-Leibler divergence} \label{sec:kl}

In this section we slightly modify the Euclidean framework above. Instead of considering a Euclidean distance $d_f$ in (i,\ref{eq:S-proximal-operator}), we consider an entropy based Bregman distance called Kullback-Leibler divergence.  Let us define
\begin{align*}
\mathcal{C}_1= & \ \Big\{
(m_1,w,m_2,D) \in \mathcal{C} \,|\,  m_1(t,x),w(t,x,y) \in (0,1], \, (t,x,y) \in \mathcal{T} \times S \times S \Big\}, \\
\mathcal{C}_2= & \ \Big\{
(m_1,w,m_2,D) \in \mathcal{C} \,|\,  m_1(t,x),w(t,x,y) \in (0,2), \, (t,x,y) \in \mathcal{T} \times S \times S \Big\}.
\end{align*}
For any $(m_1,w,m_2,D) \in \bar{\mathcal{C}}_1$, we define
\begin{align} \notag
f(m_1,w,m_2,D) = & \sum_{(s,x)  \in \bar{\mathcal{T}} \times S} m_1(s,x) \ln(m_1(s,x)) \\
& \quad +  \sum_{(t,x,y) \in \mathcal{T} \times S^2} w(t,x,y) \ln(w(t,x,y)) + \frac{1}{2}|(m_2,D) |^2, \label{eq:f_kullback}
\end{align}
 with the convention that $0 \ln (0) = 0$.
Then, given $(m_1,w,m_2,D) \in \mathcal{C}_1$ and $(m_1',w',m_2',D') \in \mathcal{C}_1$, we have
\begin{align} \nonumber
d_f((m_1,w,m_2,D),(m_1',w',m_2',D')) = & \ d_{KL}((m_1,w),(m_1',w'))  \\ & \quad + \frac{1}{2}|(m_2,D) -  (m_2',D')|^2, \nonumber
\end{align}
where
\begin{align} \nonumber
d_{KL}((m_1,w),(m_1',w')) = \sum_{(s,x)  \in \bar{\mathcal{T}} \times S}  m_1(s,x) (\ln(m_1(s,x)/m_1'(s,x)) - 1) \\ +  \sum_{(t,x,y) \in \mathcal{T} \times S^2} w(t,x,y)( \ln(w(t,x,y)/w'(t,x,y)) - 1).
\end{align}

As can be easily verified, the map $f$ is $1$-strongly convex on $\bar{\mathcal{C}}_1$. The domain of $\mathcal{F}$ is not contained in $\bar{\mathcal{C}}_1$ in general (as required by Theorem \ref{theo:chamboll_pock}), however $f$ is not $1$-strongly convex on $\mathcal{C}$. This is a minor issue, since any solution to \eqref{eq:duality-chambolle-pock} lies in $\bar{\mathcal{C}}_1$, thus we can replace $\mathcal{F}$ by $\mathcal{F} + \chi_{\bar{\mathcal{C}}_1}$ without modifying the solution set to the problem.

%

Compared to the Subsection \ref{sec:euclidean-distance}, the computations of
\eqref{eq:m_2-details} still hold. The projection step \eqref{eq:proj-details} is now replaced by
\begin{align} \nonumber
 (\hat{m}_1,\hat{w})  = & \argmin_{(m_1,w) \in \mathcal{R}} \sum_{(t,x) \in \mathcal{T} \times S}  \tilde{\bm{\ell}}[m_1,w](t,x) +\langle m_1,\gamma'-u' \rangle +   \langle w ,\bm{A}^\star P' + \bm{S}^\star u' \rangle  \\
 & \quad + \frac{1}{\tau} d_{KL}((m_1,w),(m_1',w'))
 + \sum_{(t,x) \in \bar{\mathcal{T}}\times S} \chi_{\R^-} \big( m_1(t,x)-1 \big). \label{eq:m1hat-what}
\end{align}

Note that it is not necessary to explicit the constraint $w(t,x,y) \leq 1$ in the above problem; it is satisfied as a consequence of
Assumption \ref{A:l-convex} and Lemma \ref{lemma:persp}.

In general, the computation of this proximal operator can be difficult. In Section \ref{sec:numerics} we consider a linear running cost and explain (in Appendix \ref{Appendix}) how to solve explicitly problem \eqref{eq:m1hat-what} in this specific case.

\subsection{ADMM and ADM-G \label{sec:ADM}}

We now present ADMM and ADM-G.
Introducing the variables
\begin{equation}
(a,b)=(u-\gamma,-\bm{A}^\star P - \bm{S}^\star u)
\end{equation}
and recalling the definition of $Q$ and $\mathcal{F}^*$ (see the proof of Theorem \ref{thm:qualif}), the problem \eqref{Pb:dual} can be written as follows:
\begin{equation} \label{pb:ADMM}  \begin{array}{c}
 \displaystyle \sup_{\begin{subarray}{c}
 (u,\gamma,P) \in \mathcal{K}, \; (a,b) \in Q
\end{subarray}}  \mathcal{D}(u,\gamma,P) \\
 \text{s.t.: }  \begin{cases}
 \begin{array}{ll}
u(s,x) - \gamma(s,x) = a(s,x)  & (s,x) \in \bar{\mathcal{T}} \times S,\\
-\alpha(t,x,y)P(t) -u(t+1,y) = b(t,x,y)  & (t,x,y) \in \mathcal{T} \times S^2.
\end{array}
\end{cases}
\end{array}
\end{equation}

\begin{remark}
Let $D_t $ and $D_x$ be finite difference operators defined for any $(t,x,y) \in \mathcal{T} \times S \times S$ by
\begin{equation*} 
\begin{array}{rl}  D_t [u] (t,x)  & = \begin{cases} u(t+1,x) - u(t,x) & \text{\normalfont if } t < T, \\
-u(T,x) & \text{\normalfont if } t = T, \end{cases}
 \\[2em]
 D_x [u](t,x,y)  &= u(t+1,x) - u(t+1,y).
\end{array}
 \end{equation*}
Since $\dom(\ell(t,x,\cdot)) \subseteq \Delta(S)$, for any $(u,b) \in \mathcal{R}$  we have that
\begin{equation} \nonumber
\bm{\ell}^\star[b + \bm{S}^\star u](t,x)  = \bm{\ell}^\star[b + D_x u ](t,x) - u(t+1,x),
\end{equation}
for any $(t,x) \in \mathcal{T} \times S$. Then we have that $(a,b) \in Q$ if and only if $(\tilde{a},\tilde{b}) \in Q$, where
\begin{equation*} \tilde{a}(t,x) = a(t,x) - u(t+1,x), \qquad  \tilde{b}(t,x,y) = b(t,x,y) + u(t+1,x).
\end{equation*}
Thus the problem \eqref{pb:ADMM} can be alternatively written 
\begin{equation*} \begin{array}{c}
 \displaystyle  \sup_{\begin{subarray}{c}
 (u,\gamma,P) \in \mathcal{K}, \;  (\tilde{a},\tilde{b})  \in Q
\end{subarray}}  \mathcal{D}(u,\gamma,P), \quad
 \text{\normalfont subject to: }  \begin{cases}
D_t u = - \gamma - \tilde{a}, \\
D_x u= A^\star P + \tilde{b}.
\end{cases}
\end{array}
\end{equation*}
This problem is close to the problem studied in \cite[Section 4]{benamou2000computational} in the context of optimal transport theory.
\end{remark}

Let $r>0$.
The Lagrangian and augmented Lagrangian associated with problem \eqref{pb:ADMM} are defined by
\begin{align}
\mathcal{L} =  & \ \mathcal{D}(u,\gamma,P) - \chi_{Q}(a,b) + \langle m, u- \gamma - a\rangle + \langle w,-\bm{A}^\star P -\bm{S}^\star u- b\rangle \label{eq:def_lag} \\
\mathcal{L}_r = & \ \mathcal{L}(u,\gamma,P,a,b,m,w) + \frac{r}{2}| (u- \gamma- \bar{a},-\bm{A}^\star P -\bm{S}^\star u - b) |^2, \notag
\end{align}
when evaluated at $(u,\gamma,P,a,b,m,w)$. Note that their definition is different from the one introduced in \eqref{eq:duality-chambolle-pock}.
We define an ADMM step which consists in the updates of $u$, $(\gamma,P)$ and $(a,b)$ via three successive minimization steps and in the update of $(m,w)$ via a gradient ascent step of the augmented Lagrangian:
\\[1em]
\fbox{\parbox{
\dimexpr\linewidth-2
\fboxsep-2
\fboxrule
}{
\begin{algorithmic}
\STATE Iteration $(\hat{u},\hat{\gamma},\hat{P},\hat{a},\hat{b},\hat{m},\hat{w}) = L_r(u,\gamma,P,a,b,m,w)$,
\begin{equation} \label{step:Lr}
\begin{cases}  {\normalfont \text{(i)}} & \hat{u} \in \argmin_{u \in \mathbb{R}(\bar{\mathcal{T}} \times S)}  \mathcal{L}_r(u,\gamma,P,a,b,m,w),\\
{\normalfont \text{(ii)}} & (\hat{\gamma}, \hat{P}) \in \argmin_{(\gamma,P) \in \mathbb{R}(\bar{\mathcal{T}} \times S) \times \mathbb{R}(\mathcal{T})} \mathcal{L}_r(\hat{u},\gamma,P,a,b,m,w),\\
{\normalfont \text{(iii)}} & (\hat{a},\hat{b}) \in \argmin_{(a,b) \in Q } \mathcal{L}_r(\hat{u},\hat{\gamma},\hat{P},a,b,m,w),\\
{\normalfont \text{(iv)}} & (\hat{m},\hat{w}) = (m,w) + r(\hat{u}-\hat{\gamma}-\hat{a}, - \bm{A}^\star[\hat{P}] - \bm{S}^\star[\hat{u}] - \hat{b}).
\end{cases}
\end{equation}
\end{algorithmic}}}

\subsubsection{ADMM}

The ADMM method is given by Algorithm \ref{Alg:ADMM}.
\begin{algorithm}[H] 
\caption{ADMM  \label{Alg:ADMM}}
\begin{algorithmic}
\STATE Choose $r>0$, $(m^0,w^0) \in \mathcal{R}$, $(u^0,\gamma^0,P^0) \in \mathcal{K}$, $(a^0,b^0) \in Q$ . \\ Let $v^0 = (u,\gamma,P,a,b,m,w)$.
\FOR{$0\leq k < N$}
\STATE ADMM step: $v^{k+1} = L_r(v^k)$,
\ENDFOR
\RETURN {$v^N$.}
\end{algorithmic}
\end{algorithm}
Unlike in \cite{benamou2000computational} this algorithm does not reduce to ALG2, thus we have no theoretical guarantee about the convergence. But as we will see in subsection \ref{section:ADMG}, convergence results are available for ADM-G. The relation (\ref{step:Lr},i) is given by
\begin{equation} \nonumber
u^{k+1} = - (m^k - \bar{m}_0 - \bm{S} w^k)/r + \gamma^k + a^k - \bm{S}[\bm{A}^\star P^{k+1} + b^k].
\end{equation}
The relation (\ref{step:Lr},ii) can be written under a proximal form,
\begin{align*}
\gamma^{k+1}(s) & = \text{prox}_{\bm{F}^\star(s)/r}\left(m^k(s)/r + u^{k+1}(s) - \bar{a}^k(s)\right),
\\[1em]
P^{k+1}(t) &= \text{prox}_{\bm{\phi}^\star(t)/(r\bar{\alpha}(t))} \left(\bm{A}[ w^k/r - \bm{S}^\star u^k - b^k](t)/\bar{\alpha}(t)\right),
\end{align*}
for any $(t,s) \in \mathcal{T} \times  \bar{\mathcal{T}}$. The relation (\ref{step:Lr}, iii) can be written as a projection step
\begin{equation} \label{eq:projection} 
(a^{k+1},b^{k+1})  = \proj_{Q} \big((m^k/r + u^{k+1} - \gamma^{k+1}, w^k/r - \bm{A}^\star P^{k+1} - \bm{S}^\star u^{k+1})\big).
\end{equation}

\subsubsection{ADM-G \label{section:ADMG}}

We explicit now the implementation of the ADM-G algorithm introduced in \cite{he2012alternating}. To fit their framework we define
\begin{equation}
A_1 = \begin{pmatrix} id \\ - \bm{S}^\star \end{pmatrix}, \quad A_2 = \begin{pmatrix} -id & 0 \\ 0 & - \bm{A}^\star \end{pmatrix}, \quad A_3 = \begin{pmatrix} -id & 0 \\ 0 & -id\end{pmatrix}, \nonumber
\end{equation}
with appropriate dimensions, so that the constraint of problem \eqref{pb:ADMM} writes
$A_1 u + A_2 (\gamma,P)+ A_3 (a,b) = 0$.
We define
\begin{equation} \nonumber
M = \begin{pmatrix} r A_2^\star A_2 & 0 & 0 \\ r A_3^\star A_2  & r A_3^\star A_3 & 0 \\ 0 & 0 & id/r \end{pmatrix}, \quad H = \begin{pmatrix} r A_2^\star A_2 & 0 & 0 \\ 0  & r A_3^\star A_3 & 0 \\ 0 & 0 & id/r \end{pmatrix}.
\end{equation}
Then we have
\begin{equation} \nonumber
 (M^\star H^{-1})^{-1} =  \begin{pmatrix} id & -(A_2^{\star} A_2)^{-1} A_2 A_3 & 0 \\ 0 & id & 0 \\ 0 & 0 & id \end{pmatrix}.
\end{equation}

\begin{algorithm}[H] 
\caption{ADM-G  \label{Alg:ADMG}}
\begin{algorithmic}
\STATE Choose $r>0$ and $\xi \in (0,1)$. Let $(m^0,w^0) \in \mathcal{R}$, $(u^0,\gamma^0,P^0) \in \mathcal{K}$, $(a^0,b^0) \in Q$ . \\ Let $v^0 = (\gamma,P,a,b,m,w)$.
\FOR{$0\leq k < N$}
\STATE ADMM step: $(\tilde{u}^{k+1},\tilde{v}^{k+1}) = L_r(u^k,v^k)$, \phantom{$\Big|_\|$}
\STATE
Substitution step:
$\begin{cases}
\begin{array}{rl}
v^{k+1} = &   v^k + \xi(M^\star H^{-1})^{-1} (\tilde{v}^k - v^k), \\
u^{k+1} = &   \tilde{u}^{k+1},
\end{array}
\end{cases}$
\ENDFOR
\RETURN {$(u^N,v^N)$.}
\end{algorithmic}
\end{algorithm}

\begin{theorem} \label{thm:ADMG}
Let $(u^k,\gamma^k,P^k,a^k,b^k,m_1^k,w^k)_{k\in \mathbb{N}}$ be the sequence generated by Algorithm \ref{Alg:ADMG}, and let 
$m_2^k = m_1^k$,
$D^k = \bm{A}w^k$,
for any $k\in \mathbb{N}$. Then the sequence $(m_1^k,w^k,m_2^k,D^k)_{k\in \mathbb{N}}$ converges to a solution of \eqref{eq:pot_pb_conv} and the sequence $(u^k,\gamma^k,P^k)_{k\in \mathbb{N}}$ converges to a solution of \eqref{Pb:dual}. 
\end{theorem}

\begin{proof}
By \cite[Theorem 4.7]{he2012alternating} we have that $(u^k,\gamma^k,P^k,a^k,b^k,m_1^k,w^k)_{k\in \mathbb{N}}$ converges to a saddle-point of the Lagrangian \eqref{eq:def_lag}. Thus by definition of $(m_2^k,D^k)$, the sequence $(m_1^k,w^k,m_2^k,D^k)_{k\in \mathbb{N}}$ converges to a solution of \eqref{Pb:P''} and the sequence $(u^k,\gamma^k,P^k)_{k\in \mathbb{N}}$ converges to a solution of \eqref{Pb:dual-fenchel}.
 \end{proof}

\begin{remark}
 In our case the first equality of the Gaussian back substitution step in Algorithm \eqref{Alg:ADMG} can be written
\begin{align}\nonumber
v^{k+1} & = v^k + \xi (M^\star H^{-1})^{-1} (\tilde{v}^k - v^k) \\ \nonumber
& = (\tilde{\gamma}^k - \xi (\tilde{a}^k - a^k), \tilde{P}^k - \xi (\bm{A} \bm{A}^\star)^{-1}\bm{A}^\star (\tilde{b}^k - b^k), \tilde{a}^k, \tilde{b}^k, \tilde{m}^k, \tilde{w}^k).
\end{align}
The Gaussian back substitution step is thus given by
\begin{align}  \label{eq:modif-gamma-ADMG}
\gamma^{k+1} &= \tilde{\gamma}^k - \xi (\tilde{a}^k - a^k),\\ \label{eq:modif-P-ADMG}
P^{k+1} &= \tilde{P}^k - \xi (\bm{A} \bm{A}^\star)^{-1}\bm{A}^\star (\tilde{b}^k - b^k),\\ \nonumber
(u^{k+1},a^{k+1},b^{k+1},m^{k+1},w^{k+1}) &= (\tilde{u}^{k},\tilde{a}^{k},\tilde{b}^{k},\tilde{m}^{k},\tilde{w}^{k}),
\end{align}
where $(\bm{A} \bm{A}^\star)^{-1} P(t) = P(t)/ \bar{\alpha}(t)$ for any $t \in \mathcal{T}$. Then the differences between ADM-G and ADMM can be summarized by the two corrections \eqref{eq:modif-gamma-ADMG} and \eqref{eq:modif-P-ADMG}.
\end{remark}

\subsection{Residuals \label{sec:errors}}

Let $(m_1^k,w^k,m_2^k,D^k)_{k\in \mathbb{N}}$ and $(u^k,\gamma^k,P^k)_{k\in \mathbb{N}}$ denote the two sequences generated by a numerical method.
Let us consider
\begin{equation}
\hat{u}^k= \bm{U}[\gamma^k,P^k]
\quad
\text{and}
\quad
\pi^k \in \bm{\pi}[m_1^k,w^k,\hat{u}^k,\gamma^k,P^k].
\end{equation}
It was shown in Proposition \ref{potential-to-mfg} that if for some $k \in \mathbb{N}$, $(m_1^k,w^k,m_2^k,D^k)$ and $(u^k,\gamma^k,P^k)$ are solutions to \eqref{Pb:P''} and \eqref{Pb:dual-fenchel}, then
$(m_1^k,\pi^k,\hat{u}^k,\gamma^k,P^k)$ is a solution to \eqref{eq:MFG}.

Therefore, we look the sequence $(m_1^k,\pi^k,\hat{u}^k,\gamma^k,P^k)_{k \in \mathbb{N}}$ as a sequence of approximate solutions to \eqref{eq:MFG}.
Note that (\ref{eq:MFG},i) is exactly satisfied, by construction. We consider the residuals  $(\varepsilon_m,\varepsilon_\pi, \varepsilon_\gamma, \varepsilon_P) \in \mathcal{R} \times \mathcal{U}$ defined as follows, in order to measure the satisfaction of the remaining relations in the coupled system:
\begin{equation} \nonumber
\begin{cases}
\begin{array}{rl}
\varepsilon_\pi(t,x) = &    (\bm{\ell}[\pi] +  \bm{\ell}^\star[-\bm{A}^\star P - \bm{S}^\star \hat{u}])(t,x) - \langle \pi(t,x),(\bm{A}^\star P + \bm{S}^\star \hat{u})(t,x) \rangle,\\[.5em]
 \varepsilon_m(s,x) = &     m^\pi(s,x) - m(s,x), \\[.5em]
  \varepsilon_\gamma(s)  = &    m(s) - \proj_{\partial \bm{F}^\star[\gamma](s)}(m(s)),
\\[.5em]
\varepsilon_P(t)  = &     Q[m,\pi](t) - \proj_{ \partial \bm{\phi}^\star[P](t)}(Q[m,\pi](t)\Delta_x),
\end{array}
 \end{cases}
\end{equation}
for all $(t,s,x) \in \mathcal{T} \times \bar{\mathcal{T}}\times S$. If the residuals are null, then $(m_1^k,\pi^k,\hat{u}^k,\gamma^k,P^k)$ is a solution to \eqref{eq:MFG}. The errors are then defined as the norms of $\varepsilon_\pi$, $\varepsilon_m$, $\varepsilon_\gamma$, and $\varepsilon_P$.

\begin{remark}
 In the following numerical section \ref{sec:numerical}, we plot residuals for ADMM, ADMG, Chambolle-Pock and Chambolle-Pock-Bregman algorithms.
 The difference in nature of the convergence results (ergodic versus classical) makes the performance comparison between different methods delicate:
\begin{itemize}
 \item In view of the ergodic convergence result (Theorem \ref{theo:chamboll_pock}, point \ref{point:xbar-ybar-sol}), we plot the residuals associated with the averaged iterates (as defined in \eqref{eq:average_iter}). In particular, the performances of Chambolle-Pock and Chambolle-Pock-Bregman can be compared.
\item Following Theorem \ref{thm:ADMG}, we plot the sequence of residuals for ADM-G and ADMM.
\end{itemize}
\end{remark}

\section{Numerical Results \label{sec:numerical}}

In this section we provide two problems that we solve with the algorithms presented in the previous section.
We set $n = T = 50$ and we define two scaling coefficients $\Delta_x = 1/n$ and $\Delta_x = 1/T$.
We solve two instances of the following scaled system:
\begin{equation} \label{eq:delta-MFG'} \tag{MFG$_\Delta$}
\begin{cases}
\begin{array}{cl}
\text{(i)} &
\begin{cases}
 u(t,x)/\Delta_t + \bm{\ell}^\star[-\bm{A}^\star P - \bm{S}^\star u/\Delta_t](t,x) =  \gamma(t,x),\\
 u(T,x) =  \gamma(T,x),
\end{cases}
\\[1.5em]
\text{(ii)} & (\bm{\ell}[\pi] +  \bm{\ell}^\star[-\bm{A}^\star P - \bm{S}^\star u/\Delta_t])(t,x)  =- \langle \pi(t,x),(\bm{A}^\star P + \bm{S}^\star u/\Delta_t)(t,x) \rangle,
\\[1em]
\text{(iii)} &  \begin{cases}
\begin{array}{rl}
m(t+1,x)= & {\displaystyle \sum_{y \in S} m(t,y) \pi(t,y,x), } \\
m(0,x)= & m_0(x)/\Delta_x,
\end{array}
\end{cases} \\[2em]
\text{(iv)} & \gamma \in \partial \bm{F}[m],
\\[1em]
\text{(v)} & P\in \partial \bm{\phi}[\bm{Q}[m,\pi] \Delta_x].
\end{array}
\end{cases}
\end{equation}
One can show that this system is connected to two optimization problems of very similar nature as Problems \eqref{Pb:P''} and \eqref{Pb:dual-fenchel}, which can be solved as described previously.
For both examples, $\ell$ is defined by
\begin{equation} \label{eq:num_ell}
\ell(t,x,\rho) = \sum_{y\in S} \rho(y) \beta(t,x,y) + \chi_{\Delta(S_x)}(\rho), \; \;  \beta(t,x,y) = \left((y-x)\frac{\Delta_x}{\Delta_t} \right)^2/4,
\end{equation}
where $S_x= \{ x, x-1, x+1 \} \cap \{ 0,..., n-1 \}$.
Since $\ell$ is linear with respect to $\rho$, we can interpret $\beta$ as displacement cost from state $x$ to state $y$ that is fixed (in opposition to the displacement cost induced by $P(t)\alpha(t,x,y)$ that depends on the price). 
In Appendix \ref{Appendix} the reader can find detailed computations of the Euclidean projection \eqref{eq:proj-details} (Subsection \ref{sec:projection}) and the computation of \eqref{eq:m1hat-what} (Subsection \ref{computation:entropic}) for this particular choice of running cost $\ell$. The notion of residuals that we use in the following is adapted from Section \ref{sec:errors} to the scaled system \eqref{eq:delta-MFG'}.
In all subsequent graphs, the state space is represented by $\{ 0, \Delta_x,..., 1 \}$ and the set of time steps by $\{ 0, \Delta_t,...1 \}$.

\subsection{Example 1 \label{sec:resultmfg}}

In our first example, we take $\bm{\phi} = 0$ and $\alpha = 0$. 
We consider a potential $\bm{F}$ of the form
$\bm{F}[m] = \bm{F}_1[m] + \bm{F}_2[m]$,
where
\begin{equation}
\bm{F}_1[m](s) = |m(s)|^2/2, \quad   \bm{F}_2[m](s)  = \chi_{[0,\eta(s)]}(m(s)),
\end{equation}
and where $\eta \in \mathbb{R}_{+}(\bar{\mathcal{T}} \times S)$ is given by
\begin{equation} \nonumber
\eta(s,x) := \begin{cases} 0.5 & \text{ if } T/3 \leq s \leq 2T/3 \quad \text{and} \quad n/3 \leq x \leq 2n/3,\\
3 & \text{ else,} 
\end{cases}
\end{equation}
for any $(s,x) \in \bar{\mathcal{T}} \times \mathbb{R}(S)$.
We refer to $\bm{F}_1$ as the soft congestion term and to $\bm{F}_2$ as the hard congestion term.
\begin{figure}[htb]
  \centering
  \includegraphics[width=6cm,height=6cm]{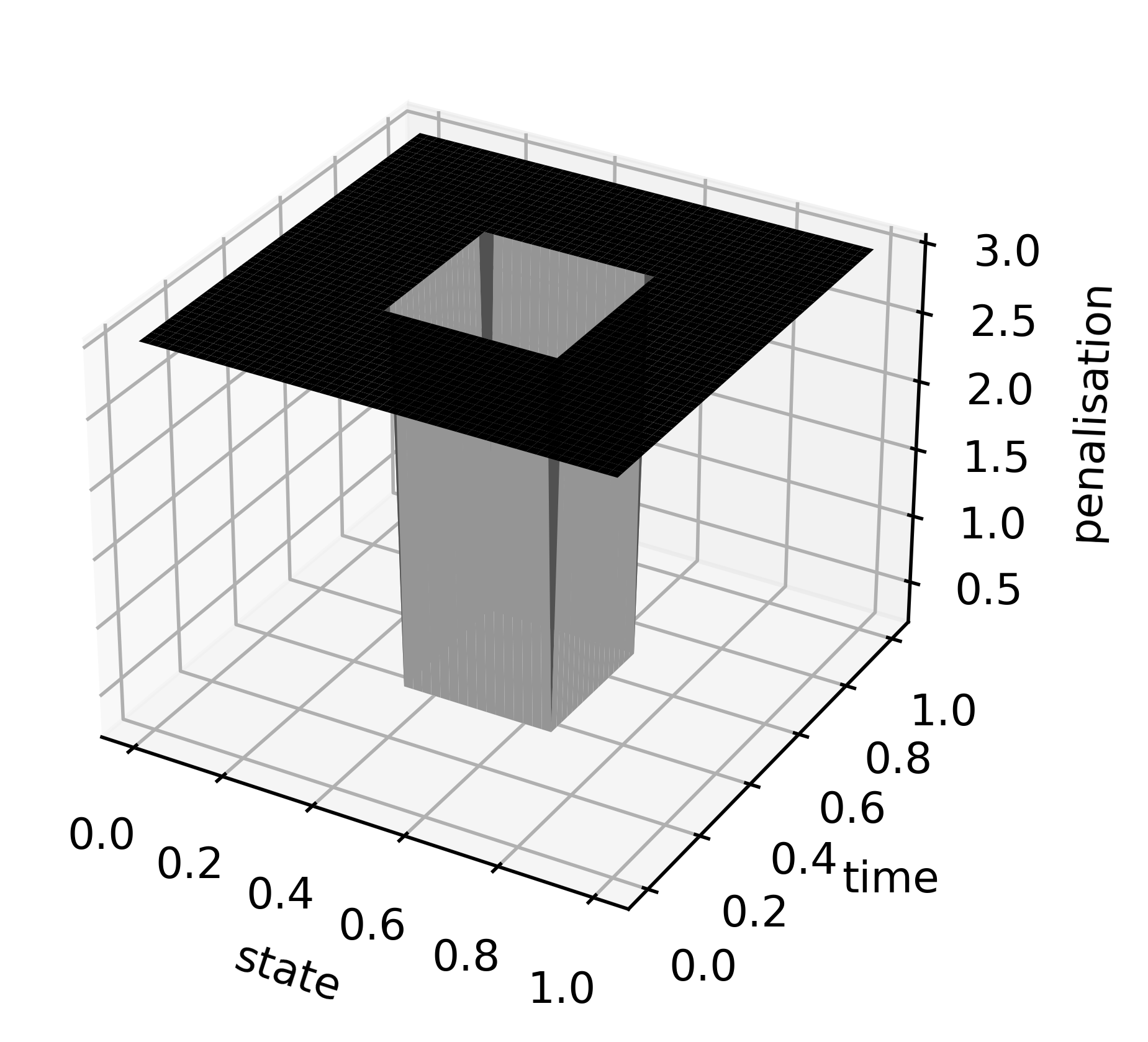}
  \caption{Hard contraint $\eta$}
  \label{fig:penalisation-eta}
\end{figure}
We call narrow region the set of points $(s,x)$ for which $\eta(s,x)= 0.5$ and we call spacious region the set of points for which $\eta(s,x)=3$.
In this situation the state of an agent represents its physical location on the interval $[0,1]$.
Each agent aims at minimizing the displacement cost  induced by $\ell$ and avoids congestion as time evolves from time $t=0$ to $t=1$. The congestion term is linked to $\eta$ 
by the following relation (see Remark \ref{remark:model}):
\begin{equation} \nonumber
\gamma \in \partial \bm{F}[m] = \nabla \bm{F}_1[m] + \partial \bm{F}_2[m]  = m + N_{[0,\eta]}(m).
\end{equation}
As shown on the graphs below, we have two regimes at equilibrium: in the spacious regions $\gamma$ plays the role of a classical congestion term and $\gamma = \nabla \bm{F}_1[m]$. In the narrow region the constraint is binding, $\gamma$ is such that the constraint $m\in [0,\eta]$ is satisfied at the equilibrium and is maximal for the dual problem. 

\begin{figure}[p]
\fbox{
\begin{minipage}[c]{0.95\linewidth}
\begin{subfigure}[c]{0.5\textwidth}
\centering
\includegraphics[width=5cm,height=4cm]{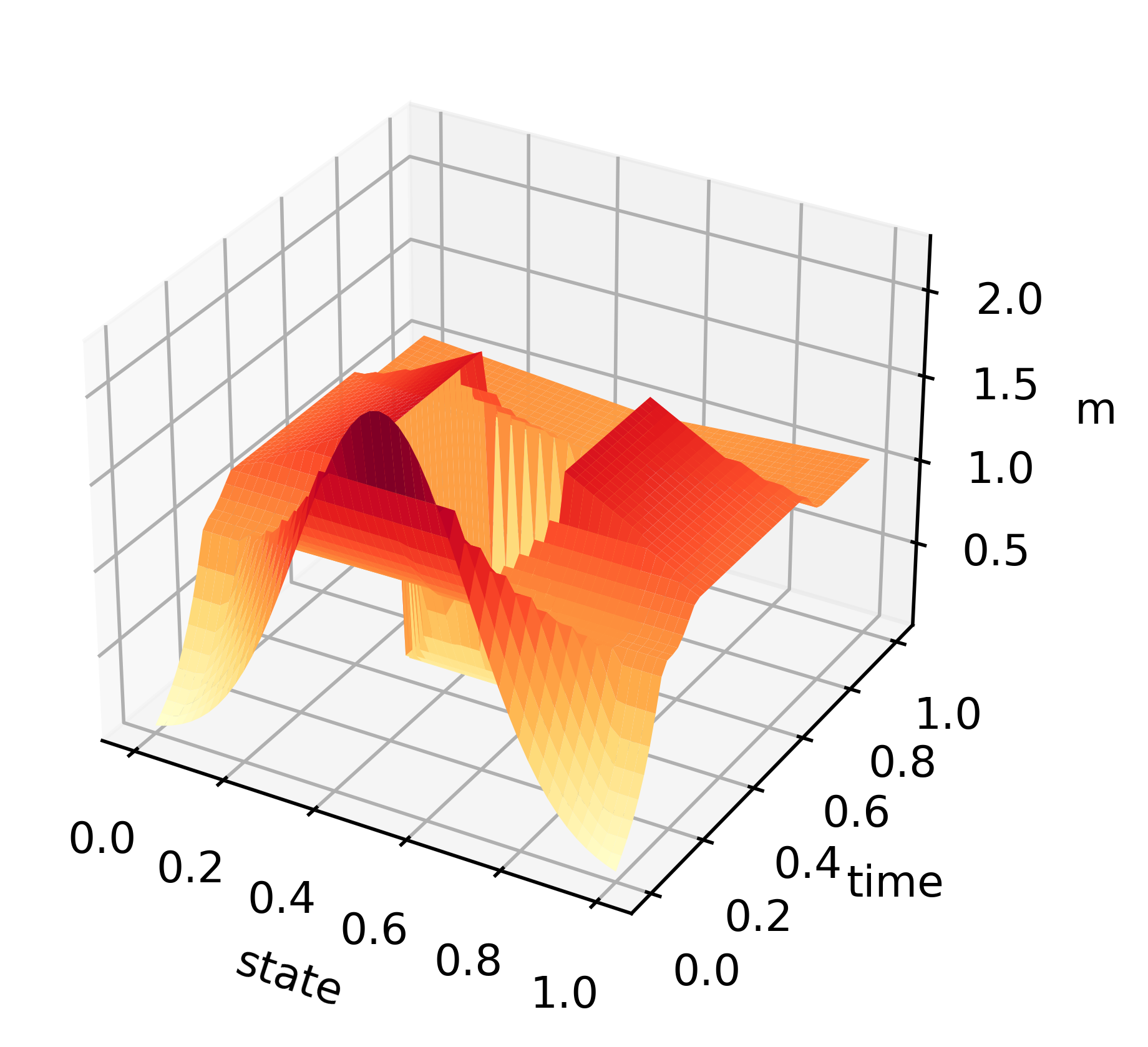}
\caption{Measure $m$, 3d plot}
\end{subfigure}
\begin{subfigure}[c]{0.5\textwidth}
\centering
\includegraphics[width=5cm,height=4cm]{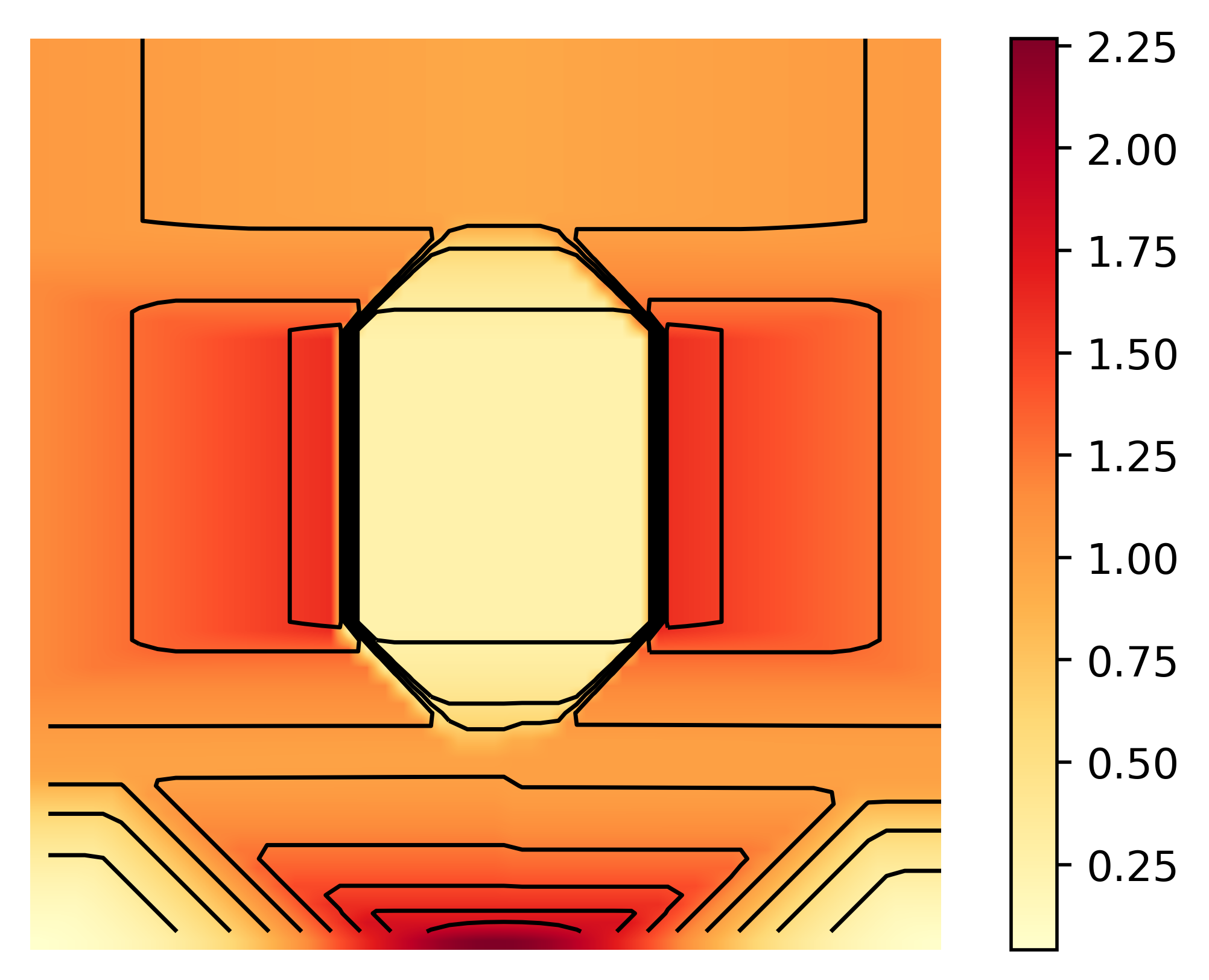}
\caption{Measure $m$, contour plot}
\end{subfigure} \\
\begin{subfigure}[c]{0.5\textwidth}
\centering
\includegraphics[width=5cm,height=4cm]{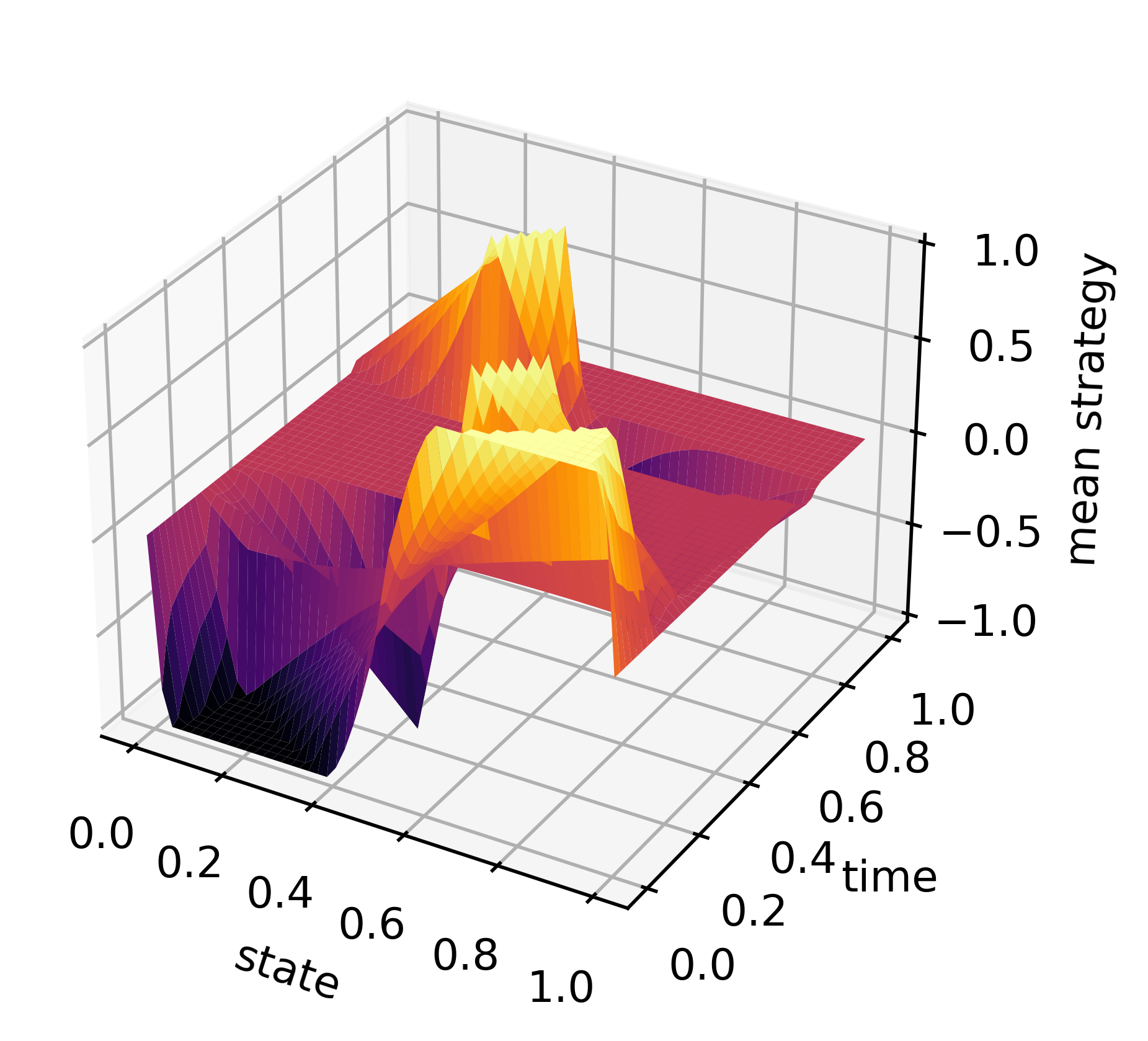}
\caption{Mean displacement $v$, 3d plot}
\end{subfigure}
\begin{subfigure}[c]{0.5\textwidth}
\centering
\includegraphics[width=5cm,height=4cm]{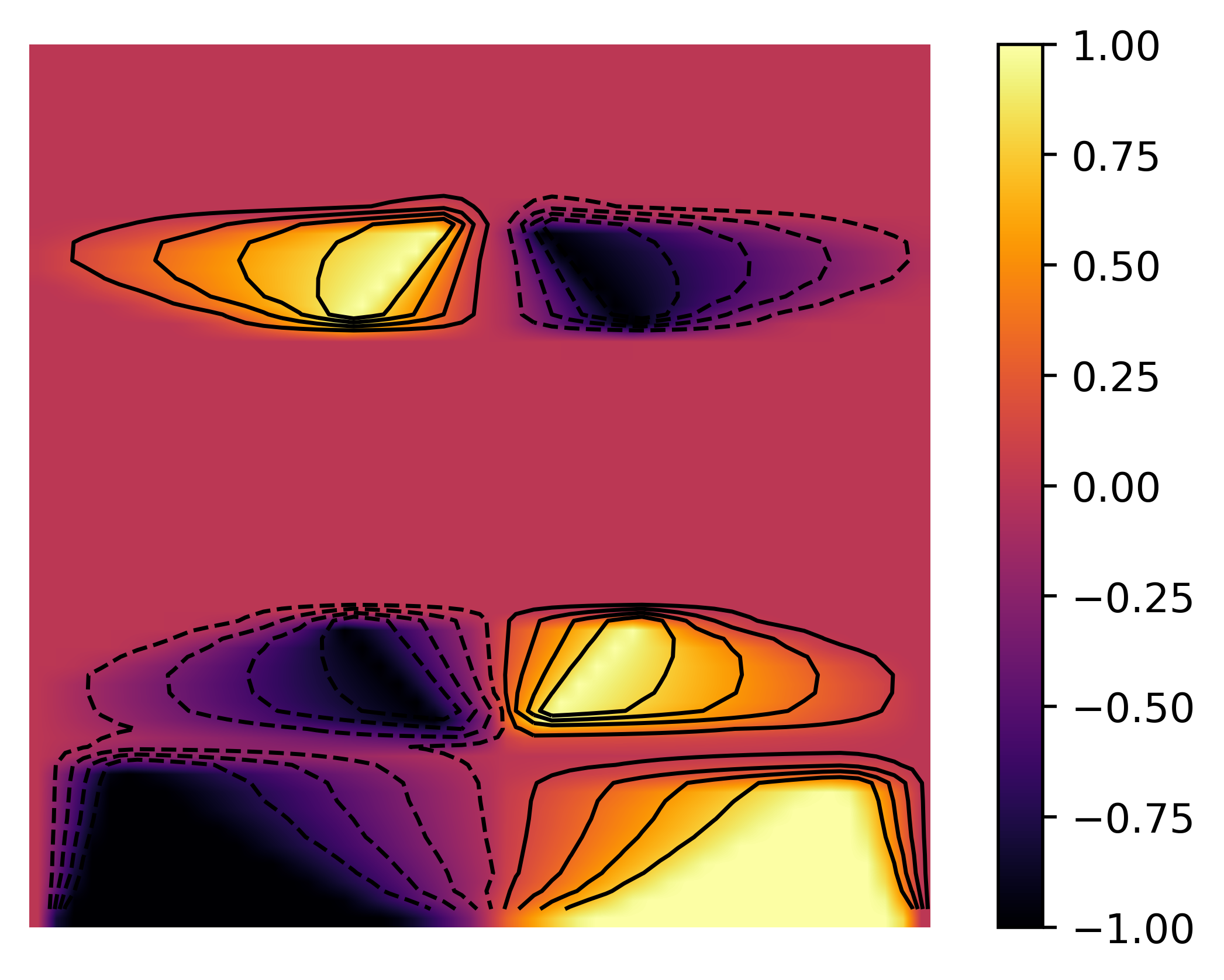}
\caption{Mean displacement $v$, contour plot}
\end{subfigure} \\
\begin{subfigure}[c]{0.5\textwidth}
\centering
\includegraphics[width=5cm,height=4cm]{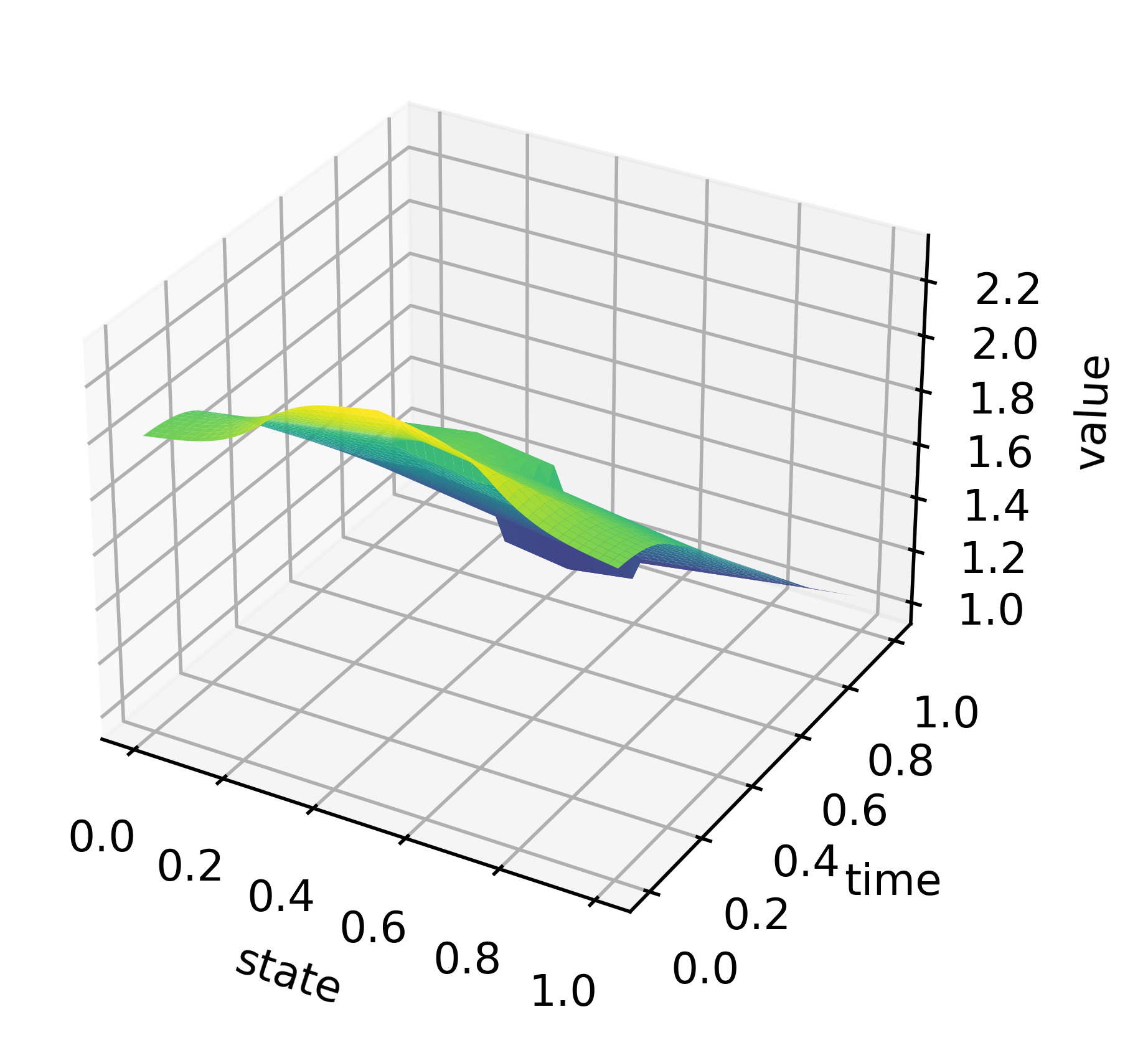}
\caption{Value function $u$, 3d plot}
\end{subfigure}
\begin{subfigure}[c]{0.5\textwidth}
\centering
\includegraphics[width=5cm,height=4cm]{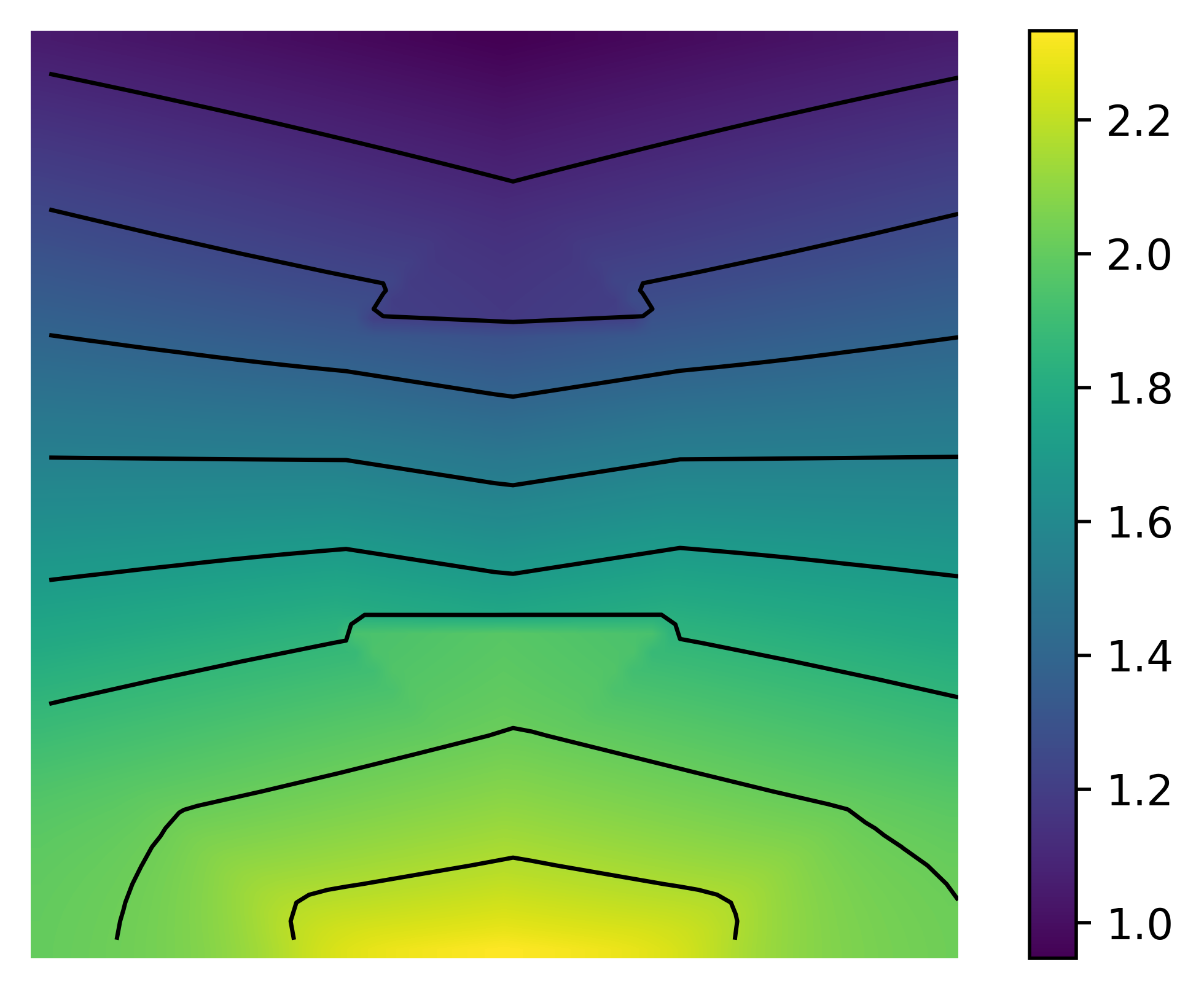}
\caption{Value function $u$, contour plot}
\end{subfigure} \\
\begin{subfigure}[c]{0.5\textwidth}
\centering
\includegraphics[width=5cm,height=4cm]{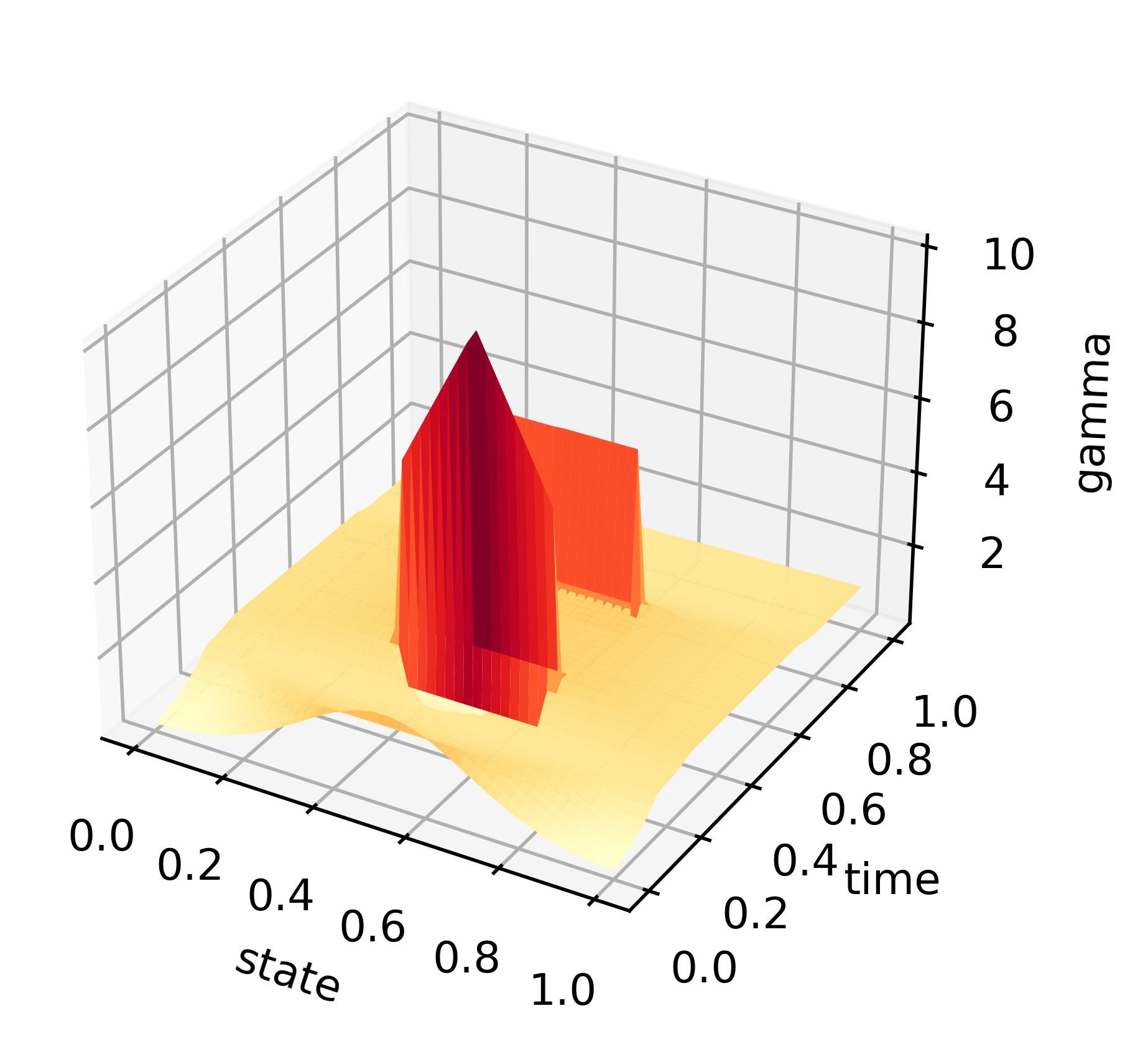}
\caption{Congestion $\gamma$, 3d plot}
\end{subfigure}
\begin{subfigure}[c]{0.5\textwidth}
\centering
\includegraphics[width=5cm,height=4cm]{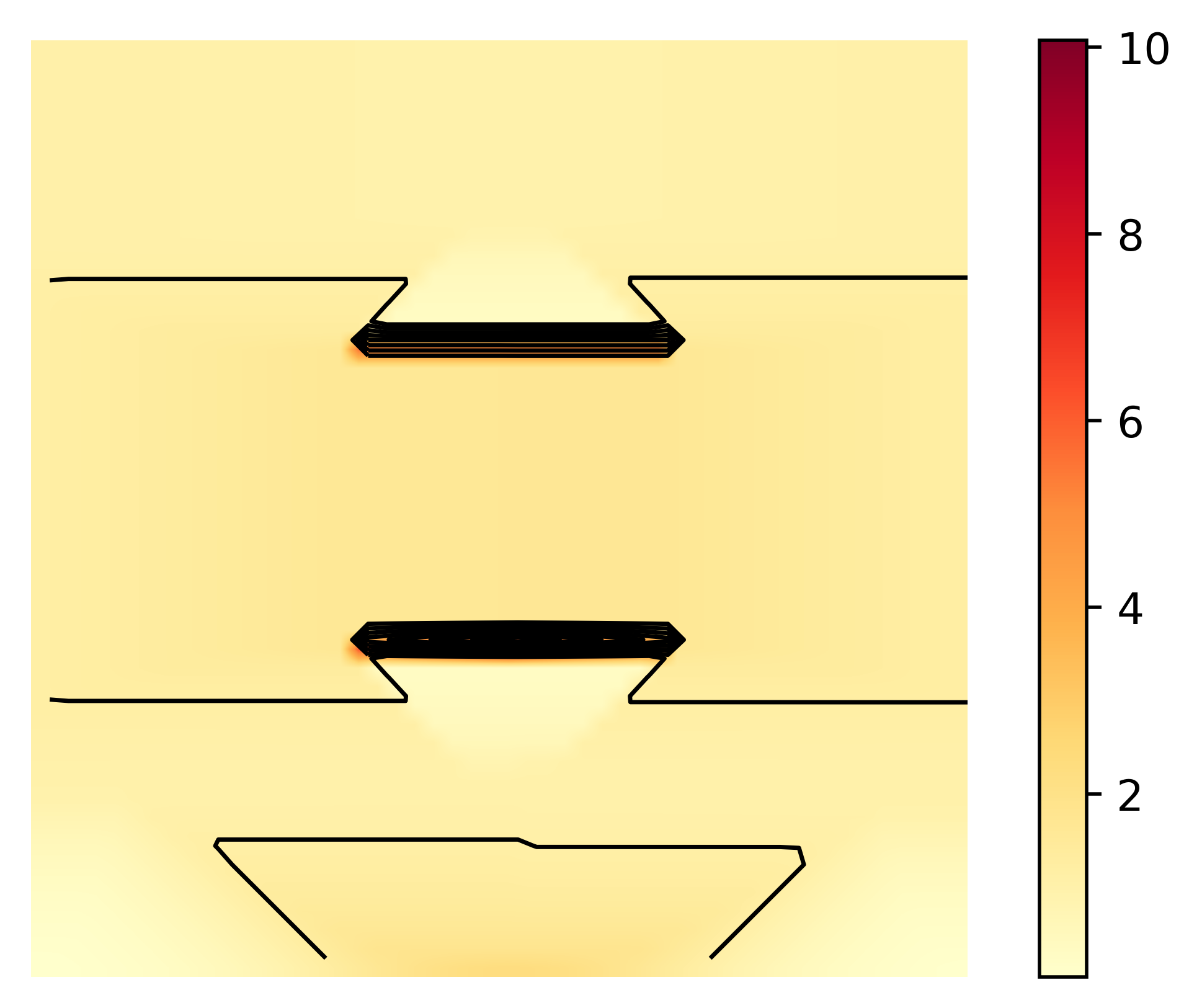}
\caption{Congestion $\gamma$, contour plot}
\end{subfigure}
\caption{Solution of Example 1}
\label{fig:sol_ex1}
\end{minipage}
}
\end{figure}

We give a representation of the solution to the mean field system in Figure \ref{fig:sol_ex1}. Since it is hard to give a graphical representation of $\pi$, we give instead a graph of the mean displacement $v$, defined by
\begin{equation} \nonumber
v(t,x) = \sum_{y\in S} \pi(t,x,y)(y-x), \quad
\forall (t,x) \in \mathcal{T} \times S.
\end{equation}
For each variable, a 3D representation of the graph and a 2D representation of the contour plots are provided. For the contour plots, the horizontal axis corresponds to the state space and the vertical axis to the time steps (to be read from the bottom to the top).


Let us comment the results. We start with the interpretation of the measure $m$ and the mean displacement $v$. 
At the beginning of the game, the distribution of players is given by the initial condition $m(0) = \bar{m}_0$. Then the players spread since they are in the spacious region to avoid congestion.

\begin{figure}[htb]
\fbox{
\begin{minipage}[c]{0.95\linewidth}
\begin{center}
\textbf{Convergence results}
\end{center}
\vspace{2em}
\begin{subfigure}[c]{0.5\textwidth}
\centering
\includegraphics[width=5cm,height=4cm]{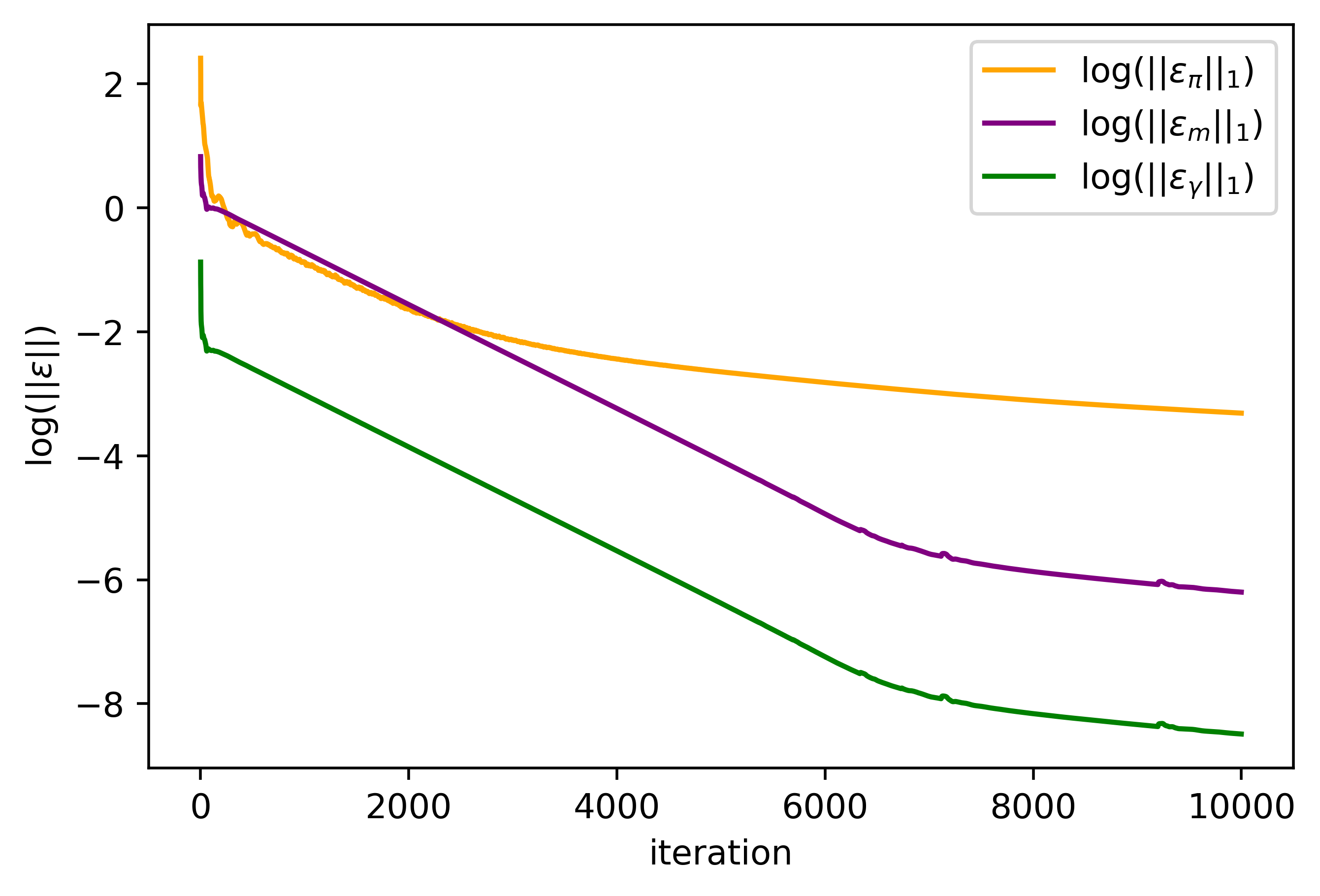}
\caption{ADMM}
\end{subfigure}
\begin{subfigure}[c]{0.5\textwidth}
\centering
 \includegraphics[width=5cm,height=4cm]{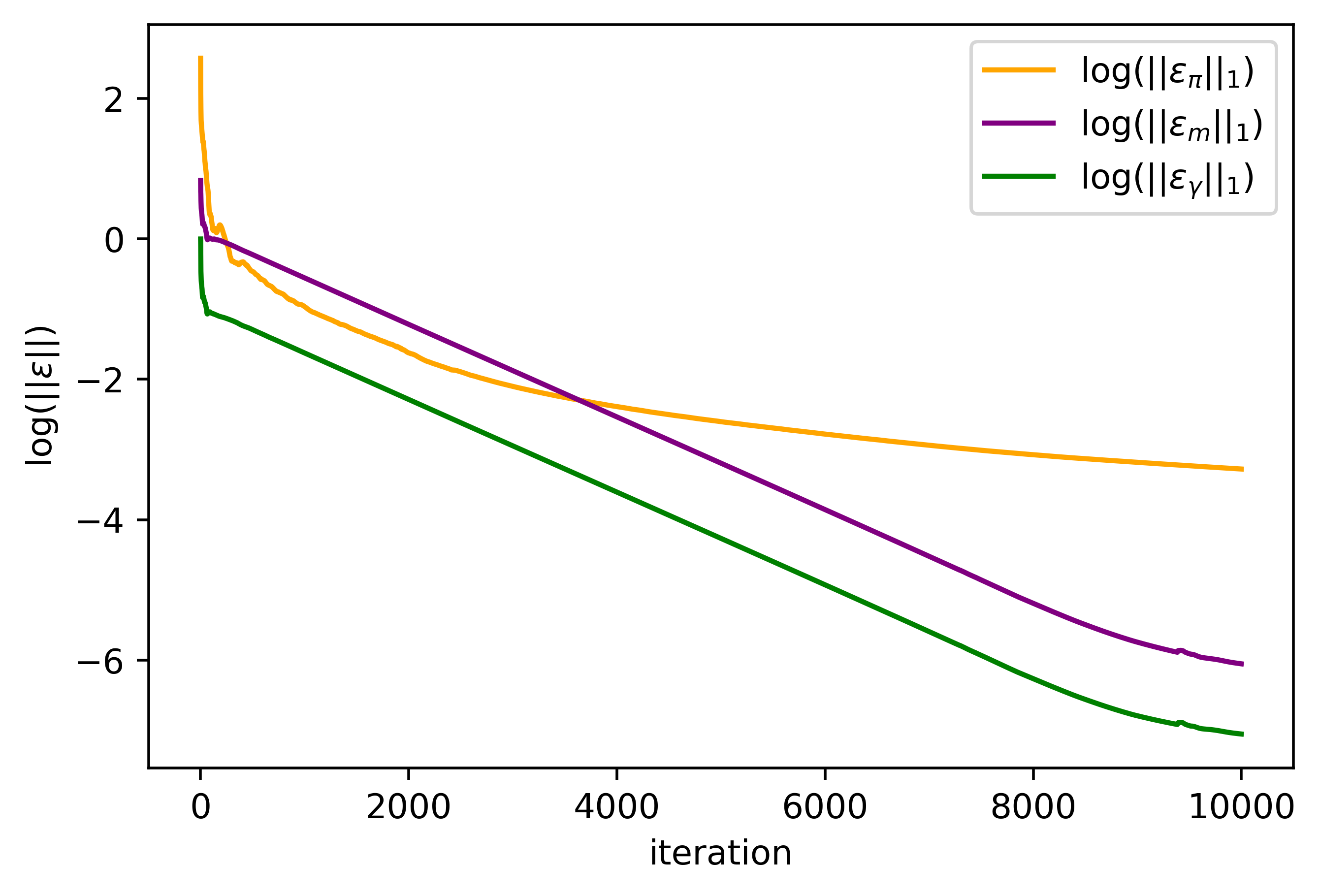} 
 \caption{ADMG}
 \end{subfigure}  \\
 \begin{subfigure}[c]{0.5\textwidth}
 \centering
\includegraphics[width=5cm,height=4cm]{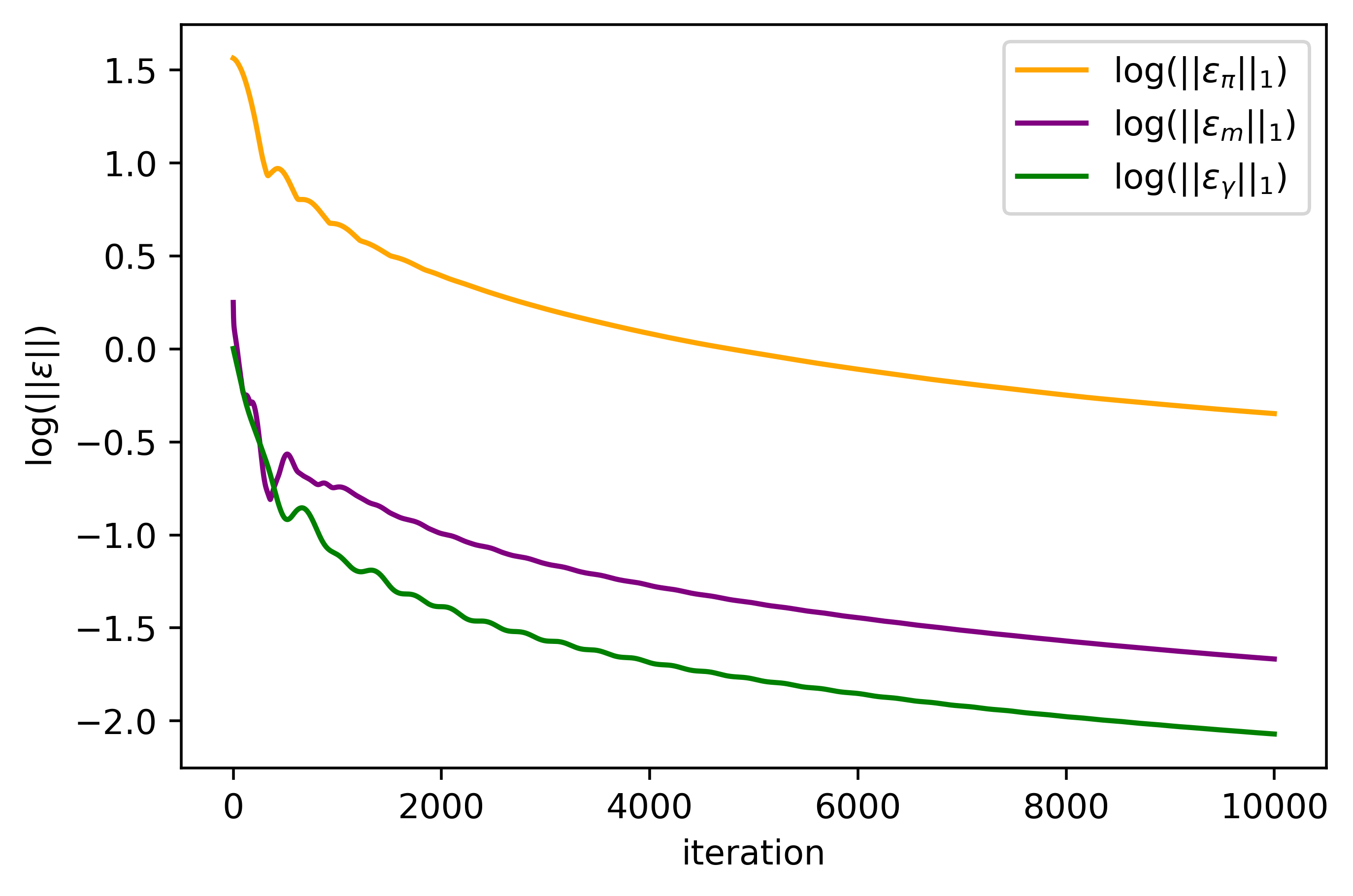} 
\caption{Chambolle-Pock}
\end{subfigure}
 \begin{subfigure}[c]{0.5\textwidth}
 \centering
\includegraphics[width=5cm,height=4cm]{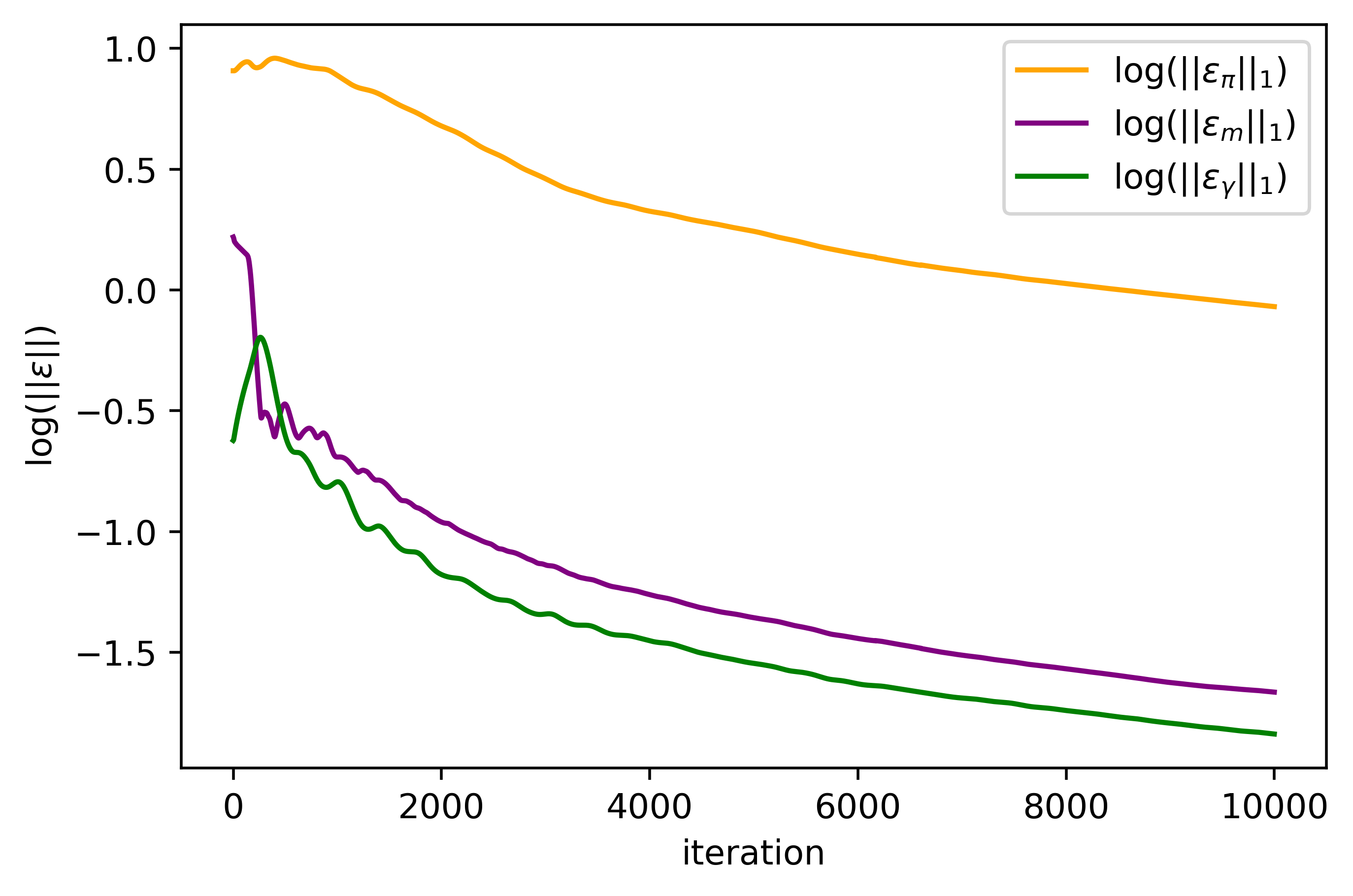}
\caption{Chambolle-Pock-Bregman} 
\end{subfigure}
\caption{Errors plots for Example 1}
\label{fig:error_plots_1}

\vspace{1em}

Figure \ref{fig:error_plots_1} shows the evolution of the error terms in function of the iterations. The execution time of each algorithm is given in the following table.

\vspace{1em}

\centering
\begin{tabular}{|c|c|c|c|c|}
\hline
& Chambolle-Pock & Chambolle-Pock-Bregman & ADMM & ADM-G \\ \hline
Time (s) & 1600 & 1300 & 2000 & 2000 \\ \hline
\end{tabular}
\caption{Execution time of each algorithm for Example 1, with $N= 10000$}
\end{minipage}
}
\end{figure}

Thus the mean displacement is negative on the left (black region) and positive on the right (yellow region), around $t= 0$. The distribution becomes uniform after some time.
In a second phase, the agents move again towards the border of the state space, anticipating the narrow region. They start their displacement before entering into the narrow region due to their limited speed and displacement cost. 
Then we are in a stationary regime (purple region), the mean displacement is null and the mass does not vary until the end of the narrow region. At the end of the narrow region, the agents spread again along the state axis and the distribution $m$ becomes uniform.

We now interpret the value $u$ and the congestion $\gamma$.
The value function has to be interpreted backward in time. At the end of the game, the terminal condition imposes that the value is equal to the congestion. Since the congestion is positive and accumulates backward in the value function (which can be seen in the dynamic programming equation), the value function increases backward in time. At the end and at the beginning of the narrow region we observe irregularities in the value function due to the irregularities of the congestion term $\gamma$. But the impact on the value function is limited due to the trade-off between the variables $u$ and $\gamma$ in the dual problem. At the beginning of the game the value function is higher at the middle of the space because of the initial distribution of players that are accumulated at this point. The congestion term $\gamma$ is high enough at the beginning of the narrow region to ensure that the constraint on the distribution of players is satisfied at this point. Then $\gamma$ is high enough at the end of the narrow region to ensure that the constraint on the distribution of players is satisfied for all time indices $T/3\leq s \leq 2T/3$. At the exception of these two moments, $\gamma$ plays the role of a classical congestion term.

\subsection{Example 2 \label{sec:resultmfgc}}

Here we assume that $\bm{F} = 0$. In this situation the state of an individual agent represents a level of stock. We set $\alpha(t,x,y) = y-x$; it represents the quantity bought in order to ``move" from $x$ to $y$.
Therefore the variable $D$ (used in the primal problem) is the average quantity which is bought; it has to be understood as a demand, since at equilibrium,
\begin{equation} \nonumber
D(t) = \bm{Q}[m,\pi](t) = \sum_{(x,y) \in S^2} m(t,x) \pi(t,x,y) \alpha(t,x,y).
\end{equation}
We define the potential $\bm{\phi}[D] = \bm{\phi}_1[D] + \bm{\phi}_2[D]$,
where 
\begin{equation} \nonumber
\bm{\phi}_1[D] =  \frac{1}{4}(D + \bar{D})^2, \quad \bm{\phi}_2[D] =  \chi_{(-\infty,D_{\text{max}}]}(D).
\end{equation}
The potential $\bm{\phi}$ is the sum of a convex and differential term $\bm{\phi}_1$ with full domain and a convex non-differentiable term $\bm{\phi}_2$. The quantity $\bar{D}$ is a given exogenous quantity which represent a net demand (positive or negative) to be satisfied by the agents. In this example $\bar{D}(t) = 2\sin(4\pi t/(T-1))$ for any $t \in \mathcal{T}$ and $D_{\text{max}} = 0$. 

\begin{figure}[htb]
\fbox{
\begin{minipage}[c]{0.95\linewidth}
\begin{center}
\textbf{Market equilibrium}
\end{center}
\vspace{2em}
\begin{subfigure}[c]{0.5\textwidth}
\centering
\includegraphics[width=5cm,height=4cm]{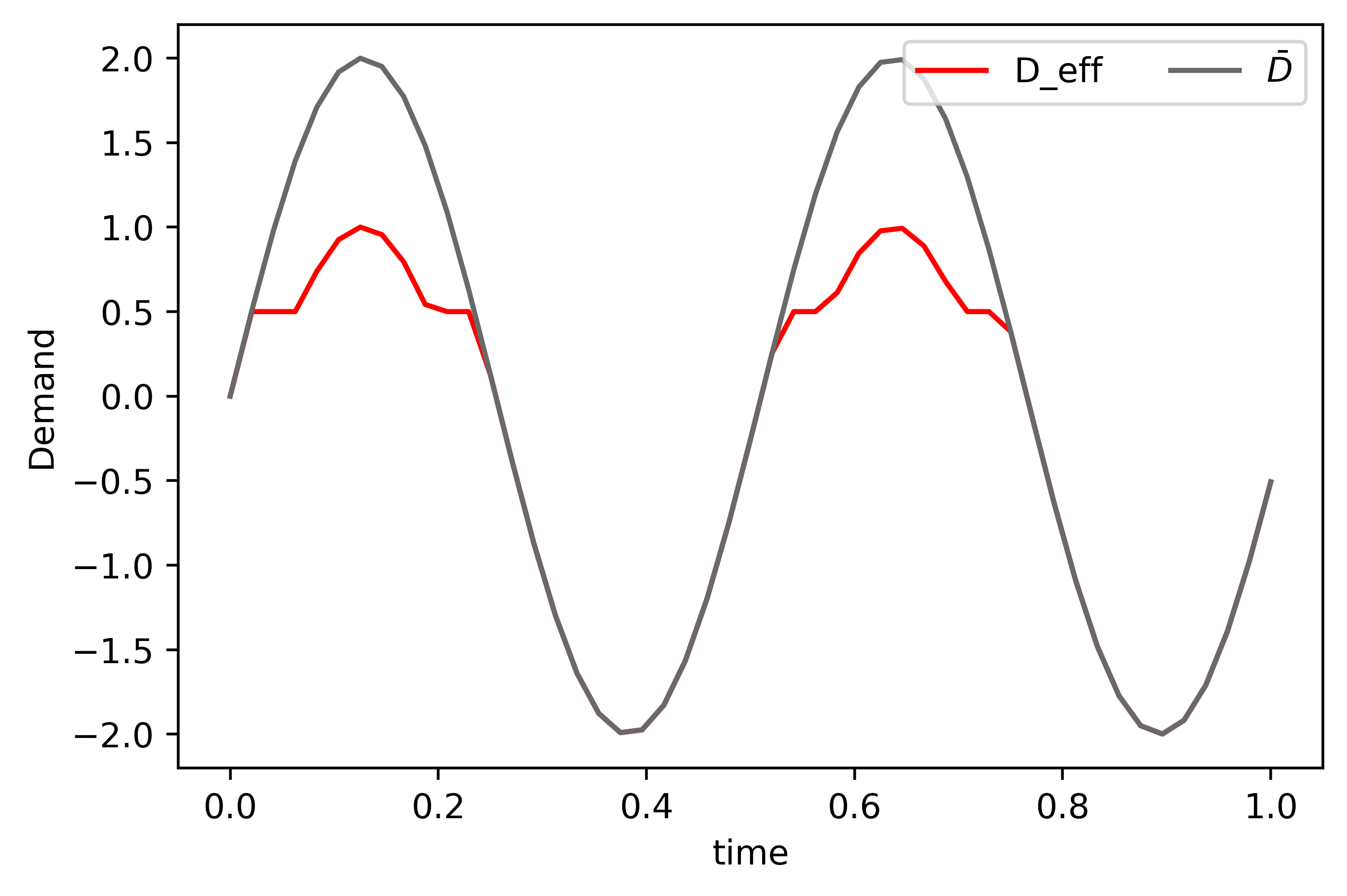}
\caption{Exogenous quantity $\bar{D}$ and effective demand $D_{\text{eff}}$}
\end{subfigure}
\begin{subfigure}[c]{0.5\textwidth}
\centering
\includegraphics[width=5cm,height=4cm]{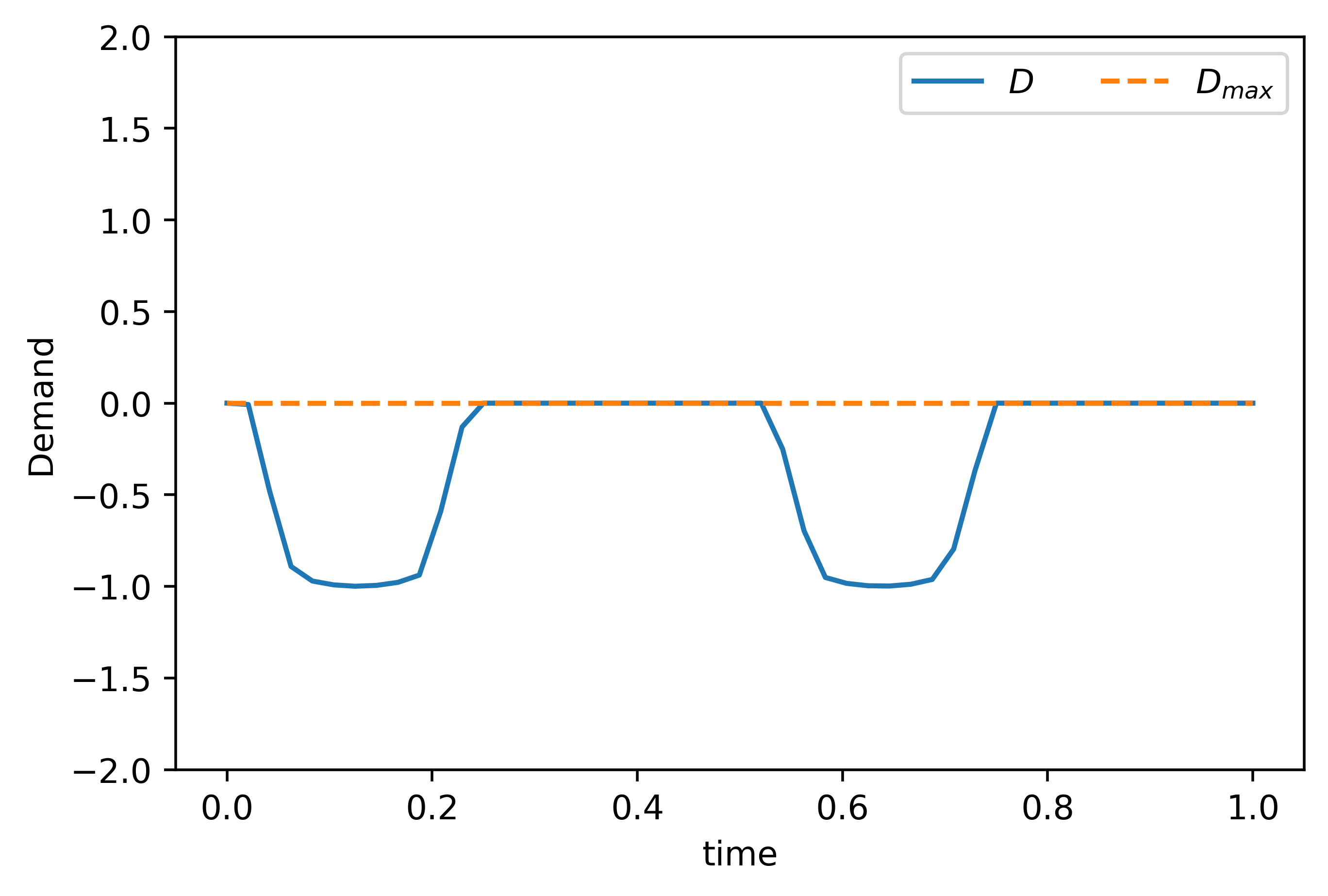}
\caption{Maximal aggregated demand $D_{\text{max}}$ and demand $D$}
\end{subfigure} \\
\begin{subfigure}[c]{0.5\textwidth}
\centering
\includegraphics[width=5cm,height=4cm]{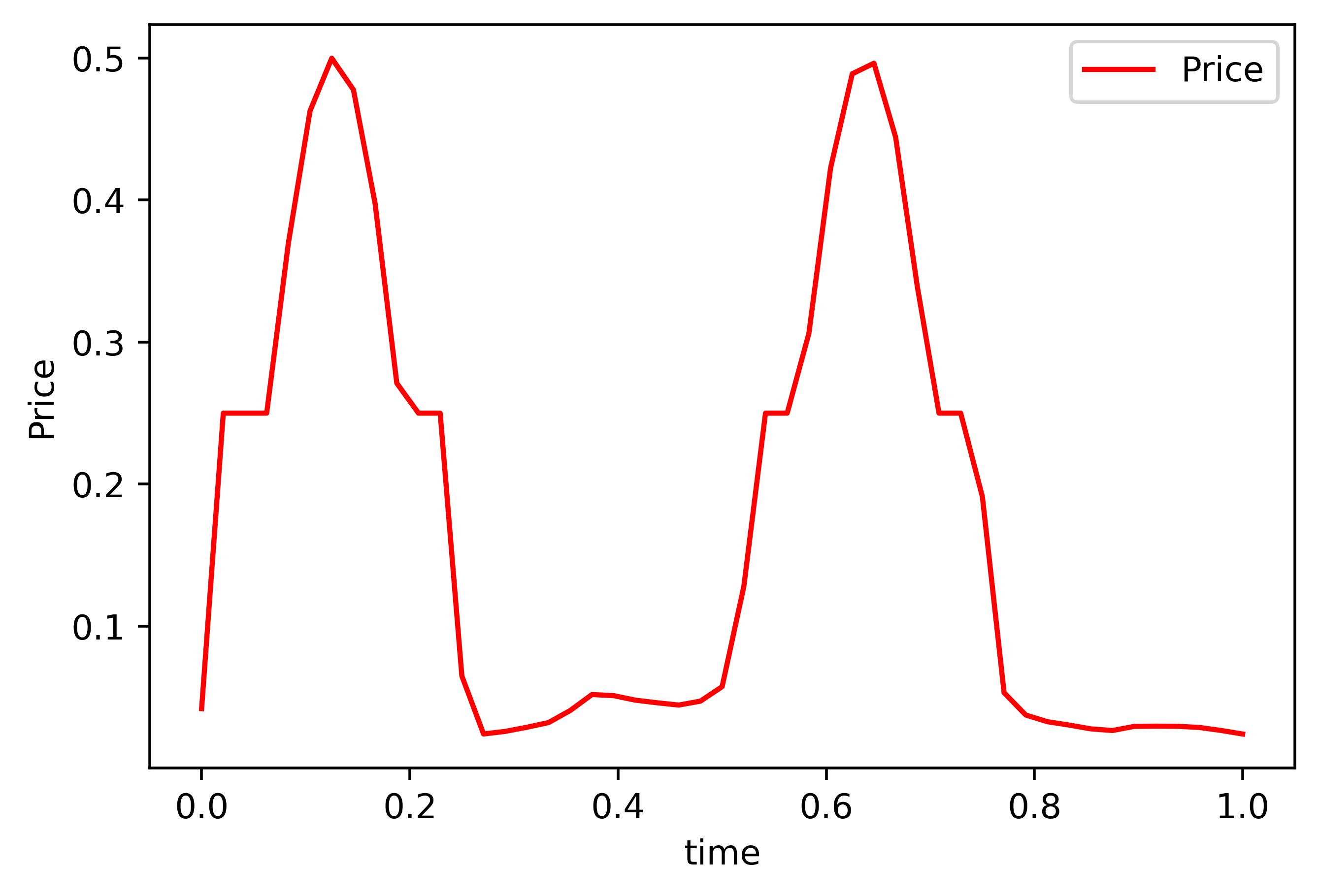}
\caption{Price $P$}
\end{subfigure}
\end{minipage}
}
\end{figure}

In this situation each agent faces a price and chooses to increase or deplete her stock. The price mechanism is given by
\begin{equation} \nonumber
P(t) \in \partial \bm{\phi}[D](t) = \nabla \bm{\phi}_1[D] + \partial \bm{\phi}_2[D]  = \frac{1}{2}D_{\text{eff}}(t) + N_{(-\infty,D_{\text{max}}]}(D(t))
\end{equation}
where $D_{\text{eff}} := D + \bar{D}$ is called the effective demand and follows two regimes.

\begin{figure}[htb]
\fbox{
\begin{minipage}[c]{0.95\linewidth}
\begin{subfigure}[c]{0.5\textwidth}
\centering
\includegraphics[width=5cm,height=4cm]{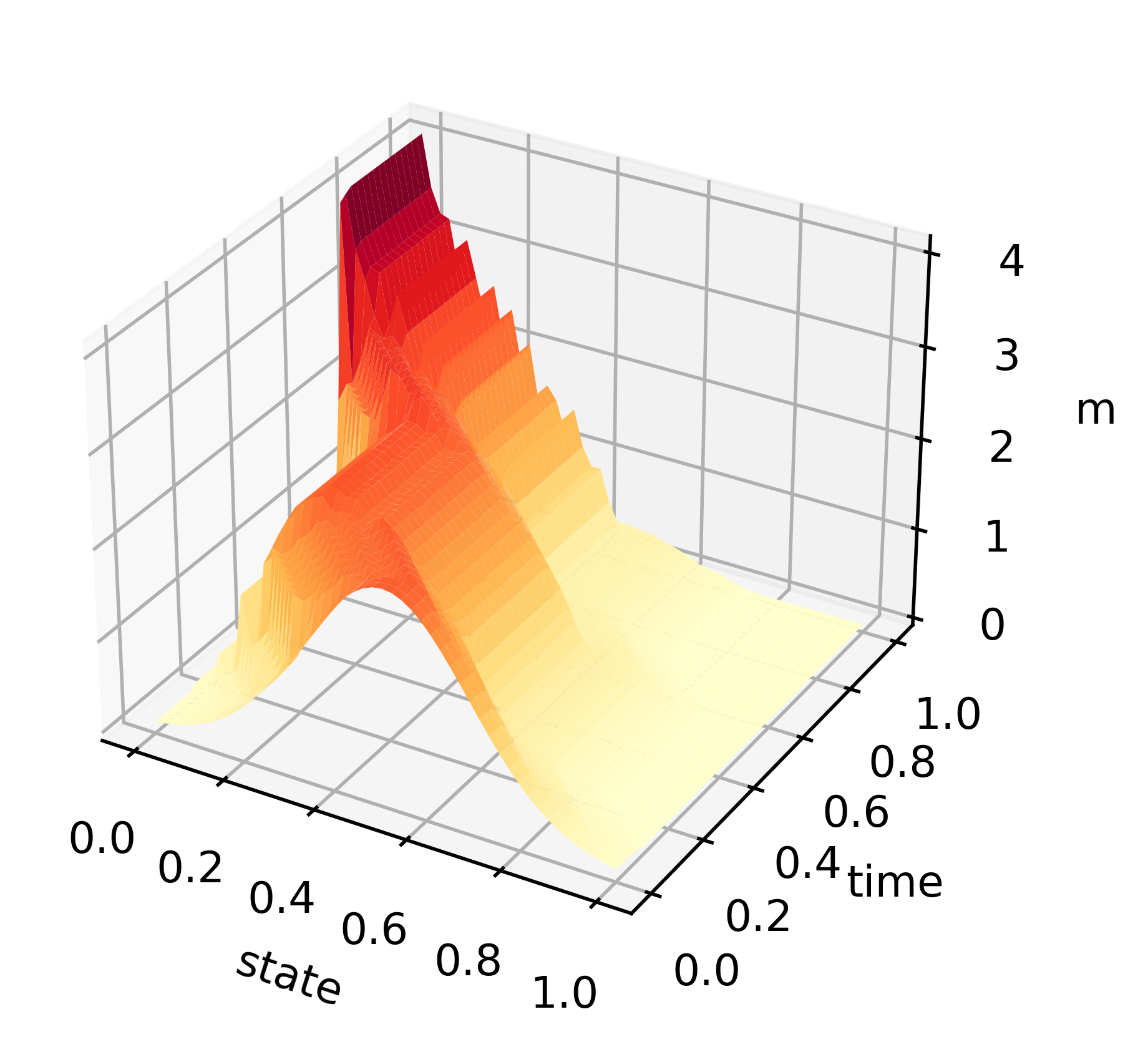}
\caption{Measure $m$, 3d plot}
\end{subfigure}
\begin{subfigure}[c]{0.5\textwidth}
\centering
\includegraphics[width=5cm,height=4cm]{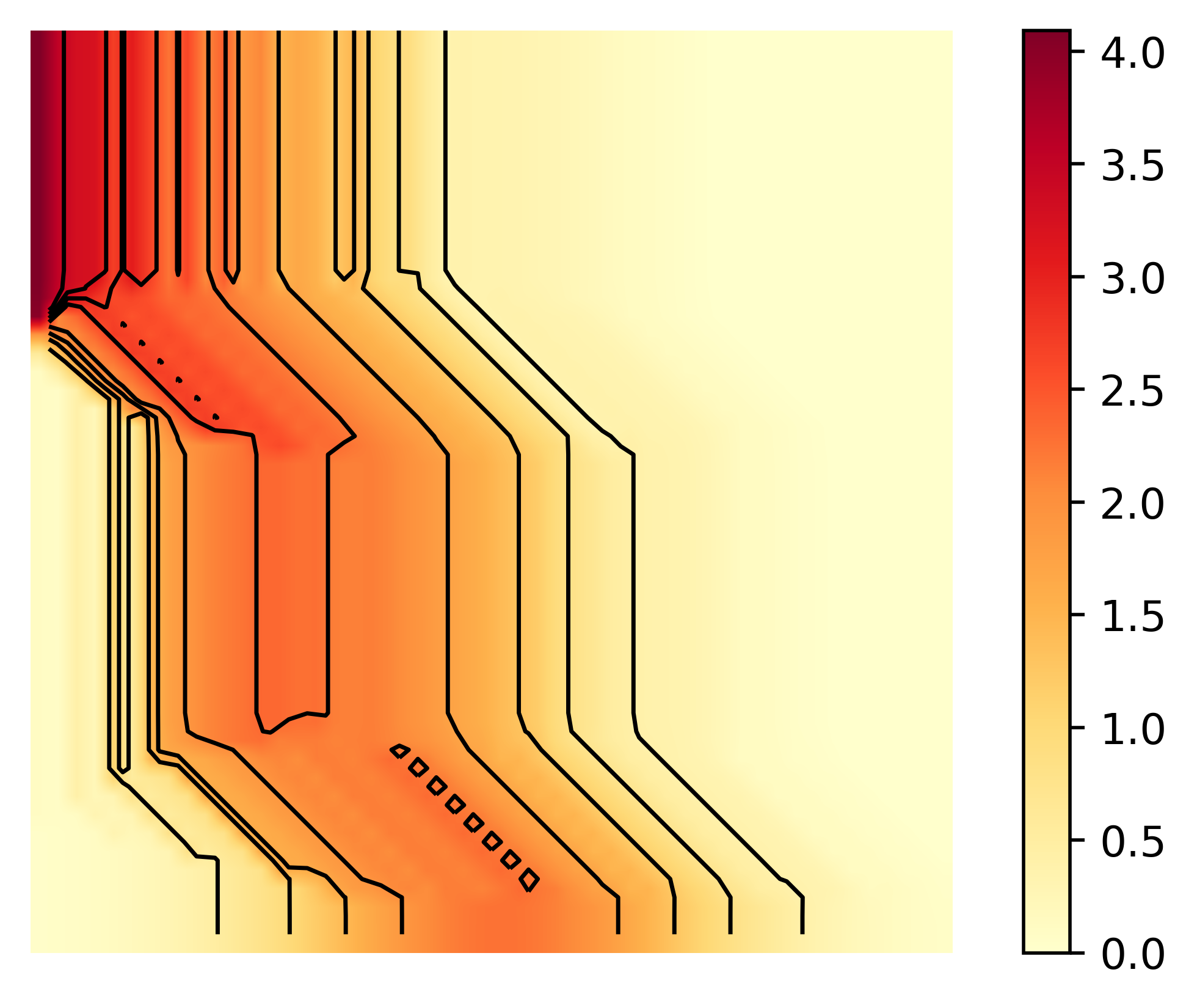}
\caption{Measure $m$, contour plot}
\end{subfigure} \\
\begin{subfigure}[c]{0.5\textwidth}
\centering
\includegraphics[width=5cm,height=4cm]{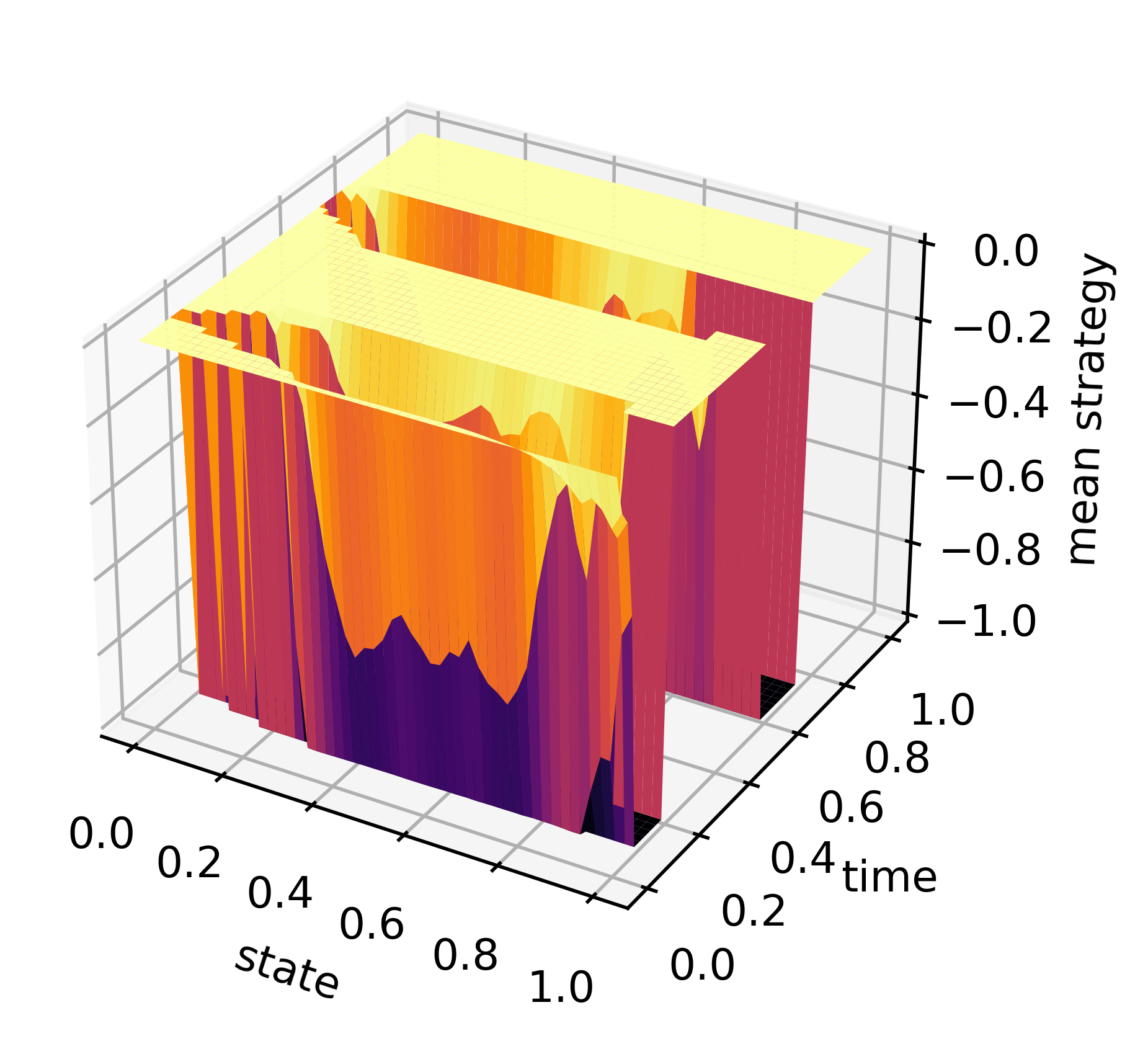}
\caption{Mean displacement $v$, 3d plot}
\end{subfigure}
\begin{subfigure}[c]{0.5\textwidth}
\centering
\includegraphics[width=5cm,height=4cm]{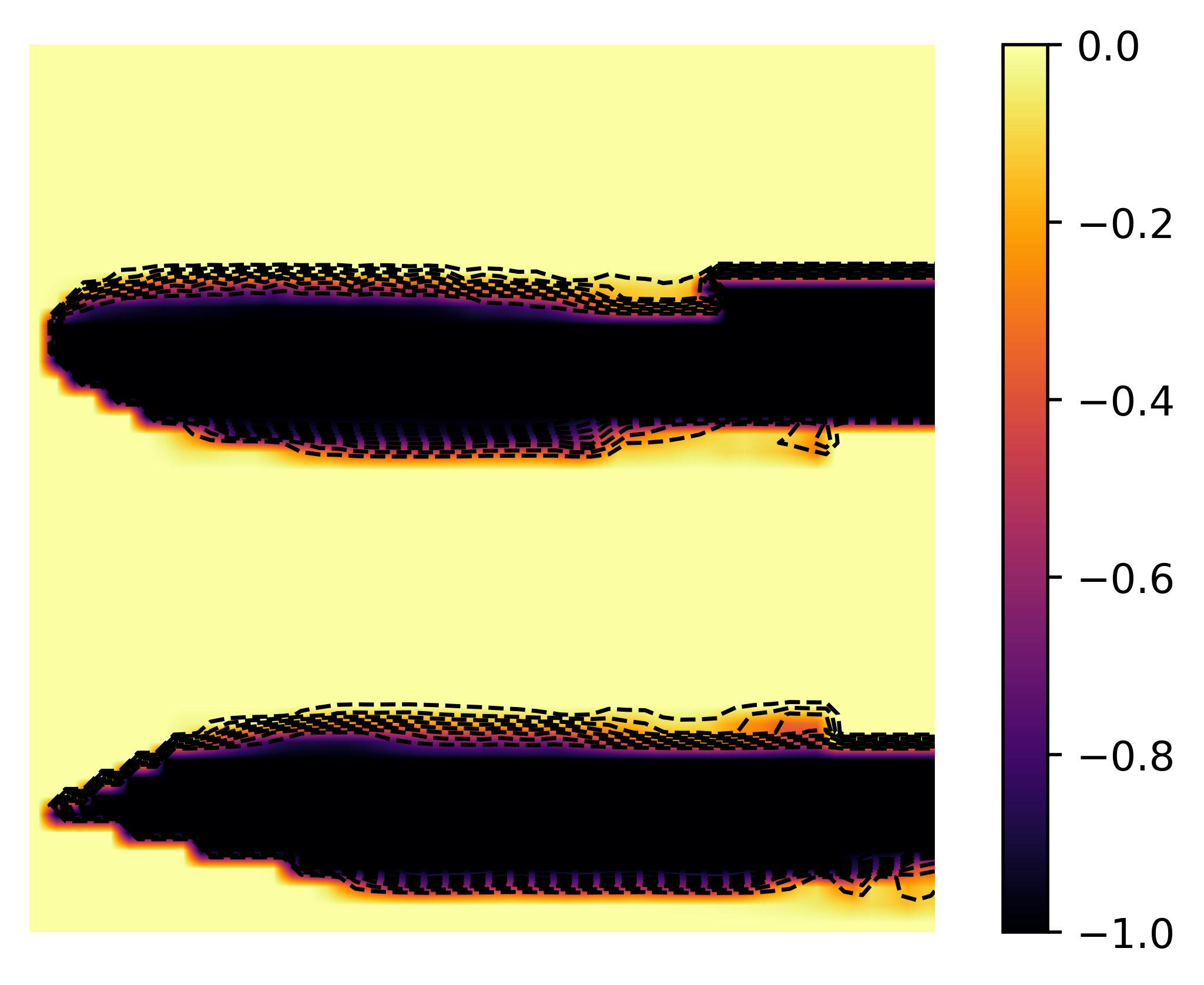}
\caption{Mean displacement $v$, contour plot}
\end{subfigure} \\
\begin{subfigure}[c]{0.5\textwidth}
\centering
\includegraphics[width=5cm,height=4cm]{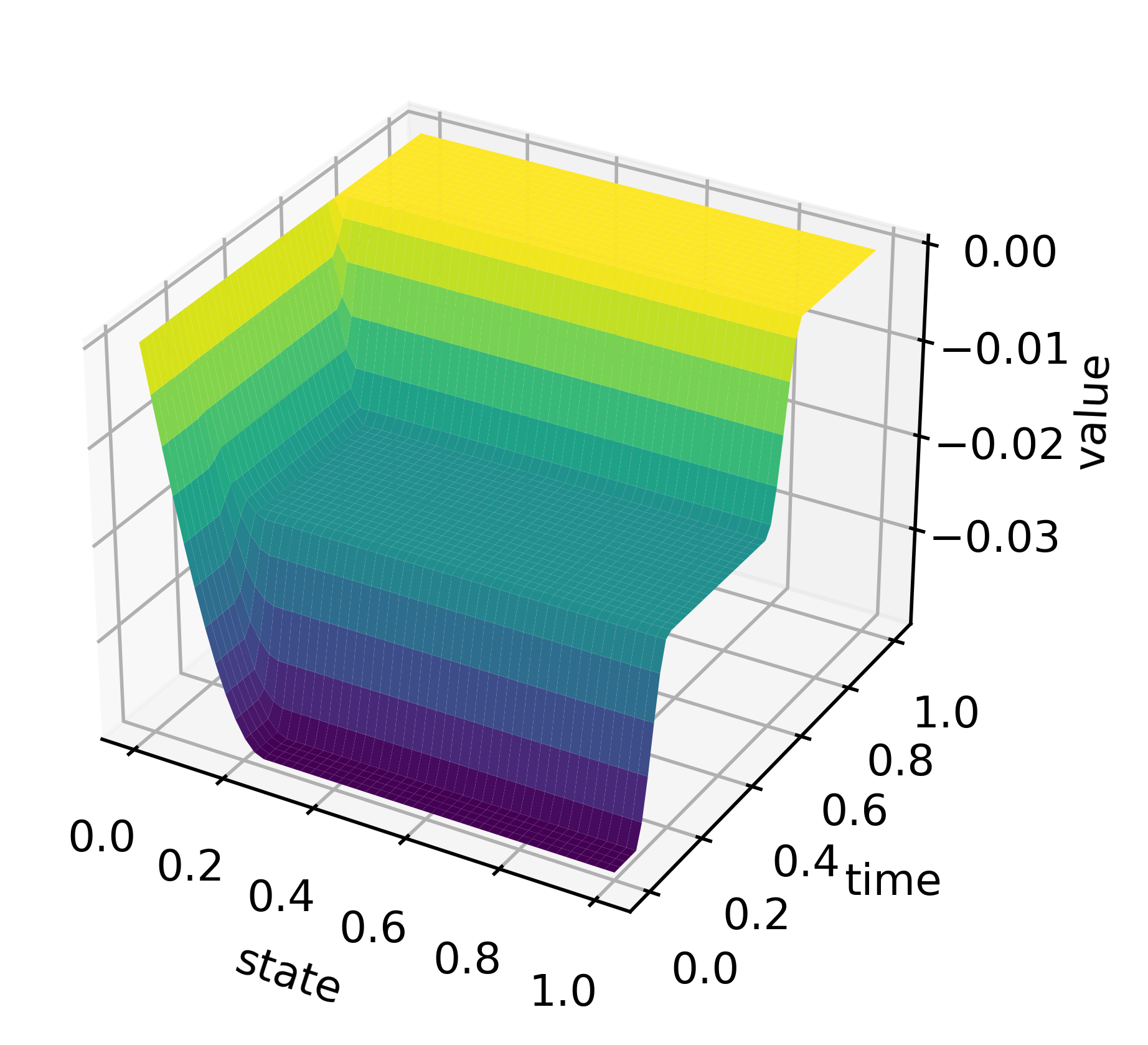}
\caption{Value function $u$, 3d plot}
\end{subfigure}
\begin{subfigure}[c]{0.5\textwidth}
\centering
\includegraphics[width=5cm,height=4cm]{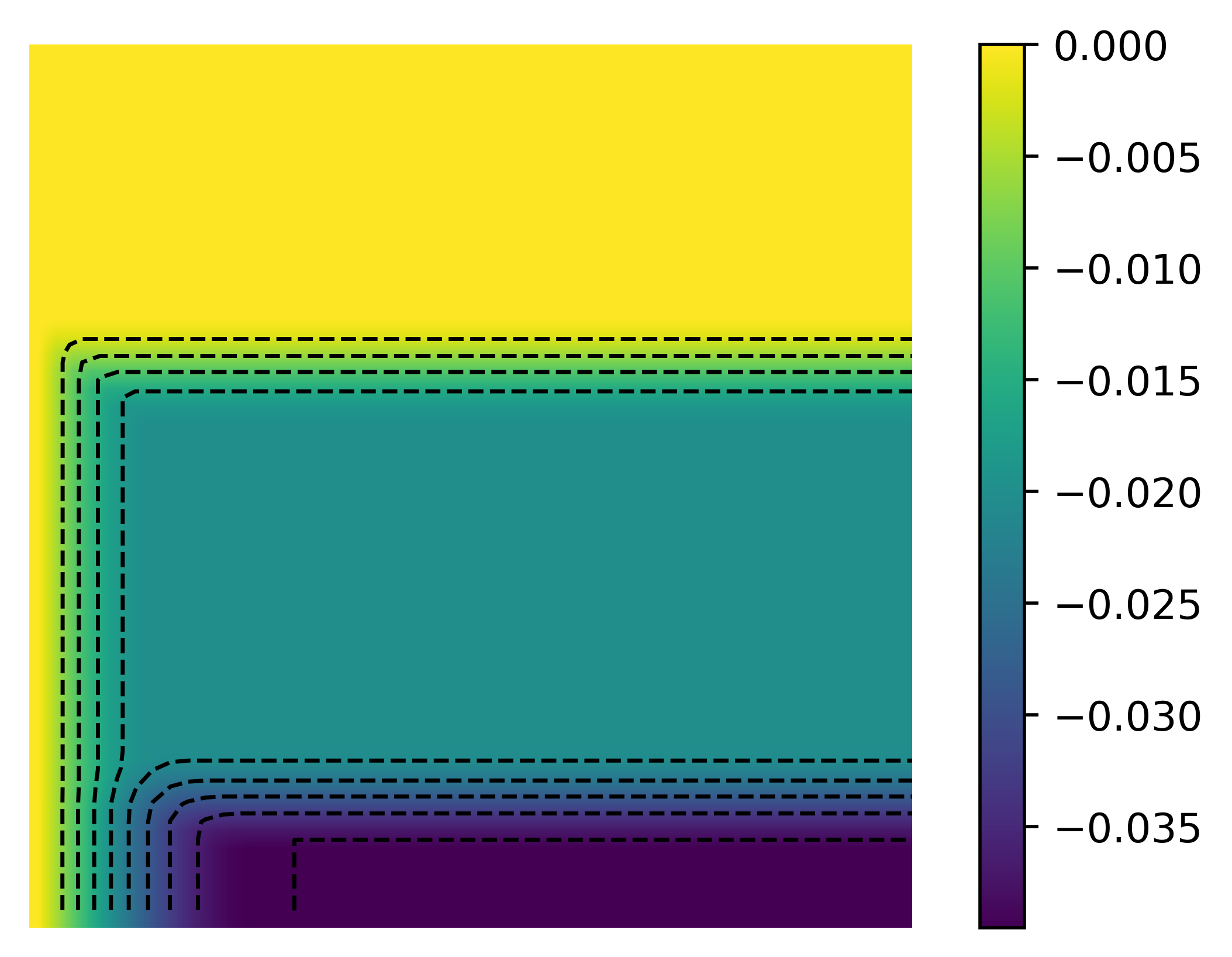}
\caption{Value function $u$, contour plot}
\end{subfigure} 
\caption{Solution of Example 2}
\label{fig:sol_ex2}
\end{minipage}
}
\end{figure}

When the constraint on the demand $D$ is not binding, we are in a soft regime, the price plays the role of a classical price term and is given by $P(t) = \frac{1}{2}D_{\text{eff}}(t)$. The quantity $\bar{D}$ is an exogenous quantity which can be positive or negative. If the quantity $\bar{D}(t) > 0$, the exogenous quantity is interpreted as being a demand and the agents have an incentive to deplete their stock to satisfy this demand. If $\bar{D}(t) < 0$, the exogenous quantity is  interpreted as being a supply. In the absence of a hard constraint, the agents would have interest to increase their stock to absorb this supply.
When the constraint on the demand is binding, we are in a hard regime and the price plays the role of an adjustment variable so that the constraint $D(t) \leq D_{max}$ is satisfied and is maximal for the dual problem.

In the case where it is not profitable to buy or sell, we have that $D(t) = 0$ and thus $D_{\text{eff}}(t) = \bar{D}(t)$. This situation occurs when the quantity $\bar{D}(t) \leq 0$, since the hard constraint prevents the agents from buying on the market. On the graph this corresponds to the case where the red and the black curves coincide.

\begin{figure}[htb]
\fbox{
\begin{minipage}[c]{0.95\linewidth}
\begin{center}
\textbf{Convergence results}
\end{center}
\vspace{2em}
\begin{subfigure}[c]{0.5\textwidth}
\centering
\includegraphics[width=5cm,height=4cm]{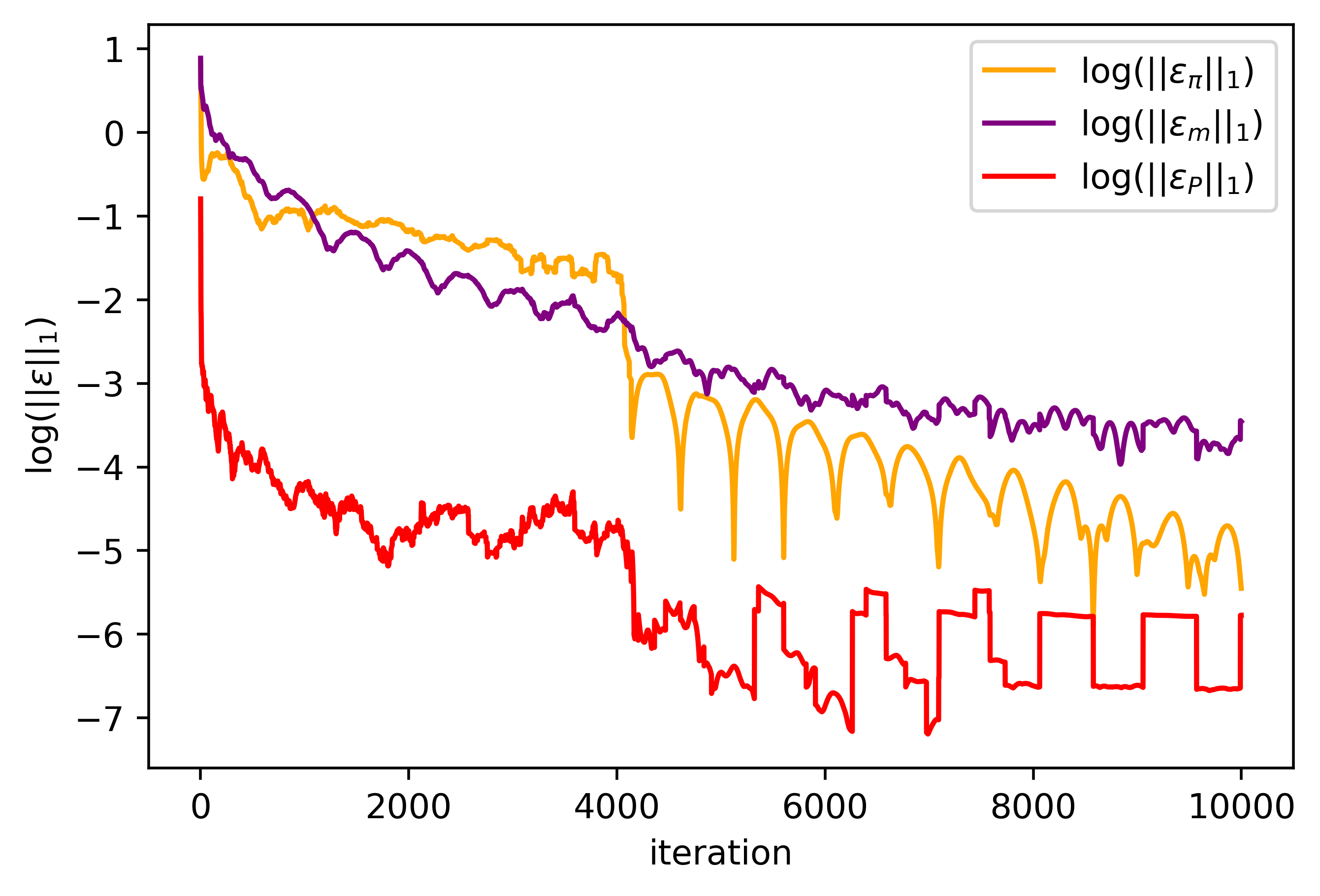}
\caption{ADMM}
\end{subfigure}
\begin{subfigure}[c]{0.5\textwidth}
\centering
 \includegraphics[width=5cm,height=4cm]{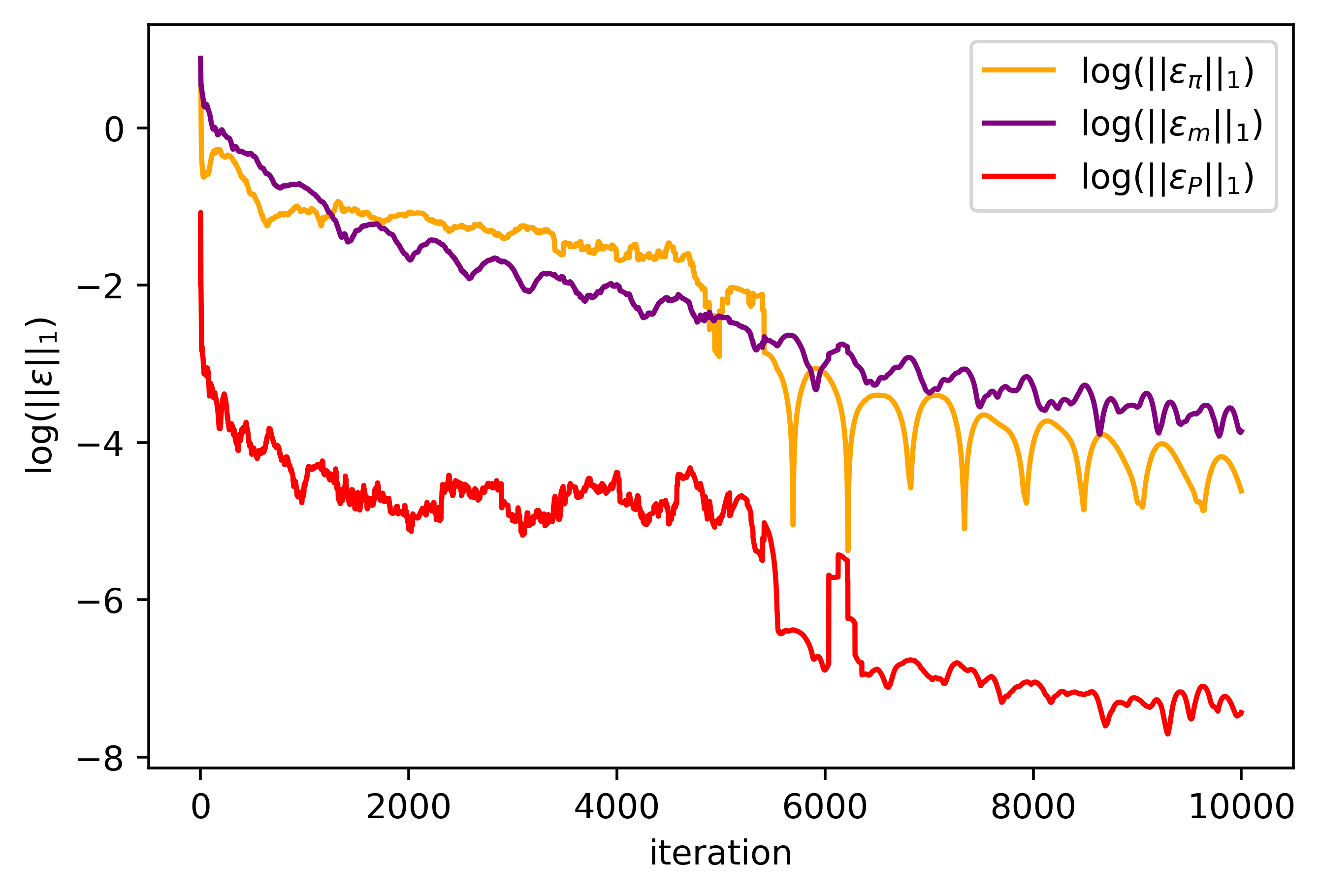} 
 \caption{ADMG}
 \end{subfigure}  \\
 \begin{subfigure}[c]{0.5\textwidth}
 \centering
\includegraphics[width=5cm,height=4cm]{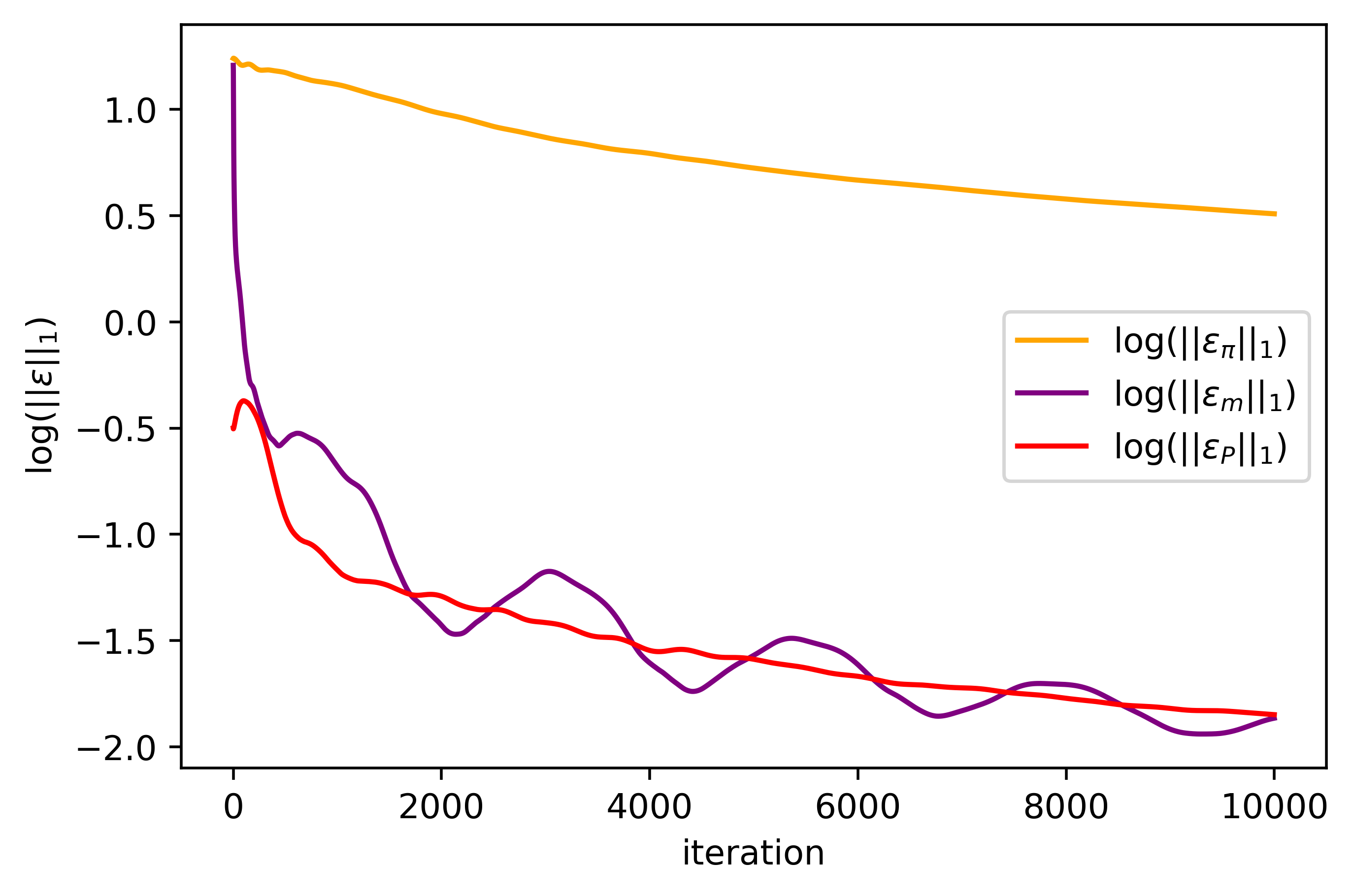} 
\caption{Chambolle-Pock}
\end{subfigure}
 \begin{subfigure}[c]{0.5\textwidth}
 \centering
\includegraphics[width=5cm,height=4cm]{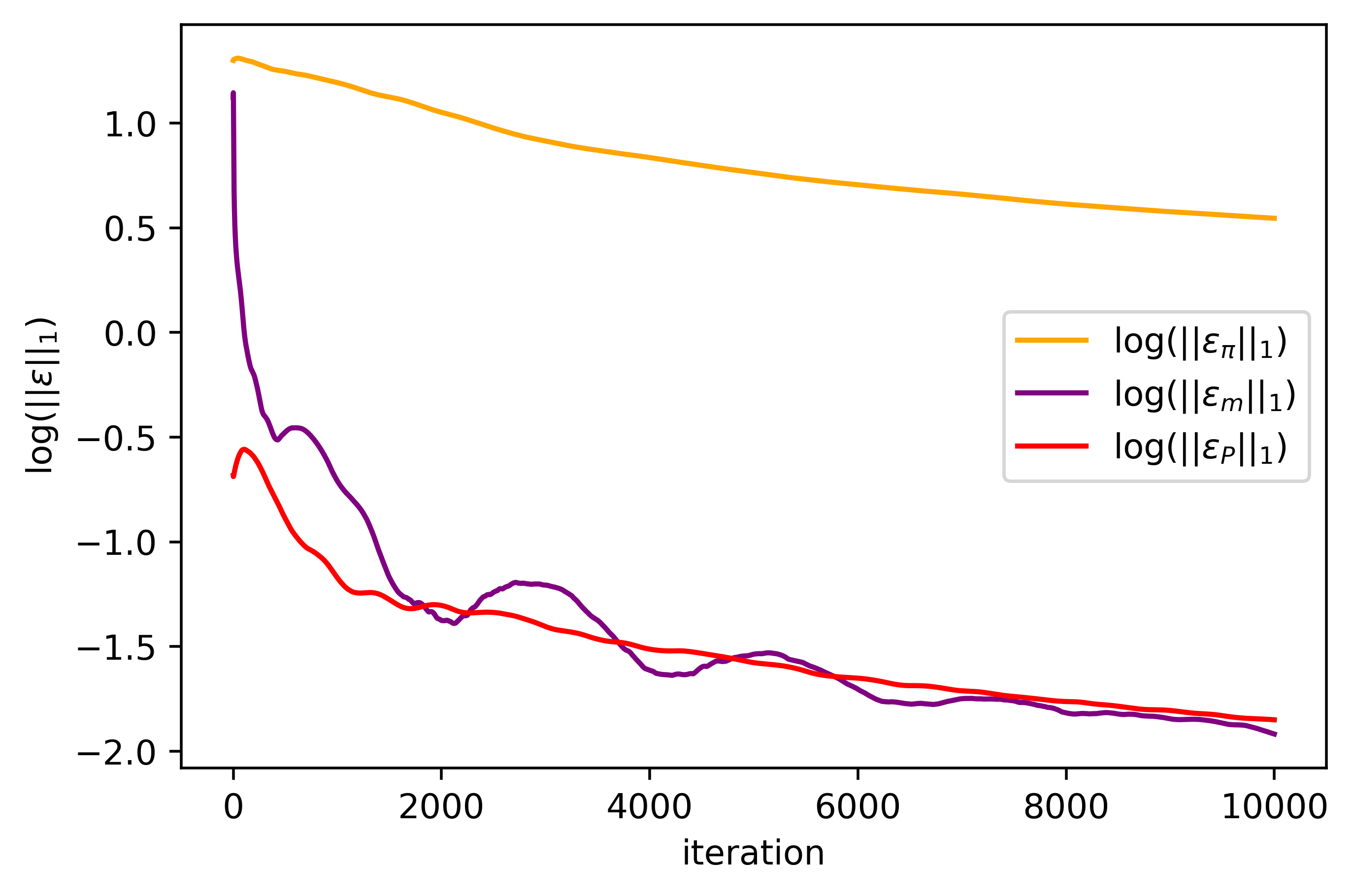}
\caption{Chambolle-Pock-Bregman} 
\end{subfigure}
\caption{Errors plots for Example 2}
\label{fig:error_plots_2}

\vspace{1em}

 The execution time of each algorithm is given in the following table.

\vspace{1em}

\centering
\begin{tabular}{|c|c|c|c|c|}
\hline
& Chambolle-Pock & Chambolle-Pock-Bregman & ADMM & ADM-G \\ \hline
Time (s) & 1700 & 1300 & 2200 & 2200 \\ \hline
\end{tabular}
\caption{Execution time of each algorithm for Example 2 and $N= 10000$}
\end{minipage}
}
\end{figure}

When $\bar{D}(t) \geq 0$ we observe that the red curve is lower than the black curve meaning that a certain amount (given by the blue curve on the following graph) of the demand has been satisfied by the agents. Three effects prevent the agents from fully satisfying the demand: their level of stock, their trading cost and their depletion speed limitation.

At the optimum we observe that the demand $D$ is indeed below the threshold $D_{\text{max}}$, meaning that the constraint is satisfied. 
We now comment the measure $m$, the mean displacement $v$ and the value function $u$. For a given initial distribution of the level of stock, we observe that the measure $m$ is shifted to the left with time. This means that the agents deplete their stocks with time. This is consistent with the mean displacement $v$ where we observe two regimes: either the agents choose to sell as much as possible or the agents choose not to sell on average. The value function $u$ can be interpreted backward. At the end of the game the value is null due to the terminal condition. Then the higher the level of stock, the lower the value function that is to say the value function is increasing in time and decreasing in space. This comes from the definition of $\alpha$ and the constraint $D \leq 0$, which implies that the price is positive.


\appendix

\section{Appendix \label{Appendix}}

We detail here the calculation of the projection on $Q$ and the non-linear proximity operator in \eqref{eq:m1hat-what}, for a running cost of the form
\begin{equation} \nonumber
\ell(t,x,\rho) = \sum_{y\in S} \rho(y) \beta(t,x,y) + \chi_{\Delta(S)}(\rho).
\end{equation}
The adaptation to the case where $\ell$ is defined by \eqref{eq:num_ell} is straightforward.

\subsection{Projection on $Q$ \label{sec:projection}}

We detail the computation of $\proj_Q$, as it appears in \eqref{eq:proj-details} and \eqref{eq:projection}. 
First notice that the projection is decoupled in space and time, then for any  $(t,x) \in \mathcal{T} \times S$ and $(\bar{a},\bar{b}) \in \mathbb{R} \times  \mathbb{R}(S)$, we need to compute
\begin{equation}  \nonumber
\proj_{Q_{t,x}}(\bar{a},\bar{b}) = \argmin_{(a,b) \in Q_{t,x}} \ (a-\bar{a})^{2}/2 + \sum_{y\in S} (b(y)-\bar{b}(y))^2/2,
\end{equation}
where 
$Q_{t,x} = \left\{ (a,b) \in \mathbb{R} \times \mathbb{R}(S), \, a + b(y) - \beta(y) \leq 0 \right\}$.
The corresponding problem is
\begin{equation} \label{eq:proj}
\min_{a \in \mathbb{R}} \, \Bigg( (a-\bar{a})^{2}/2 + \min_{\begin{subarray}{c} b \in \mathbb{R}(S) \\ b(y) \leq \beta(y) - a, \;  \forall y\in S \end{subarray}} \ \,   \sum_{y\in S} (b(y)-\bar{b}(y))^2/2 \Bigg).
\end{equation}
For any $a\in \mathbb{R}$, the solution of the inner minimization problem
is given by
\begin{equation}  \nonumber
b^\star(a,y) := \min\{\bar{b}(y),\beta(y) - a\}, \quad \forall y \in S.
\end{equation}
Then replacing into \eqref{eq:proj}, the minimization problem is now given by
\begin{equation}  \nonumber
\min_{a \in \R} g(a), \quad g(a) := (a-\bar{a})^{2}/2 + \sum_{y\in S} \max(0,a-\tilde{\beta}(y))^2/2,
\end{equation}
where $\tilde{\beta}(y) := \beta(y)-\bar{b}(y)$.
It is now relatively easy to minimize $g$.
Let us sort the sequence $(\tilde{\beta}(y))_{y \in S}$, that is, let us consider $(y_i)_{i\in \{0,\ldots,n-1\}}$ such that
$\tilde{\beta}(y_0) \leq \cdots \leq \tilde{\beta}(y_{n-1})$.
It is obvious that the function $g$ is strictly convex and polynomial of degree 2 on each of the intervals $(-\infty, \tilde{\beta}(y_0))$, $(\tilde{\beta}(y_0),\tilde{\beta}(y_1))$,..., and $(\tilde{\beta}(y_{n-1}),+\infty)$. One can identify on which of these intervals
a stationary point of $g$ exists, by evaluating $\partial g(\tilde{\beta}(y_i)$, for all $i=0,...,n-1$. Then one can obtain an analytic expresison of the (unique) stationary point $a^\star$, which minimizes $g$. Finally, we have
%
$\proj_{Q_{t,x}}(\bar{a},\bar{b})= (a^\star,b^\star(a^\star, \cdot))$.

\subsection{Entropic proximity operator\label{computation:entropic}}

Here we detail the computation of the solution to \eqref{eq:m1hat-what}.
For notational purpose we set
$c_1 = \tau(- u' + \gamma')$
and
$c_2 =\tau(\beta + \bm{A}^\star P' + \bm{S}^\star u')$.
By definition of the running cost $\ell$, we have that
\begin{equation} \nonumber
\sum_{(t,x) \in \mathcal{T} \times S} \tilde{\bm{\ell}}[m_1,w](t,x) = \langle w, \beta \rangle + \chi_{\bfdom(\tilde{\bm{\ell}})}(m_1,w).
\end{equation}
Problem \eqref{eq:m1hat-what} writes
\begin{align*}
\min_{(m_1,w) \in \mathcal{R}}  & \langle m_1, c_1 \rangle + \langle w ,c_2 \rangle  + \frac{1}{\tau} d_{KL}((m_1,w),(m_1',w')) \\\
\text{subject to: } &
\begin{cases}
\begin{array}{l}
m_1(t,x) \leq 1 \\
m_1(t,x) - \sum_{y \in S} w(t,x,y)= 0.
\end{array}
\end{cases}
\end{align*}
%
%
To find the solution, we define the following Lagrangian with associated multipliers $(\lambda_1,\lambda_2) \in \mathbb{R}(\mathcal{T} \times S) \times \mathbb{R}_{+}(\bar{\mathcal{T}} \times S)$:
\begin{align} \nonumber
& \mathcal{L}(m_1,w,\lambda_1,\lambda_2) =  \langle m_1, c_1 \rangle + \langle w ,c_2 \rangle + d_{KL}((m_1,w),(m_1',w')) \\
& \quad + \sum_{(t,x) \in \mathcal{T} \times S} \lambda_1(t,x)\Big( m_1(t,x) - \sum_{y\in S} w(t,x,y) \Big)+ \sum_{(s,x) \in \bar{\mathcal{T}} \times S} \lambda_2(s,x)( m_1(s,x) - 1). \nonumber
\end{align}
 For any $(t,s,x,y) \in \mathcal{T} \times \bar{\mathcal{T}} \times S \times S$, a saddle point of the Lagrangian is given by the following first order conditions,
\begin{equation} \label{eq:system-proj-entropic}
\begin{cases}
\hat{m}_1(T,x) &= m_1'(T,x) \exp (- \lambda_2(T,x) - c_1(T,x)),\\
\hat{m}_1(t,x) &= m_1'(t,x) \exp (-\lambda_1(t,x) - \lambda_2(t,x)-c_1(t,x)),\\
\hat{w}(t,x,y) &= w'(t,x,y)  \exp (\lambda_1(t,x) - c_2(t,x,y)), \\
\hat{m}_1(t,x) &= \sum_{y' \in S}\hat{w}(t,x,y'),\\
0 & = \min\left\{\lambda_2(s,x), \hat{m}_1(s,x) - 1\right\} .
\end{cases} 
\end{equation}

\textbf{Case 1:} $\lambda_2(s,x) > 0$. At time $s=T$ we have that $\hat{m}_1(s,x) = 1$. For any $s<T$ we have that $\hat{m}_1(s,x) = 1$ and $\sum_{y \in S}\hat{w}(s,x,y) = 1$ and by a direct computation we have that
\begin{equation}  \label{system:lambageq0}
\begin{cases}
\hat{m}_1(s,x) & = 1,\\
\hat{w}(s,x,y) & =  w'(s,x,y) \exp (-c_2(s,x,y)) C(s,x), \\
\lambda_1(s,x) & =  \ln \left(C(s,x) \right), \\
\lambda_2(s,x) & = \ln \left(m_1'(s,x)/C(s,x)) \right) - c_1(s,x),
\end{cases} 
\end{equation}
where
$C(s,x) = \sum_{y \in S } w'(s,x,y) \exp(-c_2(s,x,y))$.

\textbf{Case 2:} $\lambda_2(s,x) = 0$. At time $s=T$ we have that $\hat{m}_1(s,x) = m_1'(s,x) \exp (-c_1(s,x))$. For any $s<T$ we have by a direct computation
\begin{equation} \label{system:lamba0}
\begin{cases}
\hat{m}_1(s,x) &= m_1'(s,x) C(s,x)^{-1} \exp (-c_1(s,x)),\\
\hat{w}(s,x,y) &= w'(s,x,y) C(s,x) \exp (- c_2(s,x,y)), \\
\lambda_1(s,x)& =  \ln \left(C(s,x) \right),\\
\lambda_2(s,x) &= 0,
\end{cases} 
\end{equation}
where
$C(s,x) = \Big(m_1'(s,x) \exp (-c_1(s,x)) /  \sum_{y \in S } w'(s,x,y)   \exp(- c_2(s,x,y)) \Big)^{1/2}.
$

In order to identify which of the two cases arises, one can compute a solution with formula \eqref{system:lambageq0} and check a posteriori that $\lambda_2(s,x) > 0$. If this is not the case, we deduce that the solution to \eqref{eq:system-proj-entropic} is given by \eqref{system:lamba0}.


\bibliographystyle{plain}
\bibliography{biblio}

\end{document}